\documentclass[a4paper,11pt]{amsart}
\usepackage{amsmath,amsthm,amssymb}
\usepackage[bookmarks=false]{hyperref}

\numberwithin{equation}{section}

{\theoremstyle{definition}\newtheorem{definition}{Definition}[section]

\newtheorem{remark}[definition]{Remark}
\newtheorem{example}[definition]{Example}
\newtheorem{condition}[definition]{Condition}}

\newtheorem{proposition}[definition]{Proposition}
\newtheorem{lemma}[definition]{Lemma}
\newtheorem{theorem}[definition]{Theorem}
\newtheorem{corollary}[definition]{Corollary}
\newtheorem{step}{Step}[section]

\newenvironment{itemlist}{\begin{list}{$\bullet$}{\setlength{\labelwidth}{2ex}\setlength{\leftmargin}{5ex}\setlength{\labelsep}{1.5ex}\setlength{\itemindent}{0ex}\setlength{\rightmargin}{0ex}\setlength{\topsep}{0.5ex}\setlength{\parsep}{0.8ex}\setlength{\itemsep}{0.2ex}\setlength{\listparindent}{0ex}}}{\end{list}}

\newcounter{mycount}
\newenvironment{enumlist}{\begin{list}{\arabic{mycount}.}{\usecounter{mycount}\setlength{\labelwidth}{2ex}\setlength{\leftmargin}{5ex}\setlength{\labelsep}{1.5ex}\setlength{\itemindent}{0ex}\setlength{\rightmargin}{0ex}\setlength{\topsep}{0.5ex}\setlength{\parsep}{0.8ex}\setlength{\itemsep}{0.2ex}\setlength{\listparindent}{0ex}}}{\end{list}}

\newcommand{\M}{\operatorname{M}}
\newcommand{\C}{\mathbb{C}}

\newcommand{\embed}{\prec}
\newcommand{\F}{\mathbb{F}}

\newcommand{\actson}{\curvearrowright}
\newcommand{\SL}{\operatorname{SL}}
\newcommand{\rL}{\mathord{\text{\rm L}}}
\newcommand{\Aut}{\operatorname{Aut}}

\newcommand{\N}{\mathbb{N}}
\newcommand{\T}{\mathbb{T}}
\newcommand{\Z}{\mathbb{Z}}
\newcommand{\cF}{\mathcal{F}}
\newcommand{\cA}{\mathcal{A}}
\newcommand{\cV}{\mathcal{V}}
\newcommand{\id}{\mathord{\operatorname{id}}}
\newcommand{\si}{\sigma}

\newcommand{\recht}{\rightarrow}
\newcommand{\cU}{\mathcal{U}}
\newcommand{\vphi}{\varphi}

\newcommand{\R}{\mathbb{R}}
\newcommand{\al}{\alpha}
\newcommand{\eps}{\varepsilon}

\newcommand{\Tr}{\operatorname{Tr}}
\newcommand{\ovt}{\mathbin{\overline{\otimes}}}
\newcommand{\B}{\operatorname{B}}
\newcommand{\om}{\omega}
\newcommand{\cP}{\mathcal{P}}
\newcommand{\cZ}{\mathcal{Z}}

\newcommand{\cK}{\mathcal{K}}

\newcommand{\cH}{\mathcal{H}}

\newcommand{\ot}{\otimes}

\newcommand{\dis}{\displaystyle}
\newcommand{\Ad}{\operatorname{Ad}}
\newcommand{\cG}{\mathcal{G}}
\newcommand{\cM}{\mathcal{M}}
\newcommand{\dpr}{^{\prime\prime}}
\newcommand{\be}{\beta}

\newcommand{\lspan}{\operatorname{span}}
\newcommand{\Cal}{\mathcal}
\newcommand{\Mtil}{\widetilde{M}}

\newcommand{\D}{\operatorname{D}}
\newcommand{\Stab}{\operatorname{Stab}}

\newcommand{\PSL}{\operatorname{PSL}}

\newcommand{\Cstarred}{\text{\rm C}^*_{\text{\rm r}}}
\newcommand{\Sp}{\operatorname{Sp}}
\newcommand{\Dtil}{\widetilde{D}}
\newcommand{\ptil}{\widetilde{p}}
\newcommand{\Ptil}{\widetilde{P}}

\newcommand{\bim}[3]{\mathord{\raisebox{-0.4ex}[0ex][0ex]{\scriptsize $#1$}{#2}\hspace{-0.2ex}\raisebox{-0.4ex}[0ex][0ex]{\scriptsize $#3$}}}

\newcommand{\Om}{\Omega}
\newcommand{\de}{\Delta}
\newcommand{\Centr}{\operatorname{Centr}}
\newcommand{\sgeq}{{\scriptscriptstyle\geqslant}}
\renewcommand{\geq}{\geqslant}
\renewcommand{\leq}{\leqslant}
\newcommand{\action}[1]{\overset{#1}{\curvearrowright}}
\newcommand{\cMtil}{\widetilde{\mathcal{M}}}
\newcommand{\Norm}{\operatorname{Norm}}
\newcommand{\QN}{\operatorname{QN}}
\newcommand{\Image}{\operatorname{Im}}
\newcommand{\weakcont}{\subset_{\text{\rm\tiny weak}}}

\newcommand{\cB}{\mathcal{B}}
\newcommand{\Fix}{\operatorname{Fix}}
\newcommand{\cI}{\mathcal{I}}
\newcommand{\otalg}{\otimes_{\text{\rm alg}}}
\newcommand{\op}{^\text{\rm op}}
\newcommand{\dist}{\operatorname{dist}_{\|\cdot\|_2}}
\newcommand{\qtil}{\tilde{q}}

\begin{document}

\title[A class of superrigid group von Neumann algebras]{A class of superrigid group von Neumann algebras}

\author[Adrian Ioana]{Adrian Ioana$^{(1)}$}
\thanks{\mbox{}$^{(1)}$ Supported by a Clay Research Fellowship}
\address{Mathematics Department; University of California at Los Angeles, CA 90095-1555 (United States).}
\email{aioana@ucsd.edu}

\author[Sorin Popa]{Sorin Popa$^{(2)}$}
\thanks{\mbox{}$^{(2)}$ Partially supported by NSF Grant DMS-0601082}
\address{Mathematics Department; University of California at Los Angeles, CA 90095-1555 (United States).}
\email{popa@math.ucla.edu}

\author[Stefaan Vaes]{Stefaan Vaes$^{(3)}$}
\thanks{\mbox{}$^{(3)}$ Partially
    supported by ERC Starting Grant VNALG-200749, Research
    Programme G.0639.11 of the Research Foundation --
    Flanders (FWO) and K.U.Leuven BOF research grant OT/08/032.}
\address{Department of Mathematics;
    K.U.Leuven; Celestijnenlaan 200B; B--3001 Leuven (Belgium).}
\email{stefaan.vaes@wis.kuleuven.be}



%\dedicatory{\rm To appear in {\it Annals of Mathematics}}

\begin{abstract}
We prove that for any group  $G$ in a fairly large class of
generalized wreath product groups, the associated von Neumann
algebra $\rL G$ completely ``remembers'' the group $G$. More
precisely, if $\rL G$ is isomorphic to the von Neumann algebra $\rL
\Lambda$ of an arbitrary countable group $\Lambda$, then $\Lambda$
must be isomorphic to $G$. This represents the first superrigidity
result pertaining to group von Neumann algebras.
\end{abstract}

\maketitle

\section{Introduction and statement of main results}

A countable discrete group $G$ gives rise to a variety of rings and
algebras, studied in several areas of mathematics, such as algebra,
finite group theory, geometric group theory, representation theory,
non-commutative geometry, C$^*$- and von Neumann operator algebras.
A common underlying theme is the investigation of how the
isomorphism class of the ring/algebra depends on the group $G$.

Thus, by letting the (complex) group algebra $\C G$ act on the
Hilbert space $\ell^2 G$ by (left) convolution and then taking
its closure in the operator norm, one obtains the {\it reduced
group} C$^*$-{\it algebra} $\Cstarred G$, an important object of
study in non-commutative geometry (e.g., related to the Novikov
conjecture, see \cite{Co94}). In turn, by taking the closure of $\C
G$ in the weak operator topology one obtains the {\it group von
Neumann algebra} $\rL G$, introduced and studied by Murray and von
Neumann in \cite{MvN36, MvN43}.

When passing from $\C G$ to $\rL G$, the memory of $G$ tends to fade
away. This is best seen in the torsion free abelian case, where $\C
G$ remembers $G$ completely (see e.g.\ \cite{Hi39}), while all $\rL
G$ are isomorphic (because $\rL G = \rL^\infty(\hat{G}) \cong
\rL^\infty([0,1])$). Conjecturally, if $G$ is an arbitrary torsion
free group, then the only unitary elements in $\C G$ are the
multiples of the canonical unitaries $(u_g)_{g \in G}$ (see
\cite{Hi39,Ka70}, where in fact the conjecture was checked for all
orderable groups). On the other hand, the weak closure $\overline{\C
G}^w=\rL G$ entirely wipes out this structure. The intermediate case $\Cstarred G$ appears to
be closer to $\C G$ than to $\rL G$. Indeed, if $G$ is abelian
torsion free, then the group of connected components of
$\cU(\Cstarred G)$ coincides with $G$ so that $\Cstarred G$
completely remembers $G$ (this is obvious when $G = \Z^n$ and passes
to inductive limits $\Z^{n_1} \hookrightarrow \Z^{n_2}
\hookrightarrow \cdots$). The non-commutative case is very poorly understood. It seems not even known whether $\Cstarred G$ always remembers a torsion free group $G$.
This question is particularly interesting for free groups, $G=\mathbb F_n$, where a
result in \cite{PV81} already shows that $\Cstarred \mathbb F_n$ are
non-isomorphic for different $n$'s. In fact, when combined with
results in \cite{DHR97,Ri87}, if follows that the group of
connected components of $\cU(\Cstarred \mathbb F_n)$ is isomorphic
to $\mathbb Z^n$.

For von Neumann algebras, the really interesting case is when $\rL
G$ has trivial center, i.e.\ when $\rL G$ is a II$_1$ factor,
corresponding to $G$ having infinite conjugacy classes (icc), see
\cite{MvN43}. Here again, like in the abelian case, a celebrated
result of Connes \cite{Co75} shows that all II$_1$ factors coming
from icc amenable groups are isomorphic to the {\it hyperfinite}
II$_1$ factor of Murray and von Neumann. While non-amenable groups
$G$ were known since \cite{MvN43,Sc63} to produce non-hyperfinite
factors $\rL G$ and an uncountable family of icc groups with the
associated II$_1$ factors non-isomorphic was constructed in
\cite{McD69}, very little is known of how $\rL G$ depends on the
group $G$, especially when $G$ is a ``classical'' group like
$\SL(n,\Z)$, or a free group $\F_n$. For instance, it is a famous
open problem whether the factors $\rL \F_n$, $n \geq 2$, are
non-isomorphic. In the same vein, a well known conjecture of Connes
\cite{Co80b} asks whether $\rL G \cong \rL \Lambda$ for icc property
(T) groups $G, \Lambda$ implies $G \cong \Lambda$. This conjecture
remains wide open, notably for $G = \SL(n,\Z)$, $n \geq 3$. Note
however that by \cite{CH88}, if $G,\Lambda$ are lattices in
$\Sp(n,1)$, respectively $\Sp(m,1)$, then $\rL G \cong \rL \Lambda$
implies $n=m$. Along these lines, several recent results in
deformation rigidity theory provide classes of groups $\Cal G$ for
which any isomorphism $\rL G \simeq \rL \Lambda$, with $G, \Lambda
\in \Cal G$, entails isomorphism of the groups $G \simeq \Lambda$
(see e.g. \cite{Po01b,Po04,IPP05,PV06}, etc). This is for instance
the case for the class $\Cal G$ of all wreath product groups
$\mathbb Z/2\mathbb Z \wr \Gamma$ with $\Gamma$ having property (T)
\cite{Po04}. At the opposite end, using \cite{Co75} and free
probability it has been shown that $\rL (\Gamma_1
* \Gamma_2 * \cdots * \Gamma_n) \simeq \rL \F_n$, for any infinite
amenable groups $\Gamma_i$ and $n\geq 2$, see \cite{Dy92b}. Other
unexpected isomorphisms between group factors can be found in
Section \ref{sec.counterex}.

In fact, more than just distinguishing between property (T) group
factors, a positive answer to Connes' rigidity conjecture implies
that the II$_1$ factor $\rL G$ of an icc property (T) group $G$
uniquely determines the group $G$. Indeed, by \cite{CJ83}, if $\rL G
\simeq \rL \Lambda$ and $G$ has property (T), then $\Lambda$
automatically has this property, showing that in Connes' conjecture
it is sufficient to assume property (T) only on the group $G$. This
gives its statement a W$^*$-superrigidity flavor, in the same spirit
as the recent superrigidity results for group measure space II$_1$
factors (\cite{PV09,Io10}), showing that certain classes of free
ergodic probability measure preserving group actions $G \actson
(X,\mu)$ can be completely recovered from their associated II$_1$
factors $\rL^\infty(X) \rtimes G$.

However, the superrigidity question for group factors is much
harder, and all this progress in group measure space factors could
not be exploited to obtain even one single example of a W$^*$-{\it
superrigid icc group} $G$, i.e.\ for which $\rL G$ completely
remembers $G$, in the sense that any isomorphism of $\rL G$ and an
arbitrary group factor $\rL \Lambda$ forces the groups $G, \Lambda$
to be isomorphic.

In this paper we provide a large class of generalized wreath product
groups $G$ which are W$^*$-superrigid. For instance, we show that
given ANY non-amenable group $\Gamma$, its canonical
``augmentation'' $G = \mathbb (Z/2\mathbb Z)^{(I)} \rtimes
(\Gamma\wr \mathbb Z)$ is superrigid, where the set $I$ is the
quotient $(\Gamma\wr \mathbb Z)/\mathbb Z$ on which the group
$\Gamma\wr \mathbb Z=\Gamma^{(\mathbb Z)}\rtimes \mathbb Z$ acts by
left multiplication. In fact, we show that any isomorphism between
$\rL G$ and an arbitrary group factor $\rL \Lambda$ is implemented
by an isomorphism of the groups. More precisely, we prove the
following general result:

\begin{theorem}\label{thm.specialmain}
Let $\Gamma_0$ be any non-amenable group and let $S$ be any infinite amenable group.
Define the wreath product group $\Gamma = \Gamma_0^{(S)} \rtimes S$ and consider
the action of $\Gamma$ on $I = \Gamma/ S$ by left multiplication. Let $n$ be a square-free integer
and define the generalized wreath product group
$$G = \bigl({\textstyle\frac{\Z}{n \Z}}\bigr)^{(I)} \rtimes \Gamma \; .$$
If $\Lambda$ is any countable group and $\pi : \rL \Lambda \recht \rL(G)^t$ a surjective
$*$-isomorphism for some $t > 0$, then $t=1$ and $\Lambda \cong G$.

In the special case where $n = 2,3$, the $*$-isomorphism $\pi$ is necessarily group-like:
there exists an isomorphism of groups $\delta : \Lambda \recht G$, a character $\om : \Lambda \recht \T$
and a unitary $w \in \rL G$ such that
$$\pi(v_s) = \om(s) \, w \, u_{\delta(s)} \, w^* \quad\text{for all}\;\; s \in \Lambda \; .$$
Here $(v_s)_{s \in \Lambda}$ and $(u_g)_{g \in G}$ denote the canonical generating unitaries
of $\rL \Lambda$, resp.\ $\rL G$.
\end{theorem}

Theorem \ref{thm.main} below provides a much wider class of generalized wreath product groups
$G=\frac{\Z}{n\Z} \wr_I \Gamma$ such that the group factor $\rL G$ remembers the group $G$.

The conclusions of Theorem \ref{thm.specialmain} do not hold however
for plain wreath products $G = \frac{\Z}{n\Z} \wr \Gamma$.
Nevertheless we will see in Theorem \ref{thm.main-gen} that it is
still possible to describe more or less explicitly all groups
$\Lambda$ with $\rL G \cong \rL \Lambda$. But this description does
not allow to classify these groups $\Lambda$ up to isomorphism. The
groups $\Lambda$ with $\rL G \cong \rL \Lambda$ can be quite
different from $G$, as illustrated by the following result that we prove in Section \ref{sec.counterex}.

\begin{theorem} \label{thm.counterex}
Let $\Gamma$ be a non-trivial torsion free group and $H_0$ a non-trivial finite abelian group.
Then, there exists a torsion free group $\Lambda$ with $\rL(\Lambda) \cong \rL (H_0 \wr \Gamma)$.
In particular, $\Lambda \not\cong H_0 \wr \Gamma$.

Let $n \geq 2$ and let $H_0$ be a non-trivial finite abelian group. There are infinitely many
non-isomorphic groups $\Lambda$ for which $\rL \Lambda \cong \rL(H_0 \wr \PSL(n,\Z))$.
\end{theorem}

We mention that in the final Section \ref{sec.Wstarsuperrigidity} we
show that some of our methods allow to extend \cite[Theorem A]{Io10}
and prove W$^*$-superrigidity for Bernoulli actions $\Gamma \actson
(X_0,\mu_0)^\Gamma$ of groups $\Gamma$ that admit an infinite normal
subgroup with non-amenable centralizer. We refer to Theorem
\ref{thm.Wstarsuperrigid} for a precise statement.

\subsection*{Structure of the article and comments on the proofs}

The fact that large classes of generalized wreath product groups
turn out to be W$^*$-superrigid should come as no surprise, since
such groups have been recognized for some time to be ``exceptionally
rigid'' in the von Neumann algebra context (cf.
\cite{Po01a,Po03,Po04,Po06a,PV06,PV09,Io06,CI08,Io10}). This is due
to the Bernoulli type crossed-product decomposition that a wreath
product group has, $G=H\wr_I \Gamma=H^{(I)}\rtimes \Gamma$, a
feature that makes its associated von Neumann algebra $M=\rL G$
``distinctly soft'' on the side of $\rL H^{(I)}\subset M$, once $H$
is assumed amenable. Such softness is a consequence of the {\it
malleable deformations} that II$_1$ factors arising from Bernoulli
actions were shown to have (\cite{Po01a,Po03}). This property allows
the recovery of all ``rigid parts'' of $\rL \Gamma$, such as
subalgebras generated by subgroups $\Gamma_0\subset \Gamma$ having
either relative property (T), or non-amenable centralizer. Playing
rigidity against deformability properties of an algebra in this
manner became a paradigm of deformation/rigidity theory (see
\cite{Po01a,Po03,Po04,Po06a,Po06b}). Then in \cite{PV06} it was
realized that if $\Gamma \curvearrowright I$ is of the form $\Gamma
\curvearrowright \Gamma/\Gamma_0$, with $\Gamma_0$ a ``malnormal''
subgroup of $\Gamma$, the overall rigidity of $M$ can be
considerably enhanced, while the discovery in \cite{Io06} of a new
malleable deformation for generalized Bernoulli actions and wreath
product groups unraveled more of their rigidity properties. The
recent work in \cite{PV09,Io10} pushed the deformation/rigidity
analysis of such group actions even deeper, notably through the
systematic usage of ``comultiplication''-type embeddings $\Delta:M
\hookrightarrow M\overline{\otimes} M$ in \cite{Io10} (cf. also
\cite{PV09}).

In order to prove Theorem \ref{thm.specialmain} we use the entire
arsenal of ideas and techniques developed in these previous papers.
Yet recovering the discrete structure $G=H\wr_I \Gamma$ (rather than
the action $\Gamma \curvearrowright \rL H^{(I)}$, as in
\cite{Po04,PV09,Io10}) inside the algebra $\rL G$ requires more
intricate deformation/rigidity arguments and a lot of technical
effort. This work, which takes the entire Sections \ref{sec.tensorlength} through \ref{sec.proof},
leads us to a crucial correlation between $G$ and any other group
implementing the same von Neumann algebra. More precisely, we show
that the comultiplication on $\rL G$ induced by an arbitrary group
$\Lambda \subset \rL G$ satisfying $\rL \Lambda = \rL G$ is
unitarily conjugate to the initial comultiplication induced by $G$,
with the corresponding unitary satisfying a ``dual'' $2$-cocycle
relation\footnote{To be more precise we only find a unitary $\Omega$
satisfying the formulas \eqref{eq.specialform} on page \pageref{eq.specialform},
but this suffices to deduce that $\Omega$ is a dual $2$-cocycle.}.
One of the big novelties in this paper is how we derive an
isomorphism of the groups $G,\Lambda$ out of this $2$-cocycle. We do
this in Theorem \ref{thm.triv-2-coc}, which is essentially a
vanishing of $2$-cohomology result.

One should mention that a particular case of this result, which we
emphasize separately as Theorem \ref{thm.isogroupvnalg}, provides a
surprising characterization for the unitary conjugacy of arbitrary
icc groups $\Lambda$, $G$ giving the same II$_1$ factor, $\rL
\Lambda = \rL G$. To state it, we use the (asymmetric) Hausdorff
distance between subgroups $\cU$ and $\cV$ of the unitary group $\cU(M)$ of a II$_1$ factor,
defined by
$$\dist(\cU,\cV) := \sup_{u \in \cU} \Bigl( \inf_{v \in \cV} \|u-v\|_2\Bigr) \; .$$
Denote by $\T \cU$ the group of unitaries $\lambda u$, $\lambda \in \T$, $u \in \cU$ and
notice that $\dist(\T \cU,\T \cV)\leq \sqrt{2}$ for any subgroups
$\cU, \cV \subset \cU(M)$. We prove in Theorem
\ref{thm.isogroupvnalg} that if $M=\rL G = \rL \Lambda$ are two
group von Neumann algebra decompositions of the same II$_1$ factor
$M$ then $\dist(\T G, \T \Lambda) < \sqrt{2}$ if and only if $\T
\Lambda$ and $\T G$ are conjugate by a unitary in $M$.

To describe in more details the content of Sections \ref{sec.tensorlength}-\ref{sec.proof}, let $G =
H_0 \wr_I \Gamma$ be a generalized wreath product group as in
Theorem \ref{thm.specialmain} (or the more general Theorem
\ref{thm.main}). Write $M := \rL G$ and assume that $M = \rL
\Lambda$ is another group von Neumann algebra decomposition. Denote
by $\Delta : M \recht M \ovt M : \Delta(v_s) = v_s \ot v_s, s \in
\Lambda$, the comultiplication corresponding to the decomposition $M
= \rL \Lambda$. Observe that $M = \rL G$ can be viewed as the group
measure space construction $M = \rL^\infty(X_0^I) \rtimes \Gamma$,
where $X_0 = \widehat{H_0}$ is the Pontryagin dual of $H_0$ equipped
with the Haar probability measure and where $\Gamma \actson X_0^I$
is the generalized Bernoulli action.

In \cite{Io10}, a classification result for embeddings $\Delta : M
\recht M \ovt M$ was obtained in the case where $M =
\rL^\infty(X_0^\Gamma) \rtimes \Gamma$ is the group measure space
II$_1$ factor given by the \emph{plain} Bernoulli action of an icc
property (T) group $\Gamma$, or more generally an icc group $\Gamma$
that admits an infinite normal subgroup with the relative property
(T). We extend these results to generalized Bernoulli actions. This
generalization is technically painful, but unavoidable in the light
of Theorem \ref{thm.counterex}.

We analyze the embedding $\Delta : M \recht M \ovt M$ in three different steps, corresponding to the Sections \ref{sec.tensorlength}, \ref{sec.intertwine-abelian} and \ref{sec.conjugacy-actions}. In this analysis we use much of the ideas and techniques developed in deformation/rigidity theory over the last years. Nevertheless, apart
from the preliminary Section \ref{sec.prelim} where we also recall the notion of intertwining bimodules \cite{Po03}, our article is essentially self-contained and the Sections \ref{sec.tensorlength}, \ref{sec.intertwine-abelian} and \ref{sec.conjugacy-actions} contain independent results, each having an interest on their own.

We write $A = \rL^\infty(X_0^I)$ and denote by $(u_g)_{g \in \Gamma}$ the canonical unitaries in the crossed product $M = A \rtimes \Gamma$.

\begin{itemlist}
\item In Section \ref{sec.tensorlength} we elaborate results from \cite{Po03,Io06} implying that under
the right assumptions rigid subalgebras of generalized Bernoulli crossed products
$M = \rL^\infty(X_0^I) \rtimes \Gamma$ have an intertwining bimodule into $\rL \Gamma$, see Corollary
\ref{cor.malleable-rigid}. Following \cite{Io06} we consider the ``tensor length deformation''
$\theta_\rho : M \recht M$ which is roughly defined as $\theta_\rho(F u_g) = \rho^n F u_g$ when $g \in \Gamma$
and $F \in \rL^\infty(X_0^I)$ only depends on $n$ variables in $I$. In Theorem \ref{thm.malleable}
we describe which subalgebras $Q \subset M$ have the property that $\theta_\rho$ converges uniformly
to the identity on the unit ball of $Q$. This result readily applies when $Q \subset M$ has the relative property (T),
but also when $Q$ has a non-amenable relative commutant, by the spectral gap argument from \cite{Po06a}.

Applied to the above comultiplication $\Delta : M \recht M \ovt M$ we will be able to assume that after a unitary conjugacy $\Delta(\rL \Gamma) \subset \rL (\Gamma \times \Gamma)$.

\item In Section \ref{sec.intertwine-abelian} we prove the following. Assume that $M = \rL^\infty(X_0^I) \rtimes \Gamma$ is a generalized Bernoulli crossed product and write $A = \rL^\infty(X_0^I)$. If $D \subset M \ovt M$ is an abelian von Neumann subalgebra that is normalized by many unitaries in $\rL (\Gamma \times \Gamma)$ and if a number of conditions are satisfied, then the relative commutant $D' \cap M \ovt M$ can be essentially unitarily conjugated into $A \ovt A$. This result and its proof are very similar to \cite[Theorem 6.1]{Io10} and very much inspired by the clustering sequences techniques from \cite[Sections 1-4]{Po04}.

Applied to the above comultiplication $\Delta : M \recht M \ovt M$ we may essentially assume that after a unitary conjugacy $\Delta(A)' \cap M \ovt M = A \ovt A$.

\item In Section \ref{sec.conjugacy-actions} we provide a very general conjugacy criterion for actions.
Let $N = B \rtimes \Lambda$ and $M = A \rtimes \Gamma$ be group measure space II$_1$ factors.
Assume that $N \subset M$ in such a way that there exist intertwining bimodules from $B$ into $A$ and
from $\rL \Lambda$ into $\rL \Gamma$. Under a few extra conditions, we conclude that there exists
a unitary $\Omega \in M$ such that $\Ad \Omega$ maps $B$ into $A$ and $\T \Lambda$ into $\T \Gamma$.
We refer to Theorem \ref{thm.conjugacy-actions} for a precise statement.

Applied to the above comultiplication $\Delta : M \recht M \ovt M$, it ultimately follows that there exists a unitary $\Omega \in M \ovt M$ such that
\begin{equation}\label{eq.specialform}
\Om^* \de(u_g) \Om = \om(g) \, u_{\delta_1(g)} \ot u_{\delta_2(g)}
\quad\text{for all}\;\; g \in \Gamma \quad\text{and}\quad \Om^* \de(A) \Om \subset A \ovt A \; ,
\end{equation}
for some group homomorphisms $\delta_i : \Gamma \recht \Gamma$, $\om: \Gamma \recht \T$.
\end{itemlist}

Once \eqref{eq.specialform} above is established, we conclude that $\Omega \in \rL \Lambda \ovt \rL \Lambda$ satisfies a $2$-cocycle and a symmetry relation. By the above mentioned vanishing of $2$-cohomology Theorem \ref{thm.triv-2-coc}, the main Theorems \ref{thm.specialmain} and \ref{thm.main} will follow.

We finally refer to the lecture notes \cite{Va11} for an introduction to the results and techniques of this paper and \cite{Io10}.

\section{Preliminaries}\label{sec.prelim}

\subsection{Intertwining-by-bimodules}\label{subsec.intertwine}

We recall from \cite[Theorem 2.1 and Corollary 2.3]{Po03} the theory of \emph{intertwining-by-bimodules,} summarized in the following definition.

\begin{definition}\label{def.intertwine}
Let $(M,\tau)$ be a tracial von Neumann algebra with separable predual and $P,Q \subset M$ possibly non-unital von Neumann subalgebras. We write $P \embed_M Q$ (or $P \embed Q$ if there is no risk of confusion) when one of the following equivalent conditions is satisfied.
\begin{itemlist}
\item There exist projections $p \in P$, $q \in Q$, a $*$-homomorphism $\vphi : pPp \recht qQq$ and a non-zero partial isometry $v \in pMq$ such that $x v = v \vphi(x)$ for all $x \in pPp$.
\item There exist a projection $q \in \M_n(\C) \ot Q$, a $*$-homomorphism $\vphi : P \recht q(\M_n(\C) \ot Q)q$ and a non-zero partial isometry $v \in (\M_{1,n}(\C) \ot 1_P M)q$ such that $x v = v\vphi(x)$ for all $x \in P$.
\item It is impossible to find a sequence $u_n \in \cU(P)$ satisfying $\|E_Q(x u_n y^*)\|_2 \recht 0$ for all $x,y \in 1_Q M 1_P$.
\item There exists a subgroup $\cU \subset \cU(P)$ generating $P$ as a von Neumann algebra for which it is impossible to find a sequence $u_n \in \cU$ satisfying $\|E_Q(x u_n y^*)\|_2 \recht 0$ for all $x,y \in 1_Q M 1_P$.
\end{itemlist}
\end{definition}

\begin{remark}\label{rem.intertwine}
We freely use the following facts about the embedding property $\embed$.

If $Q_k \subset M$ is a sequence of von Neumann subalgebras and $P \not\embed Q_k$ for all $k$, considering the diagonal inclusion of $P$ into matrices over $M$ together with the subalgebra $Q_1 \oplus \cdots \oplus Q_l$, we find a sequence of unitaries $u_n \in \cU(P)$ such that for all $k$ and all $x,y \in 1_{Q_k} M 1_P$, we have $\|E_{Q_k}(x u_n y^*)\|_2 \recht 0$ (see e.g.\ \cite[Remark 3.3]{Va07} for details).

If $p \in P$ is a non-zero projection and $pPp \embed Q$, then $P \embed Q$ (see e.g.\ \cite[Lemma 3.4]{Va07}). Also, if $P \embed Q$ and $B \subset Q$ has finite index, then $P \embed B$ (see e.g.\ \cite[Lemma 3.9]{Va07}). Finally, although $\embed$ is not transitive, the following holds for von Neumann subalgebras $P$ and $B \subset Q$. If $P \embed Q$ and $P \not\embed B$, the $*$-homomorphism $\vphi$ in Definition \ref{def.intertwine} can be chosen in such a way that the subalgebra $\vphi(P) \subset Q$ satisfies $\vphi(P) \not\embed_Q B$ (see e.g.\ \cite[Remark 3.8]{Va07}).
\end{remark}

\subsection{Bimodules and weak containment}\label{subsec.bimod}

Let $M,N$ be tracial von Neumann algebras. An \emph{$M$-$N$-bimodule $\bim{M}{\cH}{N}$} is a Hilbert space $\cH$ equipped with a normal unital $*$-homomorphism $\pi: M \recht \B(\cH)$ and a normal unital $*$-anti-homomorphism $\pi':N \recht \B(\cH)$ such that $\pi(M)$ and $\pi'(N)$ commute. We write $x\xi y$ instead of $\pi(x) \pi'(y) \xi$.
The bimodule $\bim{M}{\rL^2(M)}{M}$ is called the \emph{trivial bimodule} and $\bim{(M \ot 1)}{\rL^2(M \ovt M)}{(1 \ot M)}$ is called the \emph{coarse bimodule.}
Given the bimodules $\bim{M}{\cH}{N}$ and $\bim{N}{\cK}{P}$, one can define the \emph{Connes tensor product} $\cH \ot_N \cK$ which is an $M$-$P$-bimodule, see \cite[V.Appendix B]{Co94}.

Every $M$-$N$-bimodule $\bim{M}{\cH}{N}$ gives rise to a $*$-homomorphism $\pi_\cH : M \otalg N\op \recht \B(\cH)$ given by $\pi_\cH(x \ot y)\xi = x\xi y$. We say that $\bim{M}{\cH}{N}$ is \emph{weakly contained} in $\bim{M}{\cK}{N}$, and write $\cH \weakcont \cK$, if $\|\pi_\cH(T)\| \leq \|\pi_\cK(T)\|$ for all $T \in M \otalg N\op$. Recall that $\cH \weakcont \cK$ if and only if $\bim{M}{\cH}{N}$ lies in the closure (for the Fell topology) of all finite direct sums of copies of $\bim{M}{\cK}{N}$.

For later use we record the following easy lemma and give a proof for the convenience of the reader.

\begin{lemma}\label{lem.sequence}
Let $(M,\tau)$ be a tracial von Neumann algebra and $P \subset pMp$ a von Neumann subalgebra. Let $\bim{P}{\cH}{M}$ be a $P$-$M$-bimodule and $\kappa > 0$. Assume that $\xi_n \in \cH$ satisfies
$$\|a \xi_n - \xi_n a \| \recht 0 \;\;\forall a \in P \;\; , \;\; \|\xi_n x \| \leq \kappa \|x\|_2 \;\;\forall n \in \N \;,\; x \in M \;\; , \;\; \limsup_n \|\xi_n p\| > 0 \; .$$
Then there is a non-zero projection $p_1 \in P' \cap pMp$ such that $\bim{P}{\rL^2(p_1 M)}{M}$ is weakly contained in $\bim{P}{\cH}{M}$.
\end{lemma}
\begin{proof}
Replacing $\xi_n$ by $\xi_n p$, we may assume that $\xi_n p = \xi_n$ for all $n$.
Since $\|\xi_n x\| \leq \kappa \|x\|_2$ for all $x \in M$, define $T_n \in p M^+ p$ satisfying $\|T_n \|\leq \kappa$ and $\langle \xi_n,\xi_n x\rangle = \tau(T_n x)$ for all $x \in M$. We have $\|[a,T_n]\|_1 \recht 0$ for all $a \in P$. Since $\tau(T_n) = \| \xi_n\|^2$ and $\|T_n\|$ is bounded, we can pass to a subsequence and assume that $T_n \recht T$ weakly with $T \in p M^+ p$, $\tau(T) > 0$. Note that $T \in P' \cap p M^+ p$. Take $S \in P' \cap p M^+ p$ such that $T^{1/2} S$ is a non-zero projection $p_1$. Define $\eta_n = \xi_n S$. It follows that $\langle \eta_n , a \eta_n x \rangle \recht \tau(p_1 a x)$ for all $a \in P$, $x \in M$. Hence, $\bim{P}{\rL^2(p_1 M)}{M}$ is weakly contained in $\bim{P}{\cH}{M}$.
\end{proof}

\subsection{Relative property (T)}\label{subsec.relT}

Let $(M,\tau)$ be a tracial von Neumann algebra and $P \subset M$ a
von Neumann subalgebra. Following \cite[Proposition 4.1]{Po01b}, we
say that $P \subset M$ has the \emph{relative property (T)} if every
sequence $\vphi_n : M \recht M$ of normal completely positive maps
that are sub-unital ($\vphi_n(1) \leq 1$), subtracial ($\tau \circ
\vphi_n \leq \tau$) and satisfy $\|\vphi_n(x) - x \|_2 \recht 0$ for
all $x \in M$, converges to the identity uniformly on the unit ball
of $P$, i.e.\ $$\sup_{x \in P, \|x\|\leq 1} \|\vphi_n(x) - x \|_2
\recht 0 \; .$$

If $\Gamma_0 < \Gamma_1$ are countable groups, by \cite[Proposition
5.1]{Po01b} the inclusion $\rL \Gamma_0 \subset \rL \Gamma_1$ has
the relative property (T) if and only if $\Gamma_0 < \Gamma_1$ has
the relative property (T) in the usual group theoretic sense.

\subsection{Relative amenability}\label{subsec.relamen}

Recall that a tracial von Neumann algebra $(M,\tau)$ is called \emph{amenable} if the trivial $M$-$M$-bimodule is weakly contained in the coarse $M$-$M$-bimodule.

Fix a tracial von Neumann algebra $(M,\tau)$ and a von Neumann subalgebra $Q \subset M$. \emph{Jones' basic construction} $\langle M,e_Q \rangle$ is defined as the von Neumann subalgebra of $\B(\rL^2(M))$ generated by $M$ (acting on the left) and the orthogonal projection $e_Q$ of $\rL^2(M)$ onto $\rL^2(Q)$. Note that $\langle M,e_Q \rangle$ equals the commutant of the right $Q$-action on $\rL^2(M)$. The basic construction $\langle M,e_Q \rangle$ comes with a semi-finite faithful trace $\Tr$ satisfying $\Tr(a e_Q b) = \tau(ab)$ for all $a,b \in M$. We denote, for $p = 1,2$, by $\rL^p(\langle M,e_Q \rangle)$ the corresponding $\rL^p$-spaces.

Following \cite[Definition 2.2]{OP07}, a von Neumann subalgebra $P \subset pMp$ is said to be \emph{amenable relative to $Q$} if $\bim{P}{\rL^2(p \langle M,e_Q\rangle)}{M}$ weakly contains $\bim{P}{\rL^2(p M)}{M}$. By \cite[Theorem 2.1]{OP07}, $P$ is amenable relative to $Q$ if and only if there exists a sequence $T_n \in p \rL^1(\langle M,e_Q \rangle)^+ p$ satisfying
$$\|a T_n - T_n a \|_1 \recht 0 \;\;\text{for all}\;\; a \in P \;\;\text{and}\;\; \Tr(T_n x) \recht \tau(x) \;\;\text{for all}\;\; x \in pMp \; .$$

We say that a von Neumann subalgebra $P \subset pMp$ is \emph{strongly non-amenable relative to $Q$} if for all non-zero projections $p_1 \in P' \cap pMp$, the von Neumann algebra $P p_1$ is non-amenable relative to $Q$. Equivalently, none of the bimodules $\bim{P}{\rL^2(p_1 M)}{M}$ with $p_1$ a non-zero projection in $P' \cap pMp$, is weakly contained in $\bim{P}{\rL^2(p \langle M,e_Q \rangle)}{M}$.

If $P \subset pMp$ is amenable relative to $Q$ and if $A \subset e M e$ is a von Neumann subalgebra satisfying $A \embed_M P$, then there exists a non-zero projection $f \in A' \cap eMe$ such that $Af$ is amenable relative to $Q$.

Note that $\bim{M}{\rL^2(\langle M,e_Q \rangle)}{M} \cong \bim{M}{\bigl(\rL^2(M) \ot_Q \rL^2(M)\bigr)}{M}$. In particular, a von Neumann subalgebra $P \subset p (N \ovt M)p$ is amenable relative to $N \ot 1$ if and only if $\bim{P}{\rL^2(p (N \ovt M))}{(N \ovt M)}$ is weakly contained in $\bim{(P \ot 1)}{\rL^2(p (N \ovt M) \ovt M)}{(N \ovt 1 \ovt M)}$.

\section{Symmetric dual $2$-cocycles and isomorphism of group von Neumann algebras}\label{sec.intertwine-groups}

The main aim of this section is to prove the following result. Whenever $\Lambda$ is a countable group and $(v_s)_{s \in \Lambda}$ are the canonical unitaries generating $\rL \Lambda$, we denote by $\T \Lambda$ the group of unitaries in $\rL \Lambda$ of the form $\lambda v_s$ for $\lambda \in \T$ and $s \in \Lambda$.

\begin{theorem}\label{thm.isogroupvnalg}
Let $\Gamma$ and $\Lambda$ be icc groups and $\rL \Gamma = \rL \Lambda$. Denote by $(u_g)_{g \in \Gamma}$ and $(v_s)_{s \in \Lambda}$ the respective canonical unitaries. Denote by $$\dist(\T \Gamma, \T \Lambda) = \sup_{u \in \T \Gamma} \Bigl(\inf_{v \in \T \Lambda} \|u-v\|_2 \Bigr)$$ the (asymmetric) upper Hausdorff distance. Then the following two statements are equivalent.
\begin{itemlist}
\item $\dist(\T \Gamma, \T \Lambda) < \sqrt{2}$.
\item There exists a unitary $w \in \rL \Lambda$, a character $\gamma : \Gamma \recht \T$ and an isomorphism of groups $\delta : \Gamma \recht \Lambda$ such that
$$w u_g w^* = \gamma(g) v_{\delta(g)} \quad\text{for all}\;\; g \in \Gamma \; .$$
\end{itemlist}
\end{theorem}

Defining the \emph{height} of an element $x \in \rL \Lambda$ as
\begin{equation}\label{eq.defheight}
h_\Lambda(x) := \max \{ |\tau(x v_s^*)| \mid s \in \Lambda \} \; ,
\end{equation}
it is an easy exercise to check that
$$\dist(x,\T\Lambda) = \sqrt{1 + \|x\|_2^2 - 2 h_\Lambda(x)} \; .$$
In particular, the assumption $\dist(\T \Gamma, \T \Lambda) < \sqrt{2}$ in Theorem \ref{thm.isogroupvnalg} is equivalent with the existence of a $\delta > 0$ such that $h_\Lambda(u_g) \geq \delta$ for all $g \in \Gamma$.

\begin{remark}
Assume that $\Gamma$ and $\Lambda$ are countable groups and that $\rL \Gamma$ is a von Neumann subalgebra of $\rL \Lambda$. Assume that $\dist(\T \Gamma, \T \Lambda) < \sqrt{2}$. We do not know whether it is still true that there exists a unitary $w \in \rL \Lambda$, a character $\gamma : \Gamma \recht \T$ and an injective group homomorphism $\delta : \Gamma \recht \Lambda$ such that $wu_g w^* = \gamma(g) v_{\delta(g)}$ for all $g \in \Gamma$.
\end{remark}

We will not be able to prove our main Theorem \ref{thm.main} by a direct application of Theorem \ref{thm.isogroupvnalg}. We rather need the following \emph{vanishing of cohomology theorem} which at the same time will lead to a proof of Theorem \ref{thm.isogroupvnalg}.

Recall that every group von Neumann algebra $\rL \Lambda$ is equipped with a natural normal unital $*$-homomorphism, called \emph{comultiplication}, $\de : \rL \Lambda \recht \rL \Lambda \ovt \rL \Lambda$ given by $\de(v_s) = v_s \ot v_s$ for all $s \in \Lambda$. Observe that $(\de \ot \id)\de = (\id \ot \de)\de$ and that $\si \circ \de = \de$, where $\si(x \ot y) = y \ot x$ is the flip automorphism.
We also use the \emph{tensor leg numbering notation} for operators in tensor products. In this manner, $X_{21} = \si(X)$, $X_{23} = 1 \ot X$, $X_{13} = (\si \ot \id)(1 \ot X)$, etc.

\begin{theorem} \label{thm.triv-2-coc}
Let $\Lambda$ be a countable group and $\de : \rL \Lambda \recht \rL \Lambda \ovt \rL \Lambda$ the comultiplication. Suppose that $\Om \in \rL \Lambda \ovt \rL \Lambda$ is a unitary satisfying
$$\Omega_{21} = \mu \Omega \quad\text{and}\quad (\de \ot \id)(\Om) (\Om \ot 1) = \eta (\id \ot \de)(\Om) (1 \ot \Om)$$
for some $\mu,\eta \in \T$. Then, $\mu = \eta = 1$ and there exists a unitary $w \in \rL \Lambda$ such that
$$\Omega = \de(w^*) (w \ot w) \; .$$
\end{theorem}

\begin{proof}
Put $M = \rL \Lambda$ and $H = \ell^2(\Lambda)$. Define the unitary operators $\lambda_h, \rho_h$, $h \in \Lambda$ by the formulae $\lambda_h \delta_k = \delta_{hk}$ and $\rho_h \delta_k = \delta_{kh^{-1}}$. Realize $M := \{\rho_h \mid h \in \Lambda\}\dpr$.

We view $\ell^\infty(\Lambda)$ acting on $H$ by multiplication operators. We define the unitary
$$W \in \ell^\infty(\Lambda) \ovt M \quad\text{given by}\quad W(\delta_g \ot \delta_h) = \delta_g \ot \rho_g \delta_h = \delta_g \ot \delta_{hg^{-1}} \; .$$
Define the unitary
$$X \in \B(H) \ovt M : X = W \Omega \; .$$
It is easy to check that $\de(x) = W^*(x \ot 1) W$ for all $x \in M$. Also,
\begin{equation}\label{eq.formula}
(\id \ot \de)(X) (1 \ot \Om) = \overline{\eta} X_{13} X_{12} \; .
\end{equation}

Whenever $\cV \subset \B(H)$, we denote by $[\cV]$ the norm closed linear span of $\cV$ inside $\B(H)$. Define
$$A := [ (\id \ot \om)(X) \mid \om \in M_* ] \; .$$

{\bf Step 1.} The norm closed linear subspace $A \subset \B(H)$ is actually a C$^*$-algebra acting non-degenerately on $H$ (i.e.\ [A H] = H). Moreover, $\lambda_g A \lambda_g^* = A$ for all $g \in \Lambda$.

Applying $\id \ot \om_1 \ot \om_2$ to \eqref{eq.formula}, we get
\begin{align*}
[A A] &= [(\id \ot \om_1 \ot \om_2)\bigl( (\id \ot \de)(X) (1 \ot \Om) \bigr) \mid \om_1,\om_2 \in M_* ] \\ &= [(\id \ot \Om \om)(\id \ot \de)(X) \mid \om \in (M \ovt M)_* ] = [(\id \ot \om\de)(X) \mid \om \in (M \ovt M)_* ] = A \; .
\end{align*}

Since $\de(x) = W^* (x \ot 1) W$, we can rewrite \eqref{eq.formula} in the form
\begin{equation}\label{eq.alternative}
\overline{\eta} X_{12} X_{23}^* = X^*_{13} W^*_{23} X_{12} \; .
\end{equation}
Applying $\id \ot \om_1 \ot \om_2$, $\om_1,\om_2 \in \B(H)_*$, we get
$$A = [(\id \ot \om_1 \ot \om_2)(X^*_{13} W^*_{23} X_{12}) \mid \om_1,\om_2 \in \B(H)_*] \; .$$
Denote by $P_g \in \ell^\infty(\Lambda)$ the natural minimal projections. Then $\B(H)_* = [\om P_g \mid \om \in \B(H)_* , g \in \Lambda]$. Hence,
\begin{align*}
A &= [(\id \ot \om_1 P_g \ot \om_2)(X^*_{13} W^*_{23} X_{12}) \mid \om_1,\om_2 \in \B(H)_*, g \in \Lambda] \\
&=[(\id \ot \om_1 \ot \om_2)(X^*_{13} (1 \ot P_g \ot 1)W^*_{23} X_{12}) \mid \om_1,\om_2 \in \B(H)_*, g \in \Lambda] \; .
\end{align*}
Since $(P_g \ot 1)W^* = P_g \ot \rho_g^*$, we get
\begin{align*}
A &= [(\id \ot \om_1 P_g \ot \rho_g^* \om_2)(X^*_{13} X_{12}) \mid \om_1,\om_2 \in \B(H)_*, g \in \Lambda] \\
&= [(\id \ot \om_1 \ot \om_2)(X^*_{13} X_{12}) \mid \om_1,\om_2 \in \B(H)_*]  = [A^* A] \; .
\end{align*}
Since $A = [A A]$ and $A = [A^* A]$, it follows that $A$ is a C$^*$-algebra. Also,
$$[A H] = [(\id \ot \om)(X)H \mid \om \in M_*] = [(1 \ot \xi_1^*) X (H \ot \xi_2) \mid \xi_1,\xi_2 \in H] = H$$
since $X$ is a unitary operator. So, the C$^*$-algebra $A$ acts non-degenerately on $H$.

Since $X(\lambda_g \ot 1) X^* = \lambda_g \ot \rho_g$, also
$$\lambda_g^* (\id \ot \om)(X) \lambda_g = (\id \ot \om\rho_g)(X) \; .$$
Hence, $\lambda_g$ normalizes $A$.

{\bf Step 2.} We have $\mu = 1$ and $A$ is an abelian C$^*$-algebra.

Applying $\id \ot \sigma$ to \eqref{eq.formula} and using the fact that $\de = \sigma \circ \de$, one gets $X_{12} X_{13} = \mu X_{13} X_{12}$.
Applying $\id \ot \om_1 \ot \om_2$ to this formula, we get that $ab = \mu b a$ for all $a,b \in A$. So, this formula also holds when $a$ and $b$ belong to $A\dpr$, which contains $1$. But then, $\mu = 1$ and $A$ follows abelian.

{\bf Step 3.} The closed linear span $B:=[A \lambda_g \mid g \in \Lambda]$ is a C$^*$-algebra that is ultraweakly dense in $\B(H)$.

Since the unitaries $\lambda_g$ normalize $A$, it follows that $B$ is a C$^*$-algebra. Also, $A \subset B$ and hence, $B$ acts non-degenerately on $H$. It suffices to prove that $B' = \C 1$. Since the commutant of $\{\lambda_g \mid g \in \Lambda\}$ equals $M$, we have to prove that $M \cap A' = \C 1$. Take $x \in M \cap A'$. Denote $\cA = A\dpr$ and note that $X \in \cA \ovt M$. Since $\cA$ is abelian, we have $X_{12} X_{13} = X_{13} X_{12}$. Combining with \eqref{eq.alternative}, we have
$$W^*_{23} X_{12} X_{23} = \overline{\eta} X_{12} X_{13} \; .$$
Hence,
$$W^*_{23} X_{12} X_{23} (1 \ot x \ot 1) X_{23}^* X_{12}^* W_{23} = X_{12} X_{13} (1 \ot x \ot 1) X_{13}^* X_{12}^* \; .$$
Since $x \in \cA'$, the left hand side equals $(\id \ot \de)(X (1 \ot x)X^*)$, while the right hand side equals $X(1 \ot x)X^* \ot 1$. Denote by $\tau$ the natural trace on $M = \rL \Lambda$. Then, $(\id \ot \tau)\de(y) = \tau(y) 1$ for all $y \in M$. Applying $\id \ot \id \ot \tau$ to the equality $(\id \ot \de)(X (1 \ot x)X^*) = X(1 \ot x)X^* \ot 1$, we find $y \in \cA$ such that $X (1 \ot x) X^* = y \ot 1$. But then,
$$1 \ot x = X^* (y \ot 1) X = y \ot 1 \; .$$
We finally conclude that $x$ is a scalar multiple of $1$.

{\bf Step 4.} The formula $E(x) = (\id \ot \tau)(X (x \ot 1) X^*)$ provides a normal conditional expectation of $\B(H)$ onto $\cA$.

Since $\cA$ is abelian, we have $E(x) = x$ for all $x \in \cA$. So, it remains to prove that $E(x) \in \cA$ for all $x \in \B(H)$. By step 3 it suffices to check this for $x = a \lambda_g$, $a \in A$, $g \in \Lambda$. Since $a \ot 1$ and $X$ commute, we have
$$E(a \lambda_g) = a (\id \ot \tau)(X(\lambda_g \ot 1) X^*) = a (\id \ot \tau)(\lambda_g \ot \rho_g) = \begin{cases} a &\;\;\text{if $g = e$,}\\
0 &\;\;\text{if $g \neq e$.}\end{cases}$$

{\bf End of the proof.} Step 4 implies that $\cA$ is a discrete von Neumann algebra. Let $p \in \cA$ be a non-zero minimal projection. Since $\cA$ is abelian, define the unitary $w \in M$ such that $X(p \ot 1) = \eta p \ot w$. Multiplying \eqref{eq.formula} with $p \ot 1 \ot 1$, we get that $\de(w) \Omega = w \ot w$. So, $\Om = \de(w^*) (w \ot w)$. Then also $(\de \ot \id)(\Om)(\Om \ot 1) = (\id \ot \de)(\Om)(1 \ot \Om)$, implying that $\eta = 1$.
\end{proof}

Before proving Theorem \ref{thm.isogroupvnalg}, we state and prove the following lemma which has some interest of its own.

Recall that a unitary representation of a countable group is called \emph{weakly mixing} if $\{0\}$ is the only finite dimensional invariant subspace.

\begin{lemma} \label{lem.intertwine}
Let $\Gamma,\Lambda$ be countable groups and assume that $\rL \Gamma \subset \rL \Lambda$. Denote by $(u_g)_{g \in \Gamma}$ the canonical unitaries in $\rL \Gamma$. Denote $M = \rL \Lambda$ and let $(v_s)_{s \in \Lambda}$ be the canonical unitaries in $\Lambda$. Let $\de : \rL \Lambda \recht \rL \Lambda \ovt \rL \Lambda$ be the comultiplication. Assume that the unitary representation $\Ad u_g$ of $\Gamma$ on $\rL^2(M) \ominus \C 1$ is weakly mixing.

If $\Omega \in M \ovt M$ is a unitary satisfying $\Omega (u_g \ot u_g) \Omega^* \in \de(M)$ for all $g \in \Gamma$, there exist unitaries $w,v \in M$, a character $\gamma : \Gamma \recht \T$ and an injective group homomorphism $\rho : \Gamma \recht \Lambda$ such that
$$w u_g w^* = \gamma(g) v_{\rho(g)} \;\;\text{for all}\;\; g \in \Gamma \quad\text{and}\quad \Omega = \de(v^*)(w \ot w) \; .$$
\end{lemma}
\begin{proof}
Define $\pi : \Gamma \recht \cU(M)$ such that $\de(\pi(g)) \Omega = \Omega(u_g \ot u_g)$ for all $g \in \Gamma$. Write $X = (\de \ot \id)(\Om^*)(\id \ot \de)(\Om)$. Then, $X \in M \ovt M \ovt M$ is unitary and satisfies
\begin{equation}\label{eq.eqX}
(\de(u_g) \ot u_g) X = X (u_g \ot \de(u_g))
\end{equation}
for all $g \in \Gamma$. Define $Y = (X \ot 1) (1 \ot X)$, which is a unitary in $M \ovt M \ovt M \ovt M$ satisfying
$$(\de(u_g) \ot u_g \ot u_g) Y = Y (u_g \ot u_g \ot \de(u_g))$$
for all $g \in \Gamma$. It follows that the unitary representation $\xi \mapsto (u_g \ot u_g) \xi \de(u_g)^*$ of $\Gamma$ on $\rL^2(M \ovt M)$ is not weakly mixing. This yields a finite dimensional unitary representation $\eta : \Gamma \recht \cU(\C^n)$ and a non-zero vector $\xi \in \C^n \ot \rL^2(M \ovt M)$ satisfying
$$(\eta(g) \ot u_g \ot u_g) \xi = \xi \de(u_g)$$
for all $g \in \Gamma$. We may assume that $\eta$ is irreducible. Since $\Ad (u_g \ot u_g)$ is weakly mixing on $\rL^2(M \ovt M) \ominus \C 1$ and since $\eta$ is irreducible, it follows that $\xi \xi^*$ is a multiple of $1$. Hence, $n = 1$ and we have found a unitary $Z \in M \ovt M$ and a character $\gamma : \Gamma \recht \T$ satisfying $\gamma(g) \de(u_g) Z = Z(u_g \ot u_g)$ for all $g \in \Gamma$.

Since $\si \circ \de = \de$, it follows that $Z_{21}^* Z$ commutes with $u_g \ot u_g, g \in \Gamma$ and hence, is a scalar multiple of $1$. Since $(\de \ot \id) \de = (\id \ot \de) \de$, it also follows that $(1 \ot Z)^* (\id \ot \de)(Z)^*(\de \ot \id)(Z)(Z \ot 1)$ commutes with $u_g \ot u_g \ot u_g$ for all $g \in \Gamma$ and hence, is a scalar multiple of $1$. By Theorem \ref{thm.triv-2-coc}, we find a unitary $w \in M$ such that $Z = \de(w^*) (w \ot w)$.

It follows that $\gamma(g) \de(w u_g w^*) = w u_g w^* \ot w u_g w^*$ for all $g \in \Gamma$. This means that $w u_g w^* = \gamma(g) v_{\rho(g)}$ for an injective group homomorphism $\rho : \Gamma \recht \Lambda$ (see Lemma \ref{lem.elemcomult} below for this well known fact). Put $\Lambda_0 = \rho(\Gamma)$. Since $\Ad u_g$ is a weakly mixing representation of $\Gamma$ on $\rL^2(M)$, also $(\Ad v_s)_{s \in \Lambda_0}$ is weakly mixing, meaning that $\Lambda_0 \subset \Lambda$ has the relative icc property: $\{s t s^{-1} \mid s \in \Lambda_0 \}$ is infinite for all $t \in \Lambda - \{e\}$.

Since $\Om (u_g \ot u_g) \Om^* \in \de(M)$ for all $g \in \Gamma$, it follows that
\begin{equation}\label{eq.omom}
\Om (w^* \ot w^*) \; (v_s \ot v_s) \; (w \ot w)\Om^* \in \de(M)
\end{equation}
for all $s \in \Lambda_0$. Since $\Lambda_0 \subset \Lambda$ has the relative icc property, we can take a sequence $s_n \in \Lambda_0$ such that $s_n t s_n^{-1} \recht \infty$ for all $t \in \Lambda - \{e\}$. It follows that
$$\|E_{\de(M)}(a (v_{s_n} \ot v_{s_n}) b) - E_{\de(M)}(a) \, \de(v_{s_n}) \, E_{\de(M)}(b) \|_2 \recht 0$$
for all $a,b \in M \ovt M$. Indeed, it suffices to check this for $a$ and $b$ of the form $v_r \ot v_t$, $r,t \in \Lambda$. Together with \eqref{eq.omom}, it follows that $\|E_{\de(M)}(\Om (w^* \ot w^*))\|_2 = 1$, meaning that $\Om(w^* \ot w^*) \in \de(M)$. We have found the required unitary $v \in M$ such that $\Om = \de(v^*)(w \ot w)$.
\end{proof}

We are now ready to prove Theorem \ref{thm.isogroupvnalg}.

\begin{proof}[Proof of Theorem \ref{thm.isogroupvnalg}]
Denote by $\de : \rL \Lambda \recht \rL \Lambda \ovt \rL \Lambda$ the canonical comultiplication. Put $\rL \Gamma = M = \rL \Lambda$ and denote by $\tau$ the trace on $M$. Whenever $x \in M$, we denote by $x^s$, $s \in \Lambda$ the Fourier coefficient $x^s := \tau(x v_s^*)$. As above we define for all $x \in M$ the height $h_\Lambda(x) = \max \{ |(x)^s| \mid s \in \Lambda \}$.

First assume that $\dist(\T \Gamma,\T \Lambda) < \sqrt{2}$.
By the discussion after the formulation of Theorem \ref{thm.isogroupvnalg}, we find a $\delta > 0$ such that $h_\Lambda(u_g) \geq \delta$ for all $g \in \Gamma$. A straightforward computation then gives
$$(\tau \ot \tau)( (\de(u_g) \ot u_g) (u_g \ot \de(u_g))^*) = \sum_{s \in \Lambda} |(u_g)^s|^4 \geq \delta^4$$
for all $g \in \Gamma$. So, there exists a non-zero $X \in M \ovt M \ovt M$ satisfying
$$(\de(u_g) \ot u_g) X = X (u_g \ot \de(u_g))$$
for all $g \in \Gamma$.

We also have $M = \rL \Gamma$. So, $\Gamma$ is an icc group and $(\Ad u_g)_{g \in \Gamma}$ is a weakly mixing representation of $\Gamma$ on $\rL^2(M) \ominus \C 1$. Since $XX^*$ commutes with all $\de(u_g) \ot u_g$, $g \in \Gamma$, it follows that $XX^* \in ( \de(M)' \cap M \ovt M) \ot 1$. Since $\Lambda$ is an icc group, $\de(M)$ has trivial relative commutant in $M \ovt M$. Hence, $XX^*$ is a non-zero multiple of $1$ and we may assume that $X$ is a unitary element of $M \ovt M \ovt M$.

We can now start reading the proof of Lemma \ref{lem.intertwine} at formula \eqref{eq.eqX} and find a unitary $w \in M$, a character $\gamma : \Gamma \recht \T$ and an injective group morphism $\delta : \Gamma \recht \Lambda$ such that
$$w u_g w^* = \gamma(g) u_{\delta(g)} \quad\text{for all}\;\; g \in \Gamma \; .$$
But then, $\delta$ follows onto as well.

Conversely assume that $\dist(\T \Gamma,\T \Lambda) = \sqrt{2}$. So we can take a sequence $g_n \in \Gamma$ such that $h_\Lambda(u_{g_n}) \recht 0$. We claim that $h_\Lambda(a u_{g_n} b) \recht 0$ for all $a,b \in M$. The claim is trivial if $a$ and $b$ are finite linear combinations of $v_s, s \in \Lambda$ and follows in general by approximating in $\| \, \cdot \,\|_2$ arbitrary $a,b \in M$ by such finite linear combinations $a_0,b_0$ satisfying $\|a_0\| \leq \|a\|$ and $\|b_0\| \leq \|b\|$. If $w \in M$ would be a unitary satisfying $w u_g w^* \in \T \Lambda$ for all $g \in \Gamma$, we arrive at the contradiction that
$1 = h_\Lambda(w u_{g_n} w^*) \recht 0$.
\end{proof}

Out of Connes' rigidity paper \cite{Co80a} grew a series of rigidity
results ``up to countable classes'' (see e.g. \cite[Theorem
5.3(2)]{Po01a},\cite[Theorem 4.4]{Po01b}, \cite[Theorem 2]{Oz02},
etc). In particular, it was pointed out in \cite[Section 4]{Po06b}
that Connes' rigidity conjecture (\cite{Co80b}) does hold true up to
countable classes. More precisely, given an icc property (T) group
$\Gamma$, there are at most countably many non-isomorphic groups
$\Lambda_i$ satisfying $\rL \Gamma \cong \rL \Lambda_i$. Besides ``separability arguments'', the proof in \cite{Po06b} makes crucial
use of a result in \cite[Theorem, p.\ 5]{Sh00}, which shows that
every property (T) group is the quotient of a finitely presented
property (T) group and thus allows to assume (when arguing by
contradiction) that all $\Lambda_i$ are a quotient of one and the
same property (T) group. As a corollary of Theorem
\ref{thm.triv-2-coc} we can give an alternative proof, not relying
on Shalom's theorem.

\begin{proposition}\label{prop.Connes-up-to-countable}
Let $\Gamma$ be an icc property (T) group. There are at most countably many non-isomorphic groups $\Lambda_i$ satisfying $\rL \Gamma \cong \rL \Lambda_i$.
\end{proposition}
\begin{proof}
Put $M = \rL \Gamma$ with corresponding canonical unitaries $(u_g)_{g \in \Gamma}$. Assume that $(\Lambda_i)_{i \in I}$ is an uncountable family of groups such that $M = \rL \Lambda_i$. Denote by $(u^i_g)_{g \in \Lambda_i}$ the corresponding canonical unitaries. We need to find $i \neq j$ such that $\Lambda_i \cong \Lambda_j$. Note that all $\Lambda_i$ are icc groups. Denote by $\de_i : M \recht M \ovt M$ the comultiplication that corresponds to the group von Neumann algebra decomposition $M = \rL \Lambda_i$.

Since $\Gamma$ has property (T), take a finite subset $K \subset \Gamma$ and $\eps > 0$ such that every unitary representation of $\Gamma$ that admits a $(K,\eps)$-invariant unit vector, actually admits a non-zero invariant vector. Here, given a unitary representation $\pi : \Gamma \recht \cU(\cH)$, a unit vector $\xi$ is called $(K,\eps)$-invariant if $\|\pi(g)\xi - \xi\|\leq \eps$ for all $g \in K$.

Since the Hilbert space $\rL^2(M \ovt M) = \ell^2(\Gamma \times \Gamma)$ is separable, we can take $i \neq j$ such that $\|\de_i(u_g) - \de_j(u_g)\|_2 \leq \eps$ for all $g \in K$. Define the unitary representation $\pi : \Gamma \recht \cU(\ell^2(\Gamma \times \Gamma))$ given by $\pi(g) x = \de_i(u_g) x \de_j(u_g)^*$ for all $x \in M \ovt M$. By construction, the vector $1 \ot 1$ is $(K,\eps)$ invariant. Hence, $\pi$ admits a non-zero invariant vector $\Omega \in \rL^2(M \ovt M)$. So, $\de_i(a) \Omega = \Omega \de_j(a)$ for all $a \in M$. Since $\Lambda_i$ is an icc group, the relative commutant of $\de_i(\rL \Lambda_i)$ inside $\rL(\Lambda_i \times \Lambda_i)$ equals $\C 1$. It follows that $\Omega$ is a non-zero multiple of a unitary element in $M \ovt M$. Hence, we may assume that $\Omega \in \cU(M \ovt M)$.

Since $\de_j = \Ad \Omega^* \circ \de_i$, one deduces, as in the proof of Lemma \ref{lem.intertwine}, that $\Omega \in \rL(\Lambda_i) \ovt \rL(\Lambda_i)$ satisfies the $2$-cocycle and symmetry relation of Theorem \ref{thm.triv-2-coc}. So, by Theorem \ref{thm.triv-2-coc}, we find a unitary $w \in M$ such that $\Omega = \de_i(w^*)(w \ot w)$. Hence, for all $g \in \Lambda_j$,
$$wu^j_gw^*  \ot wu^j_g w^* = (w \ot w)\de_j(u^j_g)(w^* \ot w^*) = \de_i(w u^j_g w^*) \; .$$
So, by Lemma \ref{lem.elemcomult} below, we find for every $g \in \Lambda_j$ an element $\delta(g) \in \Lambda_i$ such that $w u^j_g w^* = u^i_{\delta(g)}$. It follows that $\delta$ is an isomorphism of groups and hence $\Lambda_i \cong \Lambda_j$.
\end{proof}

\section{Support length deformation and intertwining of rigid subalgebras}\label{sec.tensorlength}

Let $\Gamma \actson I$ be an action of a countable group $\Gamma$ on a countable set $I$ and let $(A_0,\tau)$ be a tracial von Neumann algebra. We denote $(A_0^I,\tau) := \bigotimes_{i \in I} (A_0,\tau)$. Put $(A,\tau) = (A_0^I,\tau)$ and $M = A \rtimes \Gamma$.

The following \emph{tensor length deformation} of $M = A_0^I \rtimes \Gamma$ was introduced in \cite{Io06}. For $0 < \rho < 1$, we define
$$\theta_\rho : M \recht M : \theta_\rho(a u_g) = \rho^n a u_g \quad\text{whenever}\;\; g \in \Gamma \;\;,\;\; a \in (A_0 \ominus \C1)^J \;\;\text{and}\;\; J \subset I \;\;,\;\; |J|=n \; .$$
By \cite[Section 2]{Io06} there is an embedding $M \hookrightarrow \Mtil$ and a $1$-parameter group of automorphisms $(\al_t)_{t \in \R}$ of $\Mtil$ such that
\begin{equation}\label{eq.dilation}
E_M(\al_t(x)) = \theta_{\rho_t}(x) \quad\text{for all}\quad x \in M \; .
\end{equation}
We will recall this construction in the proof of Theorem \ref{thm.malleable}. It follows in particular that $\theta_\rho$ is a well defined normal completely positive map on $M$. Also note that $\rho_t \recht 1$ when $t \recht 0$.

The length deformation $\theta_\rho$ is a variant of the malleable deformation that was discovered in \cite{Po03}. Both the length deformation and the malleable deformation allow to prove, under certain conditions, that rigid subalgebras of $M$ can be conjugated into $\rL \Gamma \subset M$. Theorem \ref{thm.malleable} below is an adaptation of \cite[Theorem 4.1]{Po03} and \cite[Theorem 2.1]{Io10}. We first need a technical lemma and some terminology.

Recall that if $Q \subset M$ is a von Neumann subalgebra, we define $\QN_M(Q) \subset M$ consisting of the elements $x \in M$ for which there exist $x_1,\ldots,x_n,y_1,\ldots,y_m$ satisfying
$$x Q \subset \sum_{i=1}^n Q x_i \quad\text{and}\quad Q x \subset \sum_{j=1}^m y_j Q \; .$$
Then, $\QN_M(Q)$ is a $*$-subalgebra of $M$ containing $Q$. Its weak closure is called the \emph{quasi-normalizer of $Q$ inside $M$.} By construction, both $Q$ and $Q' \cap M$ are subalgebras of $\QN_M(Q)$.

If $\Gamma \actson I$ and $\cF \subset I$, we denote by $\Stab \cF$ the subgroup of $\Gamma$ given by $\Stab \cF := \{g \in \Gamma \mid g \cdot i = i \;\text{for all}\; i \in \cF\}$. We also write $\Norm \cF := \{g \in \Gamma \mid g \cdot \cF = \cF \}$. If $\cF$ is finite, $\Stab \cF$ is a finite index subgroup of $\Norm \cF$.

\begin{lemma} \label{lem.control}
Let $\Gamma \actson I$ be an action. Let $A_0 \subset B_0$ and $N$ be tracial von Neumann algebras. Consider $\cM := N \ovt (A_0^I \rtimes \Gamma)$ and $\cMtil = N \ovt (B_0^I \rtimes \Gamma)$. Note that $\cM \subset \cMtil$.
\begin{enumlist}
\item If $P \subset p\cM p$ is a von Neumann subalgebra such that $P \not\embed_{\cM} N \ovt (A_0^I \rtimes \Stab i)$ for all $i \in I$, then the quasi-normalizer of $P$ inside $p \cMtil p$ is contained in $p \cM p$.
\item If $\cF \subset I$ is a finite subset and $Q \subset q(N \ovt A_0^\cF)q$ is a von Neumann subalgebra such that for all proper subsets $\cG \subset \cF$ we have $Q \not\embed_{N \ovt A_0^\cF} N \ovt A_0^\cG$, then the quasi-normalizer of $Q$ inside $q \cM q$ is contained in $q(N \ovt (A \rtimes \Norm \cF))q$.
\item If $\cG \subset I$ is a finite subset and $Q \subset q(N \ovt (A \rtimes \Stab \cG))q$ is a von Neumann subalgebra such that for all strictly larger subsets $\cG \subset \cG'$ we have $Q \not\embed_{N \ovt (A \rtimes \Stab \cG)} N \ovt (A \rtimes \Stab \cG')$, then the quasi-normalizer of $Q$ inside $q \cM q$ is contained in $q(N \ovt (A \rtimes \Norm \cG))q$.
\end{enumlist}
\end{lemma}
\begin{proof}
Analogous to the proof of \cite[Lemma 4.2]{Va07}.
\end{proof}

\begin{theorem} \label{thm.malleable}
Let $\Gamma \actson I$ be an action and $(A_0,\tau)$ a tracial von Neumann algebra.
Assume that $\kappa \in \N$ such that $\Stab J$ is finite whenever $J \subset I$ and $|J| \geq \kappa$. Put $M = A_0^I \rtimes \Gamma$ as above.

Let $(N,\tau)$ be a tracial von Neumann algebra and $Q \subset p(N \ovt M)p$ a von Neumann subalgebra. Denote by $P \subset p(N \ovt M)p$ the quasi-normalizer of $Q$. If for some $0 < \rho < 1$ and $\delta > 0$ we have
\begin{equation}\label{eq.boundlength}
\tau(b^* (\id \ot \theta_\rho)(b)) \geq \delta \quad\text{for all}\quad b \in \cU(Q) \; ,
\end{equation}
then at least one of the following statements is true.
\begin{itemlist}
\item $Q \embed N \ot 1$.
\item $P \embed N \ovt (A \rtimes \Stab i)$ for some $i \in I$.
\item There exists a non-zero partial isometry $v \in p(N \ovt M)$ with $vv^* \in P$ and $v^* P v \subset N \ovt \rL \Gamma$. If $\Gamma$ is icc and $N$ is a factor, we may assume that $vv^* \in \cZ(P)$.
\end{itemlist}
\end{theorem}

\begin{proof}
We recall from \cite[Section 2]{Io06} the following construction. Put $B_0 = A_0 * \rL \Z$, with respect to the natural traces. Denote by $v \in \rL \Z$ the canonical unitary generator and choose a self-adjoint element $h \in \rL \Z$ with spectrum $[-\pi,\pi]$ such that $v = \exp(ih)$. Denote by $\al^0_t \in \Aut(B_0)$ the inner automorphism given by $\al^0_t = \Ad \exp(ith)$. Put $B = B_0^I$ and $\al_t = \bigotimes_{i \in I} \al^0_t$. Since $\al_t$ commutes with the generalized Bernoulli action, we extend $\al_t$ to an automorphism of $\Mtil := B \rtimes \Gamma$ satisfying $\al_t(u_g) = u_g$ for all $g \in \Gamma$. Then \eqref{eq.dilation} above holds with $\rho_t = \bigl| \frac{\sin(\pi t)}{\pi t}\bigr|^2$.

Denote by $\beta_0 \in \Aut(B_0)$ the automorphism given by $\beta_0(a) = a$ for all $a \in A_0$ and $\beta_0(v) = v^*$. Define $\beta = \bigotimes_{i \in I} \beta_0$ and extend $\beta_0$ to $\Mtil$ by acting trivially on $\rL \Gamma$. By construction, $\beta^2 = \id$ and $\beta \circ \al_t \circ \beta = \al_{-t}$. We continue writing $\al_t,\be$ instead of $\id \ot \be$ and $\id \ot \al_t$ on $N \ovt \Mtil$. Write $\cM := N \ovt M$ and $\cMtil := N \ovt \Mtil$. By \eqref{eq.dilation}, for all $x \in \cM$, we have $E_\cM(\al_t(x)) = (\id \ot \theta_{\rho_t})(x)$.

Assume now that $Q$ and $P$ are as in the formulation of the theorem and that \eqref{eq.boundlength} holds. Assume that for all $i \in I$, we have $P \not\embed N \ovt (A \rtimes \Stab i)$. Given von Neumann subalgebras $Q_1,Q_2 \subset \cMtil$, we say that $x \in \cMtil$ is $Q_1$-$Q_2$-finite if there exist $x_1,\ldots,x_n,y_1,\ldots,y_m \in \cMtil$ such that
$$x Q_2 \subset \sum_{i=1}^n Q_1 x_i \quad\text{and}\quad Q_1 x \subset \sum_{j=1}^m y_j Q_2 \; .$$
Note that by definition $\QN_{p\cM p}(Q)$ equals the set of $Q$-$Q$-finite elements in $p\cM p$.

We follow the lines of \cite[Proof of Lemma 5.2]{Va07} to prove the following claim: there exists a non-zero $Q$-$\al_1(Q)$-finite element in $p \cMtil \al_1(p)$. Combining \eqref{eq.boundlength} and \eqref{eq.dilation}, we find an $n \in \N$ such that writing $t = 2^{-n}$, we have $\tau(b^* \al_t(b)) \geq \delta$ for all unitaries $b \in Q$. Define $v \in \cMtil$ as the element of minimal $2$-norm in the $\| \, \cdot \, \|_2$-closed convex hull of $\{ b^* \al_{t}(b) \mid b \in \cU(Q)\}$. Then, $\tau(v) \geq \delta$ and hence, $v \neq 0$. By construction, $v \in p \cMtil \al_t(p)$ and $b v = v \al_{t}(b)$ for all $b \in Q$. Hence, $v$ is $Q$-$\al_t(Q)$-finite.

To conclude the proof of the claim, it suffices to show the following statement: if there exists a non-zero $Q$-$\al_t(Q)$-finite element $v \in p \cMtil \al_t(p)$, then the same is true for $2t$ instead of $t$. For all $d \in \QN_{p \cM p}(Q)$, we have that $\al_t(\be(v^*) d v)$ is a $Q$-$\al_{2t}(Q)$-finite element in $p \cMtil \al_{2t}(p)$. So, we have to prove that there exists a $d \in \QN_{p \cM p}(Q)$ such that $\be(v^*) d v \neq 0$. If this is not the case and if we denote by $q \in p\cMtil p$ the projection onto the closed linear span of all $\{ \Image(dv) \mid d \in \QN_{p \cM p}(Q)\}$, it follows that $q$ and $\be(q)$ are orthogonal. By construction, $q$ commutes with $P$. By Lemma \ref{lem.control}.1, $q \in p \cM p$. Hence, $q = \be(q)$ and it follows that $q = 0$. But then, $v = 0$, a contradiction. Hence, the claim is proven.

Since there is a non-zero $Q$-$\al_1(Q)$-finite element in $\cMtil$, we have in particular that $\al_1(Q) \prec_{\cMtil} \cM$.

For every finite subset $\cF \subset I$, define $\cM(\cF) := N \ovt (A_0^\cF \rtimes \Stab \cF)$. By convention, $\cM(\emptyset) = N \ovt \rL \Gamma$. We now prove that there exists a finite, possibly empty, subset $\cF \subset I$ such that $Q \embed_\cM \cM(\cF)$. Assume the contrary and take a sequence of unitaries $v_n \in Q$ such that
\begin{equation}\label{eq.firsteq}
\|E_{\cM(\cF)}(a v_n b^*)\|_2 \recht 0 \quad\text{for all}\quad a, b \in \cM \quad\text{and all finite subsets}\quad \cF \subset I \; .
\end{equation}
We will deduce from this that
\begin{equation}\label{eq.secondeq}
\|E_{\cM}(x \al_1(v_n) y^*)\|_2 \recht 0 \quad\text{for all}\quad x,y \in \cMtil \; .
\end{equation}
Formula \eqref{eq.secondeq} implies that $\al_1(Q) \not\prec_{\cMtil} \cM$, contradicting the statement $\al_1(Q) \prec_{\cMtil} \cM$ proven above.

We now deduce \eqref{eq.secondeq} from \eqref{eq.firsteq}. Let $\cF \subset I$ be a finite subset and $x_i \in B_0 \ominus A_0 \al_1(A_0)$ for all $i \in \cF$. Put $x_i = 1$ when $i \in I - \cF$ and define $x = 1_N \ot \bigotimes_{i \in I} x_i$. The linear span of all $\cM x (1_N \ot \al_1(A))$ forms a dense $*$-subalgebra of $\cMtil$. So it suffices to prove \eqref{eq.secondeq} for $x,y$ having such a special form: $x$ as above and $y = 1 \ot \bigotimes_{j \in I} y_j$ where $y_j \in B_0 \ominus A_0 \al_1(A_0)$ when $j$ belongs to a finite subset $\cG \subset I$ and $y_j = 1$ when $j \not\in \cG$.

Denote $v_n = \sum_{g \in \Gamma} (v_n)^g (1 \ot u_g)$, where $(v_n)^g \in N \ovt A$, and observe that
$$E_\cM(x \al_1(v_n) y^*) = \sum_{g \in \Gamma} E_{N \ovt A}\bigl( x \al_1((v_n)^g) \si_g(y^*)\bigr) (1 \ot u_g) \; .$$
If $g \cdot \cG \neq \cF$, we have $E_{N \ovt A}\bigl( x \al_1((v_n)^g) \si_g(y^*)\bigr) = 0$. If $g \cdot \cG = \cF$, we have
$$E_{N \ovt A}\bigl( x \al_1((v_n)^g) \si_g(y^*)\bigr) = E_{N \ovt A}\bigl( x \; \al_1\bigl(E_{N \ovt A_0^\cF}((v_n)^g)\bigr) \; \si_g(y^*) \bigr) \; .$$
Take finitely many $g_1,\ldots,g_k \in \Gamma$ such that $g_i \cdot \cG = \cF$ for all $i \in \{1,\ldots,k\}$ and such that $\{g \in \Gamma \mid g \cdot \cG = \cF\}$ is the disjoint union of $(\Stab \cF) g_1,\ldots,(\Stab \cF) g_k$. Put
$$z_n = \sum_{i=1}^k E_{N \ovt (A_0^\cF \rtimes \Stab \cF)}\bigl( v_n (1 \ot u_{g_i}^*)\bigr) (1 \ot u_{g_i}) \; .$$
We have shown that
$$E_{\cM}(x \al_1(v_n) y^*) = E_\cM (x \al_1(z_n) y^*) \; .$$
Since by \eqref{eq.firsteq}, $\|z_n\|_2 \recht 0$, we get \eqref{eq.secondeq}.

So, take a finite subset $\cF \subset I$ such that $Q \embed N \ovt (A_0^\cF \rtimes \Stab \cF)$. We already assumed that for all $i \in I$ we have $P \not\prec N \ovt (A \rtimes \Stab i)$. We now also assume that $Q \not\prec N \ot 1$ and we prove that the third statement of the theorem holds.

Take a larger finite subset $\cG \supset \cF$ such that $Q \embed N \ovt (A_0^\cF \rtimes \Stab \cG)$ where $\cG$ satisfies one of two alternatives: $\Stab \cG$ is finite or $Q \not\embed N \ovt (A_0^\cF \rtimes \Stab \cG')$ whenever $\cG'$ is strictly larger than $\cG$. We claim that the first alternative does not occur. If it would, we get that $Q \embed N \ovt A_0^\cF$. Since $Q \not\embed N \ot 1$, we have $\cF \neq \emptyset$ and we can make $\cF$ smaller, but still non-empty, until for all proper subsets $\cF' \subset \cF$ we have $Q \not\embed N \ovt A_0^{\cF'}$. So we can take projections $p_0 \in Q$, $q \in N \ovt A_0^\cF$, a $*$-homomorphism $\vphi : p_0 Q p_0 \recht q(N \ovt A_0^\cF)q$ and a non-zero partial isometry $v \in p_0 \cM q$ such that $b v = v \vphi(b)$ for all $b \in p_0 Q p_0$ and such that
$$\vphi(p_0 Q p_0) \not\embed_{N \ovt A_0^\cF} N \ovt A_0^{\cF'}$$
whenever $\cF' \subset \cF$ is a proper subset. Lemma \ref{lem.control}.2 implies that $v^* P v \subset N \ovt (A \rtimes \Norm \cF)$ and hence $P \embed N \ovt (A \rtimes \Norm \cF)$. Since $\cF$ is finite and non-empty, $\Stab \cF$ has finite index in $\Norm \cF$ and we reach the contradiction that $P \embed N \ovt (A \rtimes \Stab i)$ for some $i \in I$. This contradiction proves the claim above. We conclude that $Q \embed N \ovt (A_0^\cF \rtimes \Stab \cG)$ and that $Q \not\embed N \ovt (A_0^\cF \rtimes \Stab \cG')$ whenever $\cG'$ is strictly larger than $\cG$.

Take projections $p_0 \in Q$, $q \in N \ovt (A_0^\cF \rtimes \Stab \cG)$, a $*$-homomorphism $\vphi : p_0 Q p_0 \recht q(N \ovt (A_0^\cF \rtimes \Stab \cG))q$ and a non-zero partial isometry $v \in p_0 \cM q$ such that $b v = v \vphi(b)$ for all $b \in p_0 Q p_0$ and such that
\begin{equation}\label{eq.leukeinbedding}
\vphi(Q) \not\embed_{N \ovt (A_0^\cF \rtimes \Stab \cG)} N \ovt (A_0^\cF \rtimes \Stab \cG')
\end{equation}
whenever $\cG'$ is strictly larger than $\cG$.

We claim that $\cG = \emptyset$ (and hence also $\cF = \emptyset$). Assume the contrary. Then \eqref{eq.leukeinbedding} implies in particular that
$$\vphi(Q) \not\embed_{N \ovt (A \rtimes \Stab \cG)} N \ovt (A \rtimes \Stab \cG')$$
whenever $\cG'$ is strictly larger than $\cG$. Then Lemma \ref{lem.control}.3 implies that $v^* P v \subset N \ovt (A \rtimes \Norm \cG)$ and hence $P \embed N \ovt (A \rtimes \Norm \cG)$, which leads as above to the contradiction that $P \embed N \ovt (A \rtimes \Stab i)$ for some $i \in I$. This proves the claim.

By the claim above, $\cF = \cG = \emptyset$. Note that $vv^*$ commutes with $p_0 Q p_0$ and hence belongs to $P$. Also, by \eqref{eq.leukeinbedding} and Lemma \ref{lem.control}.1 we get that $v^* P v \subset N \ovt \rL \Gamma$. Finally, assume that $N$ is a II$_1$ factor and that $\Gamma$ is icc. Take partial isometries $v_1,\ldots,v_n \in P$ with $v_i^* v_i \leq p$ and such that $\sum_{i=1}^n v_i v_i^*$ is a central projection in $P$. Since $N \ovt \rL \Gamma$ is a II$_1$ factor, take partial isometries $w_1,\ldots,w_n \in N \ovt \rL \Gamma$ such that $w_i w_i^* = v^* v_i^* v_i v$ and such that the projections $w_i^* w_i$ are orthogonal. Define $x = \sum_{i=1}^n v_i v w_i$. Then, $x$ is a partial isometry satisfying $xx^* \in \cZ(P)$ and $x^* P x \subset N \ovt \rL \Gamma$.
\end{proof}

Recall from paragraphs \ref{subsec.relT} and \ref{subsec.relamen} the concepts of relative property (T) and relative amenability of von Neumann subalgebras.

\begin{corollary}\label{cor.malleable-rigid}
Let $\Gamma$ be an icc group and $\Gamma \actson I$ an action. Assume that $\kappa \in \N$ such that $\Stab J$ is finite whenever $J \subset I$ and $|J| \geq \kappa$. Assume that $\Stab i$ is amenable for all $i \in I$. Put $A = A_0^I$ and $M = A \rtimes \Gamma$ as above. Let $(N,\tau)$ be a II$_1$ factor and $Q \subset p(N \ovt M)p$ a von Neumann subalgebra satisfying at least one of the following rigidity properties.
\begin{itemlist}
\item $Q \subset p(N \ovt M)p$ has the relative property (T).
\item $Q' \cap p(N \ovt M)p$ is strongly non-amenable relative to $N \ot 1$.
\end{itemlist}
Denote by $P \subset p(N \ovt M)p$ the quasi-normalizer of $Q$ inside $p(N \ovt M)$. Then, at least one of the following statements is true.
\begin{itemlist}
\item $Q \embed N \ot 1$.
\item $P \embed N \ovt (A \rtimes \Stab i)$ for some $i \in I$.
\item There exists $v \in N \ovt M$ with $vv^* = p$ and $v^* P v \subset N \ovt \rL \Gamma$.
\end{itemlist}
\end{corollary}

\begin{proof}
Assume that $Q \not\embed N \ot 1$ and that for all $i \in I$, we have $P \not\embed N \ovt (A \rtimes \Stab i)$. It is sufficient to prove the following statement: for every non-zero central projection $p_0 \in \cZ(P)$, there exists a $0 < \rho < 1$ and a $\delta > 0$ such that
\begin{equation}\label{eq.onestar}
\tau(b^* (\id \ot \theta_\rho)(b)) \geq \delta \quad\text{for all}\quad b \in \cU(Qp_0) \; .
\end{equation}
Indeed, in these circumstances Theorem \ref{thm.malleable} provides a non-zero partial isometry $v$ such that $vv^* \in \cZ(P) p_0$ and $v^* P v \subset N \ovt \rL \Gamma$. Moreover, since $N \ovt \rL \Gamma$ is a II$_1$ factor, we can make sure that $v^* v$ is any projection with the same trace as $vv^*$. As a result, a maximality argument allows to put together several $v$'s and find a partial isometry $v \in N \ovt M$ such that $vv^* = p$ and $v^* P v \subset N \ovt \rL \Gamma$.

Choose a non-zero central projection $p_0 \in \cZ(P)$.

If $Q \subset p(N \ovt M)p$ has the relative property (T), the same is true for $Q p_0 \subset p_0(N \ovt M)p_0$. When $\rho \recht 1$, the completely positive maps $\theta_\rho$ tend pointwise to the identity. The relative property (T) yields the existence of $0 < \rho < 1$ and $\delta > 0$ such that \eqref{eq.onestar} holds for all $b \in \cU(Qp_0)$.

If $Q' \cap p (N \ovt M)p$ is strongly non-amenable relative to $N \ot 1$, the same is true for $(Qp_0)' \cap p_0(N \ovt M)p_0$.

Consider the von Neumann algebra $\Mtil$ as in the proof of Theorem \ref{thm.malleable}. Recall that $M \subset \Mtil = B \rtimes \Gamma$ where $B = B_0^I$ and $B_0 = A_0 * \rL \Z$. Also, $\theta_{\rho_t}(x) = E_M(\al_t(x))$ for all $x \in M$.

As explained in paragraph \ref{subsec.bimod}, we denote by $\weakcont$ the weak containment of bimodules. We claim that
\begin{equation}\label{eq.claimweak}
\bim{M}{\rL^2(\Mtil \ominus M)}{M} \weakcont \bim{(M \ot 1)}{\rL^2(M \ovt M)}{(1 \ot M)} \; .
\end{equation}
In the case of plain Bernoulli actions, this claim has been proven in \cite[Lemma 5]{CI08}. For the convenience of the reader we include a proof in our generalized Bernoulli case, using the amenability of all $\Stab i$, $i \in I$.

Denote by $u$ the canonical unitary generator of $\rL \Z \subset B_0$. Choose a subset $\cA_0 \subset A_0 \ominus \C 1$ such that $\cA_0$ forms an orthonormal basis of $\rL^2(A_0) \ominus \C 1$. Define the subset $\cB_0 \subset B_0$ given by
$$\cB_0 := \{u^{n_1} a_1 u^{n_2} \cdots a_{k-1} u^{n_k} \mid k \geq 1 \; , \; n_1,\ldots,n_k \in \Z - \{0\} \; , \; a_1,\ldots,a_{k-1} \in \cA_0 \} \; .$$
By construction, we have a decomposition
$$\rL^2(B_0) = \rL^2(A_0) \oplus \bigoplus_{b \in \cB_0} \overline{A_0 b A_0}$$
of $\rL^2(B_0)$ into orthogonal $A_0$-$A_0$-subbimodules.

Whenever $\cF \subset I$ is a non-empty finite subset and $(b_i)_{i \in \cF}$ are elements in $\cB_0$, we define the element $b \in B$ as
\begin{equation}\label{eq.formb}
b = \Bigl( \bigotimes_{i \in \cF} b_i \Bigr) \ot \Bigl(\bigotimes_{i \in I - \cF} 1 \Bigr) \; .
\end{equation}
Define the subgroup $S < \Gamma$ given by
$$S := \{g \in \Gamma \mid g \cdot \cF = \cF \;\;\text{and}\;\; b_{g \cdot i} = b_i \;\;\text{for all}\;\; i \in \cF \} \; .$$
Define $M_0 = A_0^{I - \cF} \rtimes S$. One checks that the map $x \ot y \recht x b y$ defines an $M$-$M$-bimodular unitary operator
$$\rL^2(M) \ot_{M_0} \rL^2(M) \recht \overline{M b M} \; .$$
Since $\cF$ is finite, $S \cap \Stab i < S$ has finite index for all $i \in \cF$ and so, $S$ is amenable. It follows that $M_0$ is amenable and hence,
$$\bim{M}{\bigl(\rL^2(M) \ot_{M_0} \rL^2(M)\bigr)}{M} \weakcont \bim{(M \ot 1)}{\rL^2(M \ovt M)}{(1 \ot M)} \; .$$
Since the $\overline{M b M}$, $b$ as above, form an orthogonal decomposition of $\rL^2(\Mtil \ominus M)$ into $M$-$M$-sub\-bi\-modules, the claim \eqref{eq.claimweak} follows.

As in the proof of Theorem \ref{thm.malleable}, denote $\cM := N \ovt M$ and $\cMtil := N \ovt \Mtil$.
By claim \eqref{eq.claimweak} we have
$$\bim{\cM}{\rL^2(\cMtil \ominus \cM)}{\cM} \weakcont \bim{\cM_{12}}{\rL^2(N \ovt M \ovt M)}{\cM_{13}} \; .$$
Write $T := (Qp_0)' \cap p_0 \cM p_0$. Since $T$ is strongly non-amenable relative to $N \ot 1$, it follows that for all non-zero projections $p_1 \in T' \cap p_0 \cM p_0$, the bimodule $\bim{T}{\rL^2(p_1 \cM)}{\cM}$ is not weakly contained in $\bim{T}{\rL^2(p_0(\cMtil \ominus \cM))}{\cM}$. By Lemma \ref{lem.sequence}, we get a finite number of elements $a_1,\ldots,a_n \in T$ and $\eps > 0$ such that
\begin{equation}\label{eq.twostar}
\begin{split}
&\text{if}\;\; x \in p_0 \cMtil p_0 \;\; , \;\; \|x\| \leq 1 \;\;\text{and}\;\; \|a_i x - x a_i \|_2 \leq \eps \;\; \forall i = 1,\ldots,n\;\; , \\ &\text{then}\;\; \|x-E_{\cM}(x)\|_2 \leq \frac{1}{4} \|p_0\|_2 \; .
\end{split}
\end{equation}
Taking $t$ close enough to $0$, we can make $\|a_i - \al_t(a_i)\|_2$ and $\|p_0 - \al_t(p_0)\|_2$ so small that, using the commutation of $Qp_0$ with $a_1,\ldots,a_n$, we get
$$\|a_i \; p_0 \al_t(b) p_0 - p_0 \al_t(b) p_0 \; a_i \|_2 \leq \eps \quad\text{and}\quad
\|\al_t(b) - p_0 \al_t(b) p_0\|_2 \leq \frac{1}{4}\|p_0\|_2 \quad\text{for all}\;\; b \in \cU(Qp_0) \; .$$
Applying \eqref{eq.twostar} to $x = p_0 \al_t(b) p_0$, we conclude that $\|x-E_{\cM}(x)\|_2 \leq \frac{1}{4} \|p_0\|_2$ and hence,
$$\|\al_t(b) - E_{\cM}(\al_t(b)) \|_2 \leq \frac{3}{4}\|p_0\|_2 \quad\text{for all}\;\; b \in \cU(Qp_0) \; .$$
Put $\rho = \rho_t^2$. For all $b \in \cU(Qp_0)$, we get
$$\tau(p_0) - \tau(b^* (\id \ot \theta_\rho)(b)) = \tau(p_0) - \|(\id \ot \theta_{\rho_t})(b)\|_2^2 = \|\al_t(b) - E_{\cM}(\al_t(b))\|_2^2 \leq \frac{9}{16} \tau(p_0) \; .$$
Hence, \eqref{eq.onestar} holds with $\delta = \frac{7}{16}\tau(p_0)$.
\end{proof}

\begin{remark}
We make the following two observations about Corollary \ref{cor.malleable-rigid}, but we do not use them in the rest of the paper.
In the situation where $Q \subset p (N \ovt M)p$ has the relative property (T), Corollary \ref{cor.malleable-rigid} can be strengthened in two ways. First of all, the same conclusion holds without the assumption that $\Stab i$ is amenable for all $i \in I$. In the relative property (T) part of the proof, we did not use the amenability of $\Stab i$.
Secondly, if we assume that $\Stab i$ is amenable and that $Q \subset P$ has the relative property (T), then it is easy to see that the option $P \embed N \ovt (A \rtimes \Stab i)$ actually implies that $Q \embed N \ot 1$.
\end{remark}

\section{Clustering sequences techniques and intertwining of abelian subalgebras}\label{sec.intertwine-abelian}

Throughout this section assume that $\Gamma \actson I$ is an action of the countable group $\Gamma$ on the countable set $I$. Assume that $\kappa > 0$ such that the stabilizer $\Stab \cF$ is finite whenever $\cF \subset I$ is a subset with $|\cF| \geq \kappa$. Let $(X_0,\mu_0)$ be a non-trivial standard probability space and put $A := \rL^\infty(X_0^I)$, together with the action $\Gamma \actson A$ given by the generalized Bernoulli shift. Define $M = A \rtimes \Gamma$.

We prove a strong structural result for abelian von Neumann subalgebras $D \subset (M \ovt M)^t$ that are normalized by many unitaries in $(\rL \Gamma \ovt \rL \Gamma)^t$. Later we shall apply this structural result to $D = \Delta(A)$ whenever $\Delta : M \recht (M \ovt M)^t$ is (the amplification of) the comultiplication given by another group von Neumann algebra or group measure space decomposition of $M$. This structural result and its proof are very similar to \cite[Theorem 6.1]{Io10}. We give however all the details because the generalization from plain Bernoulli to generalized Bernoulli actions is not totally innocent.
Both here and in \cite{Io10} the technique is very much inspired by the \emph{clustering sequences techniques} from \cite[Sections 1-4]{Po04}. For a more gentle introduction to these matters, we refer to the lecture notes \cite{Va11}.

\begin{theorem}\label{thm.intertwine-abelian}
As above let $\Gamma \actson I$ be such that $\Stab \cF$ is finite whenever $\cF \subset I$ and $|\cF| \geq \kappa$. Put $A := \rL^\infty(X_0^I)$ and $M = A \rtimes \Gamma$.

Assume that $t > 0$ and that $D \subset (M \ovt M)^t$ is an abelian von Neumann subalgebra that is normalized by a group of unitaries $(\gamma(s))_{s \in \Lambda}$ that belong to $(\rL \Gamma \ovt \rL \Gamma)^t$. Denote by $P \subset (M \ovt M)^t$ the quasi-normalizer of $D$ inside $(M \ovt M)^t$. Make the following assumptions.

\begin{enumerate}
\item\label{assum1} $D \not\prec M \ot 1$ and $D \not\prec 1 \ot M$.
\item\label{assum2} For all $i \in I$ we have $P \not\prec M \ovt (A \rtimes \Stab i)$ and $P \not\prec (A \rtimes \Stab i) \ovt M$.
\item\label{assum3} $P \not\prec M \ovt \rL \Gamma$ and $P \not\prec \rL \Gamma \ovt M$.
\item\label{assum4} For all $i \in I$ we have $\gamma(\Lambda)\dpr \not\prec \rL(\Gamma) \ovt \rL(\Stab i)$ and $\gamma(\Lambda)\dpr \not\prec \rL(\Stab i) \ovt \rL(\Gamma)$.
\end{enumerate}

Denote $C := D' \cap (M \ovt M)^t$. Then for every non-zero projection $q \in \cZ(C)$ we have that $Cq \prec A \ovt A$.
\end{theorem}

\begin{remark}
To avoid unnecessary notational complexity we did not formulate the obvious more general result for subalgebras of $(M_1 \ovt M_2)^t$ where $M_i = \rL^\infty(X_i^{I_i}) \rtimes \Gamma_i$ and where both $\Gamma_i \actson I_i$ satisfy the finiteness assumption on the stabilizer groups. Also there is an obvious version of the theorem for subalgebras $D \subset M^t$ that are normalized by unitaries $\gamma(s) \in \rL(\Gamma)^t$.
\end{remark}

\begin{proof}
Note that because $D$ is abelian, we have $\cZ(C) = C' \cap (M \ovt M)^t$.

The main part of the proof consists in showing that for every non-zero projection $q \in \cZ(C)$ we have that $C q \prec M \ovt A$. At the end we then deduce that actually $C q \prec A \ovt A$ for every non-zero projection $q \in \cZ(C)$. Consider
\begin{align*}
\cP := \{q_1 \in \cZ(C) \mid \; & q_1 \;\text{is a projection and for every non-zero projection} \\ &\text{$q \in \cZ(C) q_1$ we have that}\; C q \prec M \ovt A \} \; .
\end{align*}
One easily checks that $\cP$ admits a maximum $q_2$ and that this maximum commutes with the normalizer of $C$, in particular with the unitaries $(\gamma(s))_{s \in \Lambda}$ (see \cite[Proposition 2.5]{Va10} for details). We have to prove that $q_2 = 1$. If not, we can replace $D$ by $D(1-q_2)$ and $\gamma(s)$ by $\gamma(s) (1-q_2)$. So in the end, we only need to prove that $\cP$ is non-empty. This means that we have to prove that $C \prec M \ovt A$.

We split the proof of the statement $C \prec M \ovt A$ into several steps. We use the following notation.

We use the letter $Q$ to denote all kind of orthogonal projections related to the infinite tensor product $A = A_0^I$ and the letter $P$ to denote all kind of orthogonal projections related to the group $\Gamma$. All these projections $Q$ and $P$ project onto subspaces of the form $\rL^2(M) \ot \cK$ and they all commute.
\begin{itemlist}
\item For every subset $\cF \subset I$, we denote by $Q_\cF$ the orthogonal projection onto the closed linear span of $\{ M \ot A_0^\cF u_g \mid g \in \Gamma\}$.
\item For every $\ell \in \N$, we denote by $Q^{\sgeq \ell}$ the orthogonal projection onto the closed linear span of $\{ M \ot (A_0 \ominus \C 1)^\cF u_g \mid \cF \subset I, \ell \leq |\cF| < \infty, g \in \Gamma \}$.
\item For every subset $S \subset \Gamma$, denote by $P_S$ the orthogonal projection onto the closed linear span of $\{ M \ot A u_g \mid g \in S\}$.
\end{itemlist}
We denote by $Q^{\sgeq \ell}_\cF$ the product of $Q^{\sgeq \ell}$ and $Q_\cF$.

In general, the projection $Q_\cF$ does not behave well with respect to the operator norm $\| \, \cdot \, \|$. Because of the formula
$$P_S(Q_\cF(x)) = \sum_{g \in S} E_{M \ovt A_0^\cF}(x (1 \ot u_g)^*) \, (1 \ot u_g) \; ,$$
we do get $\|P_S(Q_\cF(x))\| \leq |S| \, \|x\|$ and $\|P_S(x)\| \leq |S| \, \|x\|$ for all $x \in M \ovt M$ and all subsets $\cF \subset I$.

In a few cases, we use the same notation $Q_\cF, Q^{\sgeq \ell}, P_S$ to denote projections of $\rL^2(M)$ onto the corresponding obvious subspaces.

To avoid a too heavy notation, we assume that $t \leq 1$. So we have a projection $p \in \rL (\Gamma \times \Gamma)$ such that $D \subset p (M \ovt M) p$ and $\gamma(s) \in p \rL (\Gamma \times \Gamma) p$. This simplification does not hide any essential part of the argument.

\begin{step}\label{stepa}
For every $\eps > 0$ and every $\ell \in \N$, there exists a unitary $a \in D$ such that $$\|a - (Q^{\sgeq \ell} \ot Q^{\sgeq \ell})(a)\|_2 < \eps \; .$$
\end{step}
\begin{proof}
Denote by $\si : M \ovt M \recht M \ovt M$ the flip automorphism $\si(a \ot b) = b \ot a$. Consider the projection
$$\ptil := \begin{pmatrix} p & 0 \\ 0 & \si(p) \end{pmatrix} \in \M_2(\C) \ovt M \ovt M \; .$$
Define the von Neumann subalgebra $\Dtil \subset \ptil(\M_2(\C) \ovt M \ovt M) \ptil$ given by
$$\Dtil := \Bigl\{ \begin{pmatrix} a & 0 \\ 0 & \si(a) \end{pmatrix} \; \Big| \; a \in D \; \Bigr\} \; .$$
Denote by $\Ptil$ the quasi-normalizer of $\Dtil$ inside $\ptil (\M_2(\C) \ovt M \ovt M) \ptil$. By assumption (\ref{assum1}) we have $\Dtil \not\prec M \ot 1$. By assumption (\ref{assum2}) we have for all $i \in I$ that $\Ptil \not\prec M \ovt (A \rtimes \Stab i)$. By assumption (\ref{assum3}) we have $\Ptil \not\prec M \ovt \rL \Gamma$. We now apply Theorem \ref{thm.malleable}. We conclude that \eqref{eq.boundlength} in Theorem \ref{thm.malleable} cannot hold. So, given $\eps > 0$ and $\ell \in \N$, we find a unitary $b \in \Dtil$ such that $\|d - (1 \ot Q^{\sgeq \ell})(d)\|_2 < \eps / 2$. Writing
$$d = \begin{pmatrix} a & 0 \\ 0 & \si(a) \end{pmatrix}$$
we have found a unitary $a \in D$ such that $\|a - (1 \ot Q^{\sgeq \ell})(a)\|_2 < \eps / 2$ and $\|a - (Q^{\sgeq \ell} \ot 1)(a)\|_2 < \eps / 2$. Hence also
$\|a - (Q^{\sgeq \ell} \ot Q^{\sgeq \ell})(a)\|_2 < \eps$.
\end{proof}

\begin{step}\label{stepb}
There is a sequence of group elements $g_n \in \Lambda$ such that for all $i \in I$ and $g,h \in \Gamma \times \Gamma$, we have
\begin{equation}\label{eq.propgn}
\|E_{\rL (\Gamma \times \Stab i)}(u_g \gamma(g_n) u_h)\|_2 \recht 0 \quad\text{and}\quad  \|E_{\rL(\Stab i \times \Gamma)}(u_g \gamma(g_n) u_h)\|_2 \recht 0\; .
\end{equation}
\end{step}
\begin{proof}
This follows immediately from assumption (\ref{assum4}), Definition \ref{def.intertwine} and Remark \ref{rem.intertwine}.
\end{proof}

From now on, we fix a sequence $(g_n)$ in $\Lambda$ satisfying \eqref{eq.propgn}. We put $v_n := \gamma(g_n)$.

\begin{step}\label{stepc}
For all $x \in M \ovt M$ and all finite subsets $\cF \subset I$, we have
\begin{equation}\label{eq.vn}
\| v_n x v_n^* - Q_{I - \cF}(v_n x v_n^*)\|_2 \recht 0 \; .
\end{equation}
\end{step}
\begin{proof}
It suffices to check \eqref{eq.vn} when $x$ is of the form $x = x_0 \ot a u_g$ with $x_0 \in M$, $g \in \Gamma$ and $a \in A_0^\cG$ for some finite subset $\cG \subset I$. Fix a finite subset $\cF \subset I$. Define
$$K := \{k \in \Gamma \mid k \cG \cap \cF = \emptyset \} \; .$$
Define $w_n = P_K(v_n)$. Then, $w_n \in \rL^2(M \ovt M)$ and by \eqref{eq.propgn}, $\|v_n -w_n\|_2 \recht 0$. Hence, $\| v_n x v_n^* - w_n x v_n^*\|_2 \recht 0$. Since by construction $w_n x v_n^*$ lies in the image of $Q_{I - \cF}$, the formula \eqref{eq.vn} follows.
\end{proof}

\begin{step} \label{stepd}
For all $a \in D$ and all $\eps > 0$, there exists a finite subset $S \subset \Gamma$ such that
$$\|v_n a v_n^* - (P_S \ot P_S)(v_n a v_n^*)\|_2 \leq \eps \quad\text{for all}\;\; n \; .$$
\end{step}
\begin{proof}
Choose a unitary $a \in \cU(D)$ and put $a_n := v_n a v_n^*$. Since the projections $P_S \ot 1$ and $1 \ot P_S$ commute, by symmetry it suffices to prove that for all $\eps>0$, there exists a finite subset $S \subset \Gamma$ such that $\|(1 \ot P_S)(a_n)\|_2 \geq \|p\|_2 - 4\eps$ for all $n$ large enough.

Write $\delta = \eps \|p\|_2$. By step \ref{stepa} take a unitary $b \in \cU(D)$ such that $\|b - Q^{\sgeq \kappa}(b)\|_2 \leq \delta$. By the Kaplansky density theorem, take a finite subset $\cG \subset I$ and an element
$$b_0 \in \lspan \{ x_0 \ot x_1 u_g \mid x_0 \in M, x_1 \in A_0^\cG, g \in \Gamma \}$$
such that $\|b_0\| \leq 1$, $\|b - b_0\|_2 \leq \delta$ and $\|b_0\|_2 \leq \|b\|_2 = \|p\|_2$. Put $\eta = Q^{\sgeq \kappa}(b_0)$ and observe that $\|\eta\|_2 \leq \|b_0\|_2 \leq \|p\|_2$, that $\|b - \eta \|_2 \leq 2 \delta$ and that
$$\eta \in \lspan \{ y_0 \ot y_1 u_h \mid y_0 \in M, y_1 \in (A_0 \ominus \C 1)^J, J \subset \cG, |J| \geq \kappa, h \in \Gamma \} \; .$$
Since $a_n$ and $b$ are commuting unitaries in $p(M \ovt M)p$, we have $\langle a_n \, b , b \, a_n \rangle = \tau(p)$ and hence,
\begin{equation}\label{eq.myeq}
\big| \tau(p) - \langle a_n \, b_0 , \eta \, a_n \rangle \big| \leq 3 \delta
\end{equation}
for all $n$. Put $S := \{g \in \Gamma \mid \; |g \cdot \cG \cap \cG | \geq \kappa \}$. By our assumption on the action $\Gamma \actson I$, the set $S$ is finite.

{\bf Claim.} We have that $\langle P_{\Gamma - S}(a_n) \, b_0 , \eta \, a_n \rangle \recht 0$. Given the special form of $b_0$ and $\eta$, it suffices to prove the claim for $b_0 = x_0 \ot x_1 u_g$ and $\eta = y_0 \ot y_1 u_h$ where $x_0,y_0 \in M$, $x_1 \in A_0^\cG$, $y_1 \in (A_0 \ominus \C1)^J$, $J \subset \cG$, $|J| \geq \kappa$ and $g,h \in \Gamma$.

Put $d_n := Q_{I - (\cG \cup h^{-1} \cG)}(a_n)$. By construction, $\eta d_n$ lies in the closed linear span of
$$M \ot (A_0 \ominus \C 1)^J A_0^{I - \cG} u_k \;\; , k \in \Gamma \; .$$
On the other hand, $P_{\Gamma - S}(d_n) b_0$ lies in the closed linear span of
$$M \ot A_0^{r\cG \cup (I- \cG)} u_k \;\; , r \in \Gamma - S, k \in \Gamma \; .$$
Since $|r \cG \cap J|< \kappa$ for all $r \in \Gamma - S$, the two subspaces are orthogonal. Hence, $\langle P_{\Gamma - S}(d_n)\, b_0 , \eta d_n \rangle = 0$ for all $n$. By step \ref{stepc}, $\|a_n - d_n\|_2 \recht 0$. Hence, the claim follows.

Combining the claim with \eqref{eq.myeq}, we can take $n_0$ such that
$$\big| \tau(p) - \langle P_S(a_n) \, b_0 , \eta \, a_n \rangle \big| \leq 4 \delta$$
for all $n \geq n_0$. It follows that
\begin{equation*}
\tau(p) - 4 \delta \leq |\langle P_S(a_n) \, b_0 , \eta \, a_n \rangle | \leq \|P_S(a_n)\|_2 \, \|b_0\| \, \|\eta\|_2 \, \|a_n\| \leq \|P_S(a_n)\|_2 \, \|p\|_2 \; .
\end{equation*}
Since $\tau(p) - 4 \delta = \|p\|_2 (\|p\|_2 - 4 \eps)$, we have shown that $\|P_S(a_n)\|_2 \geq \|p\|_2 - 4 \eps$ for all $n \geq n_0$.
\end{proof}

Recall from \eqref{eq.defheight} the notion of the height of an element in a group von Neumann algebra. We now use this notion in the group von Neumann algebra $\rL(\Gamma \times \Gamma)$. So, for all $v \in \rL(\Gamma \times \Gamma)$, we consider
$$h(v) = \max\{ |\tau(v u_g^*)| \mid g \in \Gamma \times \Gamma\} \; .$$

\begin{step}\label{stepe}
There exists a $\delta > 0$ such that $h(v_n) \geq \delta$ for all $n$.
\end{step}
\begin{proof}
If the assertion does not hold, we can pass to a subsequence and assume that $h(v_n) \recht 0$.

{\bf Claim.} Take $J_1,J_2 \subset I$ with $|J_i| \geq \kappa$. For all $a \in (A_0 \ominus \C1)^{J_1} \ot (A_0 \ominus \C1)^{J_2}$ and for all sequences $w_n$ in the unit ball of $\rL(\Gamma \times \Gamma)$, we have
$$\|E_{A \ovt A}(v_n a w_n^*)\|_2 \recht 0 \; .$$
To prove the claim, denote by $(v)_g$, $g \in \Gamma \times \Gamma$, the Fourier coefficients of an element $v \in \rL(\Gamma \times \Gamma)$. So, by definition and with $\| \, \cdot \, \|_2$-convergence, we have
$$v = \sum_{g \in \Gamma \times \Gamma} (v)_g u_g \; .$$
Take finite sets $\cF_i \subset \Gamma$ such that for all $g \in \Gamma - \cF_i$, we have $|g \cdot J_i \cap J_i| < \kappa$. Put $\cF = \cF_1 \times \cF_2$. So, whenever $g \in (\Gamma \times \Gamma) - \cF$, we have $a \perp \sigma_g(a)$. As a result, we get
\begin{align*}
\|E_{A \ovt A}(v_n a w_n^*)\|_2^2 &= \sum_{k \in \cF} \sum_{g \in \Gamma \times \Gamma} (v_n)_g \; \overline{(v_n)_{gk}} \; \overline{(w_n)_g} \; (w_n)_{gk} \; \tau(a \sigma_k(a^*)) \\
& \leq \|a\|_2^2 \; h(v_n)^2 \; \sum_{k \in \cF} \sum_{g \in \Gamma \times \Gamma} |(w_n)_g| \; |(w_n)_{gk}| \\
& \leq \|a\|_2^2 \; h(v_n)^2 \; \sum_{k \in \cF} \left( \Bigl( \sum_{g \in \Gamma \times \Gamma} |(w_n)_g|^2 \Bigr)^{1/2} \; \Bigl( \sum_{g \in \Gamma \times \Gamma} |(w_n)_{gk}|^2 \Bigr)^{1/2}\right) \\
& \leq \|a\|_2^2 \; |\cF| \; h(v_n)^2 \recht 0 \; .
\end{align*}
This proves the claim. Applying the claim to $w_n$ of the form $w_n = u_g v_n u_h^*$, we get the following:
for all $\eta \in \rL^2(M \ovt M)$ satisfying $\eta = (Q^{\sgeq \kappa} \ot Q^{\sgeq \kappa})(\eta)$ and for all finite subsets $S \subset \Gamma$, we have
\begin{equation}\label{eq.aaa}
\|(P_S \ot P_S)(v_n \eta v_n^*) \|_2 \recht 0 \; .
\end{equation}
By step \ref{stepa} take a unitary $a \in \cU(D)$ such that $\|a - (Q^{\sgeq \kappa} \ot Q^{\sgeq \kappa})(a)\|_2 \leq \|p\|_2/2$. Formula \eqref{eq.aaa} implies that for all $S \subset \Gamma$ finite, we have
$$\limsup_n \|(P_S \ot P_S)(v_n a v_n^*)\|_2 \leq \|p\|_2/2 \; .$$
This is a contradiction with step \ref{stepd}.
\end{proof}

\begin{step}\label{stepf}
Take $\delta > 0$ such that $h(v_n) \geq \sqrt{6\delta}$ for all $n$.

For every $\eps > 0$, there exists a unitary $a \in \cU(D)$, finite subsets $S \subset \Gamma$, $\cF \subset I$ and a sequence $h_n \in \Gamma$ such that, writing $x_n = v_n a v_n^*$, we have for all $n$,
\begin{itemlist}
\item $\|x_n - P_S(x_n)\|_2 \leq \eps$,
\item $\|x_n - Q^{\sgeq \kappa}(x_n)\|_2 \leq \eps$,
\item $\|x_n - Q_{h_n \cdot \cF}(x_n)\|_2 \leq \|p\|_2 - 2 \delta$,
\item $\dis \|x_n - Q_{I - \cG}(x_n)\|_2 \recht 0$ for every finite subset $\cG \subset I$.
\end{itemlist}
\end{step}

\begin{proof}
Choose $\eps > 0$. By step \ref{stepa} take $a \in \cU(D)$ such that $\|a - Q^{\sgeq \kappa}(a)\|_2 \leq \eps$. Put $x_n = v_n a v_n^*$. Since the image of $Q^{\sgeq \kappa}$ is an $(M \ovt \rL(\Gamma))$-$(M \ovt \rL(\Gamma))$-bimodule, we have
$$\|x_n - Q^{\sgeq \kappa}(x_n)\|_2 = \|a - Q^{\sgeq \kappa}(a)\|_2 \leq \eps$$
for all $n$. By step \ref{stepd}, take a finite subset $S \subset \Gamma$ such that $\|x_n - P_S(x_n)\|_2 \leq \eps$ for all $n$. By step \ref{stepc}, we have $\|x_n - Q_{I - \cG}(x_n)\|_2 \recht 0$ for every finite subset $\cG \subset I$.

Take a finite subset $\cF \subset I$ such that $\|a - Q_\cF(a)\|_2 \leq \delta$. Choose elements $k_n \in \Gamma \times \Gamma$ such that $|\tau(v_n u_{k_n}^*)| \geq \sqrt{6\delta}$ for all $n$. Denote by $h_n \in \Gamma$ the second component of $k_n$.

Denote $w_n = \tau(v_n u_{k_n}^*) u_{k_n}$ and $y_n = w_n a v_n^*$. It follows that
$$\|x_n - Q_{h_n \cdot \cF}(x_n)\|_2 = \|(1-Q_{h_n \cdot \cF})(x_n - y_n) + (1-Q_{h_n \cdot \cF})(y_n)\|_2 \leq \|x_n - y_n\|_2 + \|y_n - Q_{h_n \cdot \cF}(y_n)\|_2 \; .$$
We consecutively get
$$\|x_n - y_n \|_2 \leq \|v_n - w_n\|_2 \leq \sqrt{\|p\|_2^2 - 6 \delta} \leq \|p\|_2 - 3 \delta$$
and
$$\|y_n - Q_{h_n \cdot \cF}(y_n)\|_2 = \| w_n (a - Q_{\cF}(a)) v_n^* \|_2 \leq \|a - Q_\cF(a)\|_2 \leq \delta \; .$$
Altogether we have $\|x_n - Q_{h_n \cdot \cF}(x_n)\|_2 \leq \|p\|_2 - 2\delta$.
\end{proof}

We finally gathered enough results to prove that $C \prec M \ovt A$.

\begin{step}
We have that $C \prec M \ovt A$.
\end{step}

\begin{proof}
Assume that $C \not\embed M \ovt A$. Note that $M \ovt M = (M \ovt A) \rtimes \Gamma$, where $\Gamma$ acts trivially on $M$. By \cite[Theorem 1.3.2]{Io10}, for every $\eps > 0$ and every $k \in \N$, there exists a unitary $d \in \cU(C)$ such that $\|P_\cG(d)\|_2 < \eps$ for all subsets $\cG \subset \Gamma$ with $|\cG| \leq k$.

Take $a \in \cU(D)$, finite subsets $S \subset \Gamma$, $\cF \subset I$ and a sequence $h_n \in \Gamma$ satisfying the conclusion of step \ref{stepf} with $\eps \leq \delta/8$. Whenever $Z \subset \Gamma$ is finite, we define the orthogonal projection
$$R_Z = \bigvee_{g \in Z} Q_{g \cdot \cF} \; .$$

{\bf Claim.} Whenever $Z_n$ is a sequence of finite subsets of $\Gamma$ such that $\sup_n |Z_n| < \infty$, there exists a sequence of larger finite subsets $Z_n' \supset Z_n$ such that $\sup_n |Z_n'| < \infty$ and
\begin{equation}\label{eq.goal}
\liminf_n \| R_{Z_n'}(x_n) - R_{Z_n}(x_n)\|_2 \geq \delta \; .
\end{equation}

Once the claim is proven, we inductively construct $Z_n^1 \subset Z_n^2 \subset \cdots$. Since the vectors $R_{Z_n^{k+1}}(x_n) - R_{Z_n^k}(x_n)$ are orthogonal for different $k$, we arrive at the contradiction
$$\|p\|_2 = \liminf_n \|x_n\|_2^2 \geq k \delta^2 \quad\text{for all}\;\; k \in \N \; .$$

We now prove the claim. Let the sequence $Z_n$ be given.
For every $n$, denote $$L_n := \{g \in \Gamma \mid \exists k \in Z_n \;\;\text{such that}\;\; |g h_n \cF \cap k \cF| \geq \kappa\} \; .$$
Since $\Stab J$ is finite whenever $|J| \geq \kappa$, it follows that $\sup_n |L_n| < \infty$. So, we can take a unitary $d \in \cU(C)$ such that $\|P_{L_n}(d)\|_2 \leq \eps/(2|S|)$ for every $n$. Take a finite set $S'\subset \Gamma$ such that $\|d - P_{S'}(d)\|_2 \leq \eps/(2|S|)$. Put $K_n = S'-L_n$. We retain that $\|d - P_{K_n}(d)\|_2 \leq \eps/|S|$ for all $n$ and that $|g h_n \cF \cap k \cF| < \kappa$ for all $g \in K_n$ and all $k \in Z_n$. Put $Z_n' = K_n h_n \cup Z_n$. We prove that $Z_n'$ satisfies \eqref{eq.goal}.

Using the Kaplansky density theorem, take a finite subset $\cG \subset I$ and $d_0 \in M \ovt M$ such that $d_0 = Q_{\cG}(d_0)$, $\|d_0\|\leq 1$ and $\|d-d_0\|_2 \leq \eps/|S|$. Write $d_n := P_{K_n}(d_0)$. Hence, $\|d-d_n\|_2 \leq 2\eps/|S|$. Also write $x_n' := Q_{h_n \cdot \cF}(P_S(x_n)) = P_S(Q_{h_n \cdot \cF}(x_n))$. Note that $\|x'_n\| \leq |S|$ and $\|x_n - x'_n\|_2 \leq \|p\|_2-2\delta+\eps$ for all $n$. As a result,
$$\|dx_n - d_n x_n'\|_2 \leq \|d\| \, \|x_n - x_n'\|_2 + \|x_n'\| \, \|d-d_n\|_2 \leq \|p\|_2 -2\delta + 3 \eps \; .$$
Define the orthogonal projection
$$R_n := \bigvee_{g \in K_n} Q_{\cG \cup gh_n \cF} \; .$$
Since $d_n x_n'$ lies in the image of $R_n$, it follows that $\|(1-R_n)(d x_n)\|_2 \leq \|p\|_2-2\delta + 3\eps$. But $dx_n = x_n d$. Hence, $\|R_n(x_n d)\|_2 \geq 2\delta - 3\eps$.

Observe that
$$\|x_n d - P_S(x_n) d_0 \|_2 \leq \|x_n - P_S(x_n)\|_2\, \|d\| + \|P_S(x_n)\| \, \|d - d_0\|_2 \leq 2 \eps \; .$$
So, $\|R_n(P_S(x_n) d_0)\|_2 \geq 2 \delta - 5 \eps$. Write
$$R'_n := \bigvee_{g \in K_n} Q_{\cG \cup S \cG \cup g h_n \cF} \; .$$
Since $R_n \leq R_n'$, we have $\|R'_n(P_S(x_n) d_0)\|_2 \geq 2 \delta - 5 \eps$. But, $R'_n(P_S(x_n) d_0) = R'_n(P_S(x_n)) d_0$ and $\|d_0\| \leq 1$. It follows that
$$\|R'_n(x_n)\|_2 \geq \|P_S(R'_n(x_n))\|_2 = \|R'_n(P_S(x_n))\|_2 \geq \|R'_n(P_S(x_n)) d_0\|_2 = \|R'_n(P_S(x_n) d_0)\|_2 \geq 2 \delta - 5 \eps \; .$$
Since $\|x_n - Q^{\sgeq \kappa}(x_n)\|_2 \leq \eps$ and since $\|x_n - Q_{I - (\cG \cup S \cG)}(x_n)\|_2 \recht 0$, we can take $n_0$ such that
$$\|R\dpr_n(x_n)\|_2 \geq 2\delta - 7 \eps \quad\text{for all}\;\; n \geq n_0 \quad\text{where}\;\; R\dpr_n := \bigvee_{g \in K_n} Q_{g h_n \cF}^{\sgeq \kappa} \; .$$
Whenever $g \in K_n$ and $k \in Z_n$, we have $|g h_n \cF_n \cap k \cF| < \kappa$. So, the projections $Q^{\sgeq \kappa}_{g h_n \cF}$ and $Q_{k \cF}$ have orthogonal ranges. Hence, $R\dpr_n$ and $R_{Z_n}$ are orthogonal as well. By construction $R_n\dpr \leq R_n'$. It follows that
$$\|R_{Z_n'}(x_n) - R_{Z_n}(x_n)\|_2 \geq \|R\dpr_n(R_{Z_n'}(x_n) - R_{Z_n}(x_n))\|_2 = \|R\dpr_n(x_n)\|_2 \geq 2\delta - 7 \eps \geq \delta$$
for all $n \geq n_0$. So, we have proven \eqref{eq.goal}.
\end{proof}

\begin{step}
End of the proof of Theorem \ref{thm.intertwine-abelian}.
\end{step}

We have shown that $C q \prec M \ovt A$ for every non-zero projection $q \in \cZ(C) = C' \cap p (M \ovt M) p$. This means that the following holds (see \cite[Lemma 2.4 and Proposition 2.5]{Va10} for details): for every $\eps > 0$ there exists a finite set $S \subset \Gamma$ such that $\|d-(1 \ot P_S)(d)\|_2 \leq \eps / 2$ for every unitary $d \in \cU(C)$. By symmetry, we also find a finite set $S' \subset \Gamma$ such that $\|d - (P_{S'} \ot 1)(d)\|_2 \leq \eps /2$ for all $d \in \cU(C)$. Taking the union of $S$ and $S'$, we have found a finite set $S \subset \Gamma$ such that $\|d - (P_S \ot P_S)(d)\|_2 \leq \eps$ for all $d \in \cU(C)$. This means that $Cq \prec A \ovt A$ for all non-zero projections $q \in \cZ(C)$.
\end{proof}

\section{A conjugacy criterion for group actions}\label{sec.conjugacy-actions}

Suppose that we are given an embedding of group measure space factors $B \rtimes \Lambda \hookrightarrow A \rtimes \Gamma$ such that $B = A$ and such that $v \rL \Lambda v^* \subset \rL \Gamma$ for some unitary $v \in A \rtimes \Gamma$. Under the right conditions, one can deduce from this information the existence of a unitary $w \in A \rtimes \Gamma$ such that $w B w^* = A$ and $w v_s w^* = \omega(s) u_{\delta(s)}$ for all $s \in \Lambda$, where $\delta : \Lambda \recht \Gamma$ is a group morphism and $\omega : \Lambda \recht \T$ is a character. Such a result was first proven in \cite[Theorem 5.2]{Po04} and generalized in \cite[Theorem 7.1]{Io10}. We now prove a further generalization, involving arbitrary amplifications and weaker assumptions. We give a more elementary proof in the spirit of \cite[Proposition 9.3]{Va06}.

\begin{theorem}\label{thm.conjugacy-actions}
Let $\Gamma \actson (X,\mu)$ be a free ergodic p.m.p.\ action. Put $A = \rL^\infty(X)$ and $M = A \rtimes \Gamma$. Let $p \in \M_n(\C) \ot \rL \Gamma$ be a projection. Assume that $C \subset p(\M_n(\C) \ot M)p$ is a von Neumann subalgebra and $\gamma : \Lambda \recht \cU(p(\M_n(\C) \ot \rL \Gamma)p)$ a group morphism such that the following conditions hold.
\begin{enumlist}
\item $C \embed A$ and $C' \cap p(\M_n(\C) \ot M)p = \cZ(C)$.
\item The unitaries $\gamma(s)$ normalize $C$ and the action $(\Ad \gamma(s))_{s \in \Lambda}$ on $\cZ(C)$ is weakly mixing.
\end{enumlist}
Then there exist
\begin{itemlist}
\item a subgroup $\Gamma_1 < \Gamma$, a finite normal subgroup $K \lhd \Gamma_1$ and a finite-dimensional unitary representation $\rho : K \recht \cU(\M_d(\C))$ with corresponding projection $p_K := |K|^{-1} \sum_{k \in K} \rho(k) \ot u_k$,
\item a group homomorphism $\delta : \Lambda \recht \cG / L$ where
$$\cG := \{u \ot u_g \mid u \in \cU(\M_d(\C)) \; , \; g \in \Gamma_1 \; , \; \rho(gkg^{-1}) = u \rho(k) u^* \; \forall k \in K \}$$
and where the normal subgroup $K \cong L \lhd \cG$ is given by $L := \{\rho(k) \ot u_k \mid k \in K \}$,
\item a $\Gamma_1$-invariant projection $q \in A$,
\item a partial isometry $v \in \M_{n,d}(\C) \ot \rL \Gamma$ with $vv^* = p$ and $v^* v$ commuting with $\delta(\Lambda) L \subset \cG$,
\end{itemlist}
such that the composition of $\delta$ and the quotient homomorphism $\cG/L \recht \Gamma_1/K$ is surjective and such that
$w := \tau(q)^{-1/2} v(1 \ot q)$ is a partial isometry with left support $p$ and right support $p_K(1 \ot q)$ satisfying
$$w^* C w = (\M_d(\C) \ot A q)^{\Ad L} \, p_K \quad\text{and}\quad w^* \gamma(s) w = \delta(s) p_K (1 \ot q)\;\;\text{for all}\;\; s \in \Lambda \; .$$
\end{theorem}

\begin{proof}
Define the automorphism $\beta_s \in \Aut(C)$ as $\beta_s = \Ad \gamma(s)$.
Since $C \embed A$, the von Neumann algebra $C$ has a direct summand that is finite of type I. Since $(\be_s)_{s \in \Lambda}$ is ergodic on $\cZ(C)$, we find an integer $d$ such that $C \cong \M_d(\C) \ot \cZ(C)$. So, we can take matrix units $(e_{ij})_{i,j=1,\ldots,d}$ in $C$ with $e := e_{11}$ satisfying $e C e = \cZ(C) e$. By construction $\cZ(C) e$ is a maximal abelian subalgebra of $e (\M_n(\C) \ot M)e$ that is semi-regular: the normalizer of $\cZ(C) e$ acts ergodically on $\cZ(C) e$. Also $\cZ(C) e \embed A$.

Denote by $\D_m(\C) \subset \M_m(\C)$ the subalgebra of diagonal
matrices. Take an integer $m$ and a projection $q_1 \in \D_m(\C) \ot
A$ such that $(\Tr \ot \tau)(q_1) = (\Tr \ot \tau)(e)$. Write $B :=
\D_m(\C) \ot A$. By \cite[Theorem A.1]{Po01b}, we find $V_1 \in
\M_{n,m}(\C) \ot M$ such that $V_1 V_1^* = e$, $V_1^* V_1 = q_1$ and
$V_1^* \cZ(C) e V_1 = B q_1$. Put the elements $V_i = e_{i1} V_1$,
$i=1,\ldots,d$ next to each other, yielding
$$V \in \M_{n,d m}(\C) \ot M \quad\text{such that}\quad V V^* = p \; , \;\; V^* V = 1 \ot q_1 \;\;\text{and}\;\; V^* C V = \M_d(\C) \ot Bq_1 \; .$$

For every $s \in \Lambda$, the unitary $V^* \gamma(s) V \in \M_d(\C) \ot q_1 (\M_m(\C) \ot M)q_1$ normalizes $\M_d(\C) \ot Bq_1$. One can describe as follows all unitaries $w \in \M_d(\C) \ot q_1(\M_m(\C) \ot M)q_1$ normalizing $\M_d(\C) \ot Bq_1$. Then, $w$ also normalizes $1 \ot B q_1$ and we define the automorphism $\beta_w$ of $B q_1$ given by $1 \ot \beta_w(b) = w(1 \ot b)w^*$.
Denote by $e_1,\ldots,e_m$ the standard minimal projections in $\D_m(\C)$.
Write $q_1 = \sum_{k=1}^m e_k \ot q_k$. For all $k,l \in \{1,\ldots,m\}$ and $g \in \Gamma$ we find a projection $q^{k,g}_l \in A q_l$ such that
$$\sum_{k=1}^m \sum_{g \in \Gamma} q^{k,g}_l = q_l \quad\text{and}\quad \beta_w(e_l \ot a q^{k,g}_l) = e_k \ot \si_g(a q^{k,g}_l) \;\;\forall a \in A \; .$$
It follows that
$$w_1 := \sum_{k,l=1}^m \sum_{g \in \Gamma} e_{kl} \ot u_g q^{k,g}_l$$
is a unitary element in $q_1(\M_m(\C) \ot M) q_1$ satisfying $\beta_w(b) = w_1 b w_1^*$ for all $b \in B q_1$. It follows that $w_0 := w (1 \ot w_1^*)$ commutes with $1 \ot Bq_1$ and hence belongs to $\cU(\M_d(\C) \ot Bq_1)$. By construction, $w = w_0 (1 \ot w_1)$.

Define $X_m = X \sqcup \cdots \sqcup X$ as the disjoint union of $m$ copies of $X$. Identify $\rL^\infty(X) = B$. Let $Y \subset X_m$ be the support of the projection $q_1$.
Define the closed subgroup $\cG_1 \subset \cU(\M_d(\C) \ot \rL(\Gamma))$ given by
$$\cG_1 := \{ u \ot u_g \mid u \in \cU(\M_d(\C)) \;\;\text{and}\;\; g \in \Gamma \} \; .$$
We can view $w_0$ as a measurable function from $Y$ to $\cU(\M_d(\C))$. We then denote by $\Om_w : Y \recht \cG_1$ the measurable function given by
$$\Om_w(y) = w_0(y) \ot u_h \quad\text{whenever $y$ belongs to the support of}\;\; e_k \ot \si_h(q^{k,h}_l) = \beta_w(e_l \ot q^{k,h}_l) \; .$$

To make computations easier, we provide an alternative description of $\Om_w$. Define the Hilbert space $\cK := \M_{n,d m}(\C) \ot \rL^2(M)$ that we view as an $(\M_n(\C) \ot M)$-$(\M_{d m}(\C) \ot M)$-bimodule. Define the Hilbert space $\cH := \M_{n,d}(\C) \ovt \ell^2(\Gamma) \ovt \rL^2(B)$ that we view as an $(\M_n(\C) \ovt \rL \Gamma \ovt B)$-$(\M_d(\C) \ovt \rL \Gamma \ovt B)$-bimodule. Define the unitary operator
$$\eta : \cK \recht \cH : \eta(e_{i,jk} \ot u_g a) = e_{ij} \ot \delta_g \ot (e_k \ot a) \quad\text{for all indices}\;\; i,j,k, \;  g \in \Gamma, a \in A \; .$$
Viewing $\M_n(\C) \ot \rL \Gamma \subset \M_n(\C) \ot M$ and $\M_d(\C) \ot B \subset \M_{d m}(\C) \ot M$, we have for all $\xi \in \cK$ the obvious formulae
$$\eta(a \xi) = (a \ot 1) \xi \;\;\text{when}\;\; a \in \M_n(\C) \ot \rL \Gamma, \quad\text{and}\quad \eta(\xi b) = \eta(\xi) b_{13} \;\;\text{when}\;\; b \in \M_{d}(\C) \ot B \; .$$
An elementary computation yields
\begin{equation}\label{eq.goodformula}
\eta(\xi) \Om_w = (\id \ot \id \ot \beta_w)\eta(\xi w) \quad\text{for all}\quad \xi \in \cK(1 \ot q_1) \; .
\end{equation}
The following $1$-cocycle relation is then an immediate consequence.
\begin{equation}\label{eq.cocycle}
\Omega_{wv} =  \Omega_w \, (\id \ot \id \ot \beta_w)(\Omega_v) \quad\text{when both}\;\; w,v \;\;\text{normalize}\;\; \M_d(\C) \ot Bq_1 \; .
\end{equation}

Define for $s \in \Lambda$, $w_s := V^* \gamma(s) V$. Since $\M_d(\C) \ot Bq_1 = V^* C V$, the unitaries $w_s$ normalize
$\M_d(\C) \ot Bq_1$. So we can define the action $(\beta_s)_{s \in \Lambda}$ on $Bq_1$ given by $\beta_s = \beta_{w_s}$. We denote by $s * y$, $s \in \Lambda$, $y \in Y$, the corresponding action of $\Lambda$ on $Y$. By assumption, $\Lambda \actson Y$ is weakly mixing.

Thanks to the construction above, we can define the measurable function $\omega_1 : \Lambda \times Y \recht \cG_1$ given by $\omega_1(s,y) = \Omega_{w_s}(s * y)$. The $1$-cocycle relation \eqref{eq.cocycle} now becomes
$$\omega_1(st,y) = \omega_1(s,t * y) \; \omega_1(t,y) \quad\text{for all}\;\; s,t \in \Lambda \;\;\text{and almost every}\;\; y \in Y \; .$$
Hence, $\omega_1$ is a $1$-cocycle for the action $\Lambda \actson Y$ with values in $\cG_1$. Define the vector
$$\vphi \in \M_{n,d}(\C) \ovt \ell^2(\Gamma) \ovt \rL^2(Bq_1) \quad\text{given by}\quad \vphi = \eta(V) \; .$$
View $\vphi$ as a measurable function from $Y$ to $\M_{n,d}(\C) \ot \ell^2(\Gamma)$ and view the latter as an $(\M_n(\C) \ot \rL\Gamma)$-$(\M_d(\C) \ot \rL \Gamma)$-bimodule. By definition $p V = V$ and $\gamma(s) V = V w_s$ for all $s \in \Lambda$. The properties of $\eta$ imply that $p \vphi(y) = \vphi(y)$ a.e.\ and that $\eta(\gamma(s) V)$ equals a.e.\ the function given by $y \mapsto \gamma(s) \vphi(y)$. By \eqref{eq.goodformula} we have that $\eta(V w_s)$ equals a.e.\ the function given by $y \mapsto \vphi(s * y) \omega_1(s,y)$. So, we conclude that
$$\vphi : Y \recht p(\M_{n,d}(\C) \ot \ell^2(\Gamma)) \quad\text{and}\quad \gamma(s) \vphi(y) = \vphi(s * y) \omega_1(s,y) \quad\text{a.e.}$$
From now on, identify $p(\M_{n,d}(\C) \ot \ell^2(\Gamma)) = p \rL^2(\M_{n,d}(\C) \ot \rL \Gamma)$. So, we can define $P(y) := \vphi(y) \vphi(y)^*$ as an element in $p \rL^1(\M_n(\C) \ot \rL \Gamma) p$. We have $P(s * y) = \gamma(s) P(y) \gamma(s)^*$. Since $\Lambda \actson Y$ is weakly mixing, \cite[Lemma 5.4]{PV08} implies that $P$ is essentially constant. So, we have found an element $P \in p \rL^1(\M_n(\C) \ot \rL \Gamma) p$ such that $P(y) = P$ a.e. We claim that $P = (\Tr \ot \tau)(q_1)^{-1} p$. Indeed, for an arbitrary projection $f \in p(\M_n(\C) \ot \rL \Gamma)p$, we get
\begin{align*}
(\Tr \ot \tau)(f) &= \langle f V, V \rangle = \langle \eta(f V),\eta(V) \rangle = \int_Y \langle f \vphi(y),\vphi(y) \rangle \; d \rho(y) \\
& = \int_Y (\Tr \ot \tau)(f P) \; d\rho(y) = (\Tr \ot \tau)(fP) \; (\Tr \ot \tau)(q_1) \; .
\end{align*}
Since this holds for all projections $f$, the claim follows.

Define $\psi_1(y) := (\Tr \ot \tau)(q_1)^{1/2} \vphi(y)$. Denote by $\cI$ the set of all partial isometries in $\M_{n,d}(\C) \ot \rL \Gamma$ with left projection equal to $p$. So, $\psi_1 : Y \recht \cI$ and $\psi_1$ satisfies
$$\gamma(s) \psi_1(y) = \psi_1(s * y) \omega_1(s,y) \quad\text{a.e.}$$
The $\|\, \cdot\,\|_2$-distance turns $\cI$ into a Polish space on which $\cU(p(\M_n(\C) \ot \rL \Gamma)p)$ acts by left multiplication and $\cG_1$ by right multiplication. Both actions are isometric. The action of $\cG_1$ on $\cI$ by right multiplication is proper, so that the set $\cI/\cG_1$ of $\cG_1$-orbits equipped with the distance between orbits is still a Polish space on which $\cU(p(\M_n(\C) \ot \rL \Gamma)p)$ acts isometrically. Since $\psi_1(s * y) \cG_1 = \gamma(s) \psi_1(y) \cG_1$ a.e.\ and since $\Lambda \actson Y$ is weakly mixing, \cite[Lemma 5.4]{PV08} implies that $y \mapsto \psi_1(y) \cG_1$ is essentially constant. Take $v \in \cI$ such that $\psi_1(y) \in v \cG_1$ a.e.\ and denote $p_1 := v^* v$.

Define the compact subgroup $L \subset \cG_1$ consisting of the unitaries $u \ot u_g$ that satisfy $p_1 (u \ot u_g) = p_1$. Define the measurable map $\psi_2 : Y \recht L \backslash \cG_1$ such that $\psi_1(y) = v \psi_2(y)$ a.e. Composing $\psi_2$ with a measurable cross-section $L \backslash \cG_1 \recht \cG_1$, we find a measurable map $\psi : Y \recht \cG_1$ satisfying $\psi_1(y) = v \psi(y)$ a.e. Define the $1$-cocycle $\omega : \Lambda \times Y \recht \cG_1$ given by $\omega(s,y) = \psi(s * y) \omega_1(s,y) \psi(y)^{-1}$. Define the group morphism $\pi : \Lambda \recht \cU(p_1(\M_d(\C) \ot \rL \Gamma) p_1)$ given by $\pi(s) = v^* \gamma(s) v$. By construction,
$$\pi(s) = p_1 \omega(s,y) \quad\text{a.e.}$$
Define the closed subgroup $\cG_2 \subset \cG_1$ consisting of the unitaries $u \ot u_g$ that commute with $p_1$. It follows that $\omega$ takes values a.e.\ in $\cG_2$ and hence $\pi(s) \in \cG_2 p_1$ for all $s \in \Lambda$. Note that $L$ is a normal subgroup of $\cG_2$. We get a well-defined group morphism $\delta : \Lambda \recht \cG_2/L$ such that $\pi(s) = \delta(s) p_1$. So, $\delta(\Lambda)L$ commutes with $p_1 = v^* v$ and $v^* \gamma(s) v = \delta(s) p_1$ for all $s \in \Lambda$.

Write $\psi(y) = \zeta(y) \ot u_{\theta(y)}$. View $\zeta$ as a unitary element in $\M_d(\C) \ot B q_1$. Replacing $V$ by $V \zeta^*$, we may assume that $\psi(y) = 1 \ot u_{\theta(y)}$. Define the projection $q_g \in B q_1$ with support $\{y \in Y \mid \theta(y) = g\}$. Write
$$q_g = \sum_{k=1}^m e_k \ot q^k_g \quad\text{and}\quad v = \sum_{ij} e_{ij} \ot v_{ij} \; .$$
Since $\eta(V)$ equals a.e.\ the function $y \mapsto (\Tr \ot \tau)(q_1)^{-1/2} v(1 \ot u_{\theta(y)})$, it follows that
\begin{equation}\label{eq.onestarbis}
V = (\Tr \ot \tau)(q_1)^{-1/2} \sum_{g \in \Gamma} \sum_{i,j,k} e_{i,jk} \ot v_{ij} u_g q^k_g \; .
\end{equation}
Define the projections $\qtil^k_g := u_g q^k_g u_g^*$. Since $V^* V = 1 \ot q_1$ we have $$(1 \ot q_h) V^* V (1 \ot q_g) = \delta_{g,h} 1 \ot q_g$$
so that by \eqref{eq.onestarbis}, it follows that
$$
(1 \ot \qtil^j_h) p_1 (1 \ot \qtil^i_g) = (\Tr \ot \tau)(q_1) \delta_{i,j} \delta_{g,h} 1 \ot \qtil^i_g \; .
$$
Applying $\Tr \ot E_A$, it follows that the projections $\qtil^i_g$ are orthogonal. In particular, the sum of their traces is at most $1$, so that $(\Tr \ot \tau)(q_1) \leq 1$. Hence, we may assume from the beginning that $m=1$ and that $q_1 \in A$. We do not write the upper indices $i,j,k$ any more. Since the projections $\qtil_g$ are orthogonal, $u := \sum_{g \in \Gamma} u_g q_g$ is a partial isometry in $M$ with right support $q_1$, with left support in $A$ and such that $u \, A q_1 \, u^* = A \, uu^*$. Replacing $V$ by $V(1 \ot u^*)$ and $q_1$ by $uu^*$, we may further assume that $V = \tau(q_1)^{-1/2} v (1 \ot q_1)$. By construction, the $1$-cocycle $\om$ that corresponds to the group of unitaries $(V^* \gamma(s) V)_{s \in \Lambda}$ normalizing $\M_d(\C) \ot A q_1$ satisfies
\begin{equation}\label{eq.deltaomega}
p_1 \om(s,y) = \pi(s) = p_1 \delta(s) \quad\text{and hence,}\quad \om(s,y) L = \delta(s) L \; .
\end{equation}

Let $p_1 = \sum_{g \in \Gamma} P_g \ot u_g$, with $P_g \in \M_d(\C)$, be the Fourier decomposition of $p_1$.
Since $V = \tau(q_1)^{-1/2} v (1 \ot q_1)$, we have
\begin{equation}\label{eq.twostarbis}
(1 \ot q_1) p_1 (1 \ot q_1) = \tau(q_1) \, 1 \ot q_1 \; .
\end{equation}
Applying $\id \ot E_A$, we get that $P_e = \tau(q_1) \, 1$. So, when $u \ot u_k \in L$, the formula $p_1 (u^* \ot u_k^*) = p_1$ implies
that $P_k = \tau(q_1) u$. In particular, the homomorphism $L \recht \Gamma : u \ot u_k \mapsto k$ is injective. We denote the image by $K$ and define the unitary representation $\rho : K \recht \cU(\M_d(\C))$ such that $L = \{\rho(k) \ot u_k \mid k \in K\}$. Define $\Gamma_1$ as the image of $\delta(\Lambda) L$ in $\Gamma$. By construction, $K$ is a finite normal subgroup of $\Gamma_1$. Define $\cG$ as in the formulation of the theorem, i.e.\ as the unitaries $u \ot u_g$, $g \in \Gamma_1$, that normalize $L$. So, $\delta(\Lambda) L \subset \cG$.

Let $k \in K - \{e\}$. Multiplying \eqref{eq.twostarbis} on the right by $\rho(k)^* \ot u_k^*$ and applying $\id \ot E_A$, it follows that $q_1 \, \si_k(q_1) = 0$. Define the projection $q = \sum_{k \in K} \si_k(q_1)$. We claim that $q$ is $\Gamma_1$-invariant. Recall that $s * y$ denotes the action of $s \in \Lambda$ on $y \in Y$ implemented by $\Ad V^* \gamma(s) V$. Denote by $\mu : \Lambda \times Y \recht \Gamma$ and $\delta_1 : \Lambda \recht \Gamma_1/K$ the compositions of the $1$-cocycle $\om$ and the group morphism $\delta$ with the natural morphism $\cG \recht \Gamma$. By \eqref{eq.deltaomega} we have $\mu(s,y) K = \delta_1(s) K$, so that $\mu$ takes values in $\Gamma_1$ and
$$\delta_1(s) K \cdot y = \mu(s,y)K \cdot y = K \mu(s,y) \cdot y = K \cdot (s * y) \; .$$
Hence $\delta_1(\Lambda)K \cdot Y = K \cdot Y$, proving the claim.

Define the projection $p_K = |K|^{-1} \sum_{k \in K} \rho(k) \ot u_k$. Put $w := \tau(q)^{-1/2} v (1 \ot q)$. We make several computations to check that all the conclusions of the theorem hold. We freely use that $\tau(q) = |K|\, \tau(q_1)$, that $V = \tau(q_1)^{-1/2} v(1 \ot q_1)$ and that $v u = v p_1 u = v$ for all $u \in L$. First we get that
\begin{align*}
w w^* &= \tau(q)^{-1} \, v (1 \ot q) v^* = \sum_{u \in L} \tau(q)^{-1} \, v u (1 \ot q_1) u^* v^* \\
&= |L| \, \tau(q)^{-1} \, v (1 \ot q_1) v^* = V V^* = p \; .
\end{align*}
On the other hand,
\begin{align*}
w^* w &= \tau(q)^{-1} \, (1 \ot q) \, v^* v \, (1 \ot q) = \tau(q)^{-1} \sum_{u_1,u_2 \in L} u_1 (1 \ot q_1) u_1^* \, p_1 \,  u_2 (1 \ot q_1) u_2^* \\
&= \tau(q)^{-1} \sum_{u_1,u_2 \in L} u_1 (1 \ot q_1) p_1 (1 \ot q_1) u_2^* = |L|^{-1} \sum_{u_1,u_2 \in L} u_1 V^* V u_2^* \\
& = |L|^{-1} \sum_{u_1,u_2 \in L} u_1 (1 \ot q_1) u_2^* = p_K (1 \ot q) \; .
\end{align*}
Since $\delta(\Lambda) L$ commutes with $1 \ot q$ and $v^* \gamma(s) v = \delta(s) p_1$, it follows that $w^* \gamma(s) w = \delta(s) p_K(1 \ot q)$ for all $s \in \Lambda$. Finally,
\begin{align*}
w^* C w &= (1 \ot q) v^* C v (1 \ot q) = \sum_{u_1,u_2 \in L} u_1 (1 \ot q_1) v^* C v (1 \ot q_1) u_2^* \\
&= \sum_{u_1,u_2 \in L} u_1 V^* C V u_2^* = p_K (\M_d(\C) \ot A q_1) p_K = (\M_d(\C) \ot Aq)^{\Ad L} p_K \; .
\end{align*}
This ends the proof of the theorem.
\end{proof}

\begin{corollary}\label{cor.conjugacy-actions}
The conclusions of Theorem \ref{thm.conjugacy-actions} can be strengthened if we impose extra conditions. Denote by $N$ the von Neumann algebra generated by $C$ and $\gamma(\Lambda)$.
\begin{enumlist}
\item If we impose the extra condition that $N \not\embed A \rtimes \Centr g$ whenever $g \neq e$, it follows that $K=\{e\}$, $q=1$, $w=v$ and $v^* \gamma(s) v = \pi(s) \ot u_{\delta_1(s)}$ for all $s \in \Lambda$, where $\pi : \Lambda \recht \cU(\M_d(\C))$ and $\delta_1 : \Lambda \recht \Gamma$ are group morphisms. If moreover the weak mixing assumption is strengthened by imposing that $\C 1$ is the only non-zero, finite dimensional, globally $(\Ad \gamma(s))_{s \in \Lambda}$-invariant vector subspace of $C$, it follows that $d=1$ and that $\pi:\Lambda \recht \T$ is a character.
\item If we impose the extra condition that $N \not\embed A \rtimes \Gamma_1$ whenever $\Gamma_1 \actson (X,\mu)$ is non-ergodic, it follows that $q=1$ and $v^* v = p_K$.
\end{enumlist}
\end{corollary}
\begin{proof}
1. Choose a projection $q_1 \in Aq$ such that $q = \sum_{k \in K} \si_k(q_1)$. It follows that $w(\M_d(\C) \ot q_1) w^*$ is a globally $(\Ad \gamma(s))_{s \in \Lambda}$-invariant vector subspace of $C$. So, $d=1$ and the rest follows immediately.

2. Denote by $\delta_1 : \Lambda \recht \Gamma_1/K$ the composition of $\delta$ and the natural homomorphism $\cG / L \recht \Gamma_1/K$. Replacing $\Gamma_1$ by $\delta_1(\Lambda) K$, we may assume that $\delta_1$ is surjective. The conclusions of Theorem \ref{thm.conjugacy-actions} say in particular that $w^* N w \subset \M_d(\C) \ot (A \rtimes \Gamma_1)$. The extra condition $N \not\embed A \rtimes \Centr g$ whenever $g \neq e$, then implies that $\{h g h^{-1} \mid h \in \Gamma_1\}$ is infinite for all $g \neq e$. So, we can take a sequence $h_n \in \Gamma_1$ such that $h_n g h_n^{-1} \recht \infty$ for all $g \neq e$. Take $u_n \in \cU(\M_d(\C))$ such that $u_n \ot u_{h_n} \in \delta(\Lambda) L$. Since $v^* v$ commutes with $\delta(\Lambda)L$, it follows that $v^* v = p_0 \ot 1$ for some projection $p_0 \in \M_d(\C)$. But then $w^* w = \tau(q)^{-1} p_0 \ot q$. Since $w^* w$ actually equals $p_K (1 \ot q)$ it follows that $q=1$ and $K = \{e\}$.

3. As in the proof of 2, we get that $w^* N w \subset \M_d(\C) \ot (A \rtimes \Gamma_1)$. Since $q$ is $\Gamma_1$-invariant, the extra condition 3 implies that $q=1$. Hence $w=v$ and $v^* v = p_K$.
\end{proof}

\section{Some properties of the comultiplication}\label{sec.comult}

Throughout this section, we fix a countable group $\Lambda$ and put $M = \rL \Lambda$. We denote by $(u_g)_{g \in \Lambda}$ the canonical unitaries generating $\rL \Lambda$. We consider the comultiplication $\de : M \recht M \ovt M$ given by $\de(u_g) = u_g \ot u_g$ for all $g \in \Lambda$.

We start with the following elementary and well known lemma.

\begin{lemma}\label{lem.elemcomult}
A non-zero element $u \in M$ satisfies $\de(u) = u \ot u$ if and only if $u = u_g$ for some $g \in \Lambda$.

A unital von Neumann subalgebra $A \subset M$ satisfies $\de(A) \subset A \ovt A$ if and only if $A$ is of the form $A = \rL \Sigma$ for some subgroup $\Sigma < \Lambda$.
\end{lemma}
\begin{proof}
Observe that $(\id \ot \tau u_g^*)\de(x) = \tau(x u_g^*) u_g$ for all $g \in \Lambda$, $x \in M$. Let $u \in M$ be a non-zero element satisfying $\de(u) = u \ot u$. Take $g \in \Lambda$ such that $\tau(u u_g^*) \neq 0$. It follows that $u$ is a non-zero multiple of $u_g$. Since $\de(u) = u \ot u$, this multiple must be $1$.

Let $A \subset M$ be a von Neumann subalgebra satisfying $\de(A) \subset A \ovt A$. Define the subset $\Sigma \subset \Lambda$ consisting of the elements $g \in \Lambda$ for which there exists $a \in A$ with $\tau(a u_g^*) \neq 0$. Since $A \ni (\id \ot \tau u_g^*)\de(a) = \tau(a u_g^*) u_g$, it follows that $u_g \in A$ for all $g \in \Sigma$. Conversely, it is obvious that $g \in \Sigma$ whenever $u_g \in A$. Since $A$ is a von Neumann subalgebra, it follows that $\Sigma$ is a subgroup of $\Lambda$ and that $A = \rL \Sigma$.
\end{proof}

Recall from paragraph \ref{subsec.relamen} the notion of relative amenablity for von Neumann subalgebras.

\begin{proposition}\label{prop.about-comult}
Let $P \subset M$ be a von Neumann subalgebra.
\begin{enumlist}
\item If $P$ is diffuse, then $\de(P) \not\embed M \ot 1$ and $\de(P) \not\embed 1 \ot M$.

\item If $\de(M) \embed M \ovt P$, there exists a non-zero projection $p \in P' \cap M$ such that $Pp \subset pMp$ has finite index.

\item Denote by $\Centr g$ the centralizer of $g \in \Lambda$ and assume that for all $g \neq e$ we have $P \not\embed \rL(\Centr g)$. If $\cH \subset \rL^2(M \ovt M)$ is a $\de(P)$-$\de(M)$-subbimodule that is finitely generated as a right $\de(M)$-module, then $\cH \subset \de(\rL^2(M))$.

In particular, the quasi-normalizer of $\de(P)$ inside $M \ovt M$ is contained in $\de(M)$.
So, if $\Lambda$ is an icc group, the quasi-normalizer of $\de(M)$ inside $M \ovt M$ equals $\de(M)$.

\item If $P$ has no amenable direct summand, $\de(P)$ is strongly non-amenable relative to $M \ot 1$. In particular, if $N \subset M$ is an amenable von Neumann subalgebra, we have $\de(P) \not\embed M \ovt N$.
\end{enumlist}
\end{proposition}
\begin{proof}
1. Let $P$ be diffuse. Take a sequence $v_n \in \cU(P)$ tending to $0$ weakly. We claim that $\|E_{M \ot 1}(x \de(v_n) y^*)\|_2 \recht 0$ for all $x,y \in M \ovt M$. It suffices to prove this claim for $x=1 \ot u_g$ and $y = 1 \ot u_h$, $g,h \in \Lambda$. Then,
$$\|E_{M \ot 1}((1 \ot u_g) \de(v_n) (1 \ot u_h)^*)\|_2 = \|\tau(u_g v_n u_h^*) u_{g^{-1}h}\|_2 = |\tau(u_g v_n u_h^*)| \recht 0$$
and the claim follows. By Definition \ref{def.intertwine} $\de(P) \not\embed M \ot 1$. The statement $\de(P) \not\embed 1 \ot M$ follows similarly.

2. Assume that $\de(M) \embed M \ovt P$. Definition \ref{def.intertwine} provides elements $h_1,\ldots,h_n \in \Lambda$ and $\delta > 0$ such that
$$\sum_{i,j=1}^n \|E_{M \ovt P}((1 \ot u_{h_i}) \de(u_g) (1 \ot u_{h_j})^*)\|_2^2 \geq \delta \quad\text{for all}\quad g \in \Lambda \; .$$
This precisely means that
$$\sum_{i,j=1}^n \|E_P(u_{h_i} u_g u_{h_j}^*)\|_2^2 \geq \delta \quad\text{for all}\quad g \in \Lambda \; .$$
So, $M \embed_M P$. This means that $Pp \subset pMp$ has finite index for some non-zero projection $p \in P' \cap M$.

3. Assume that $P \not\embed \rL(\Centr g)$ for all $g \neq e$. By Definition \ref{def.intertwine} we find a sequence of unitaries $v_n \in \cU(P)$ such that $\|E_{\rL(\Centr g)}(u_h v_n u_k^*)\|_2 \recht 0$ for all $h,k \in \Lambda$ and all $g \neq e$. To conclude the proof of the proposition, it suffices to prove the following (see e.g.\ \cite[Lemma D.3]{Va06}, based on \cite[Theorem 3.1]{Po03}):
$$\|E_{\de(M)}(x \de(v_n) y^*)\|_2 \recht 0 \quad\text{for all}\quad x,y \in (M \ovt M) \ominus \de(M) \; .$$
It is sufficient to prove this statement for $x = u_h \ot u_k$ and $y = u_{h'} \ot u_{k'}$ with $h \neq k$ and $h' \neq k'$. In that case
$$\|E_{\de(M)}((u_h \ot u_k)\de(v_n)(u_{h'} \ot u_{k'})^*)\|_2^2 = \sum_{g \in \Lambda, hg(h')^{-1} = kg(k')^{-1}} |\tau(v_n u_g^*)|^2 \; .$$
If for all $g \in \Gamma$, we have $hg(h')^{-1} \neq kg(k')^{-1}$, this last expression is zero. If there is at least one $g_0 \in \Lambda$ such that $hg_0(h')^{-1} = kg_0(k')^{-1}$, this last expression equals
$$\sum_{g \in \Centr k^{-1}h} |\tau(v_n u_{gg_0}^*)|^2 = \|E_{\rL(\Centr k^{-1}h)}(v_n u_{g_0}^*)\|_2^2 \recht 0 \; .$$

4. Note that the $M$-$(M \ovt M)$-bimodule $\bim{(\de(M) \ot 1)}{\rL^2(M \ovt M \ovt M)}{(M \ovt 1 \ovt M)}$ is isomorphic with the coarse $M$-$(M \ovt M)$-bimodule $\rL^2(M) \ovt \rL^2(M \ovt M)$. Assume that $\de(P)$ is not strongly non-amenable relative to $M \ot 1$. We get a non-zero projection $p \in \de(P)' \cap (M \ovt M)$ such that $\bim{\de(P)}{\rL^2(p (M \ovt M))}{M \ovt M}$ is weakly contained in $\bim{(\de(P) \ot 1)}{\rL^2(M \ovt M \ovt M)}{(M \ovt 1 \ovt M)}$ and hence, weakly contained in the coarse $P$-$(M \ovt M)$-bimodule. Take $z \in P$ such that $\de(z)$ is the support projection of $E_{\de(P)}(p)$. Note that $z$ is a non-zero central projection in $P$ and that $\de$ embeds the trivial $Pz$-$Pz$-bimodule into $\bim{\de(Pz)}{\rL^2(\de(z) (M \ovt M) \de(z))}{\de(Pz)}$. It follows that the trivial $Pz$-$Pz$-bimodule is weakly contained in the coarse $Pz$-$Pz$-bimodule so that $Pz$ is amenable.

If $N \subset M$ is an amenable von Neumann subalgebra, then $M \ovt N$ is amenable relative to $M \ot 1$. If $\de(P) \embed M \ovt N$, it follows that $\de(P) p$ is amenable relative to $M \ot 1$ for some non-zero projection $p \in \de(P)' \cap (M \ovt M)$. So, $\de(P)$ is not strongly non-amenable relative to $M \ot 1$. The previous paragraph implies that $P$ has an amenable direct summand.
\end{proof}

\section{Proof of Theorem \ref{thm.specialmain}: superrigidity of group von Neumann algebras}\label{sec.proof}

Theorem \ref{thm.specialmain} is a specific instance of a general superrigidity theorem for group factors $\rL G$ where $G$ arises as a generalized wreath product $G = H_0 \wr_I \Gamma$ for certain group actions $\Gamma \actson I$. The class of actions $\Gamma \actson I$ that we are able to treat, is defined as follows.

\begin{condition}\label{cond}
We say that $\Gamma \actson I$ satisfies Condition \ref{cond} if the following two sets of conditions hold.

{\bf Conditions on the group.} The group $\Gamma$ is icc and admits a chain of infinite subgroups $\Gamma_0 < \Gamma_1 < \cdots < \Gamma_n = \Gamma$ such that $\Gamma_{k-1}$ is almost normal in $\Gamma_k$ for all $k = 1,\ldots,n$. Moreover, at least one of the following rigidity properties hold.
\begin{itemlist}
\item $\Gamma_0 < \Gamma_1$ has the relative property (T).
\item The centralizer of $\Gamma_0$ inside $\Gamma_1$ is non-amenable.
\end{itemlist}

{\bf Conditions on the action.}
\begin{itemlist}
\item There exists $\kappa \in \N$ such that $\Stab J$ is finite whenever $J \subset I$ and $|J| \geq \kappa$.
\item $\Stab i$ is amenable for all $i \in I$.
\end{itemlist}
\end{condition}

The conditions on the group $\Gamma$ in \ref{cond} are satisfied whenever $\Gamma$ is an icc group with property (T), whenever $\Gamma$ is the direct product of two icc groups with at least one of them being non-amenable or whenever $\Gamma$ is itself a wreath product $\Gamma = \Gamma_0 \wr S$ with $\Gamma_0$ being non-amenable and $S$ non-trivial. Indeed, in this last case, we consider the chain of subgroups $\Gamma_0 < \Gamma_0^{(S)} < \Gamma$.

The conditions on the action in \ref{cond} are automatically satisfied when we let $\Gamma$ act on itself by multiplication. They are also satisfied when $\Gamma \actson \Gamma/S$ where $S < \Gamma$ is an amenable subgroup that is almost malnormal: $g S g^{-1} \cap S$ is finite for all $g \in \Gamma - S$.

Whenever $\Gamma \actson I$ satisfies Condition \ref{cond}, we consider the generalized wreath product $G = H_0 \wr_I \Gamma$ and describe all countable groups $\Lambda$ such that $\rL \Lambda \cong \rL G$. The main result is the following Theorem \ref{thm.main-gen}. The conclusions of Theorem \ref{thm.main-gen} can be made significantly more precise if we moreover assume that $\Stab i \cdot j$ is infinite for all $i \neq j$. This excludes plain wreath products and will lead to Theorem \ref{thm.main} below, of which Theorem \ref{thm.specialmain} is a special case.

\begin{theorem}\label{thm.main-gen}
Assume that $\Gamma \actson I$ satisfies Condition \ref{cond}. Let $H_0$ be a non-trivial abelian group and define the generalized wreath product group $G := H_0 \wr_I \Gamma := H_0^{(I)} \rtimes \Gamma$. Denote by $A$ the abelian von Neumann algebra $A =\rL(H_0^{(I)})$ and denote by $(\si_g)_{g \in \Gamma}$ the corresponding generalized Bernoulli action of $\Gamma$ on $A$.

If $\Lambda$ is any countable group and $\pi : \rL \Lambda \recht \rL(G)^t$ a $*$-isomorphism for some $t > 0$, then $t = 1$ and $\Lambda \cong \Sigma \rtimes \Gamma$ for some infinite abelian group $\Sigma$ and some action $\Gamma \action{\al} \Sigma$ by automorphisms.

More precisely, there exists a group isomorphism $\delta : \Lambda \recht \Sigma \rtimes \Gamma$, a $*$-isomorphism $\theta : \rL \Sigma \recht A$ satisfying $\theta \circ \al_g = \si_g \circ \theta$ for all $g \in \Gamma$, a character $\om : G \recht \T$ and a unitary $w \in \rL G$ such that $\pi = \Ad w \circ \pi_\om \circ \pi_\theta \circ \pi_\delta$ where
\begin{itemlist}
\item $\pi_\delta : \rL \Lambda \recht \rL(\Sigma \rtimes \Gamma)$ is the isomorphism given by $\pi_\delta(v_s) = u_{\delta(s)}$ for all $s \in \Lambda$,
\item $\pi_\theta : \rL(\Sigma) \rtimes \Gamma \recht A \rtimes \Gamma$ is given by $\pi_\theta(a u_g) = \theta(a) u_g$ for all $a \in \rL(\Sigma)$ and all $g \in \Gamma$,
\item $\pi_\om$ is the automorphism of $\rL G$ given by $\pi_\om(u_g) = \om(g) \, u_g$ for all $g \in G$,
\end{itemlist}
\end{theorem}

In order to fully understand all groups $\Lambda$ for which $\rL \Lambda \cong \rL G$, we need to classify all actions of $\Gamma$ by group automorphisms of a countable abelian group $\Sigma$ such that the corresponding measure preserving action $\Gamma \actson \widehat{\Sigma}$ is conjugate with the given generalized Bernoulli action $\Gamma \actson X_0^I$ with base space $X_0 = \widehat{H_0}$. As we illustrate in Section \ref{sec.counterex}, such a classification is untractable for plain wreath products $H_0 \wr \Gamma$. If we however specialize to the case where moreover $\Stab i \cdot j$ is infinite for all $i \neq j$, we get the following full superrigidity theorem.

\begin{theorem}\label{thm.main}
Assume that $\Gamma \actson I$ satisfies Condition \ref{cond} and that $\Stab i \cdot j$ is infinite for all $i \neq j$.
Let $H_0$ be a non-trivial abelian group and define the generalized wreath product group $G := H_0 \wr_I \Gamma = H_0^{(I)} \rtimes \Gamma$.
Let $\Lambda$ be any countable group and $\pi : \rL \Lambda \recht \rL(G)^t$ a $*$-isomorphism for some $t > 0$.
\begin{itemlist}
\item In the case where $|H_0|$ is a square-free integer, we must have $t=1$ and $\Lambda \cong G$.
\item In the general case, but assuming that $\Gamma \actson I$ is transitive, we must have $t=1$ and $\Lambda \cong H_1 \wr_I \Gamma$ for some abelian group $H_1$ with $|H_1| = |H_0|$.
\item In the case where $H_0 = \Z/2\Z$ or $H_0 = \Z/3\Z$, we must have $t=1$ and there exists an isomorphism of groups $\delta : \Lambda \recht G$, a character $\om : \Lambda \recht \T$ and a unitary $w \in \rL G$ such that
$$\pi(v_s) = \om(s) \, w \, u_{\delta(s)} \, w^* \quad\text{for all}\;\; s \in \Lambda \; .$$
\end{itemlist}
\end{theorem}

\begin{example}\label{ex.ourex}
If $\Gamma \actson I$ is defined as in Theorem \ref{thm.specialmain}, it is easy to check that all conditions of Theorem \ref{thm.main} are indeed satisfied, using the subgroup $\Gamma_0 < \Gamma$ (that we put in an arbitrary position of $\Gamma_0^{(S)}$) and the chain of normal subgroups $\Gamma_0 \lhd \Gamma_0^{(S)} \lhd \Gamma$.

Define $\Gamma = \SL(2,\Z) \ltimes \Z^2$. Let $A \in \SL(2,\Z)$ be any matrix whose eigenvalues have modulus different from $1$. Define the subgroup $\Gamma_A < \SL(2,\Z)$ consisting of the matrices $B$ such that $B A B^{-1} = A^{\pm 1}$. View $\Gamma_A$ as a subgroup of $\Gamma$.
Then, the action $\Gamma \actson \Gamma/\Gamma_A$ satisfies all conditions of Theorem \ref{thm.main} with $\kappa =2$.

More generally, whenever the icc group $\Gamma$ admits an infinite almost normal subgroup with the relative property (T) and $S < \Gamma$ is an infinite amenable almost malnormal subgroup, then $\Gamma \actson \Gamma/S$ satisfies the conditions of Theorem \ref{thm.main} with $\kappa = 2$. Examples of infinite amenable almost malnormal subgroups of $\PSL(n,\Z)$ are provided in \cite[Example 7.4]{PV06}.
\end{example}

\subsection*{Proof of Theorem \ref{thm.main-gen}}\mbox{}

Fix $\Gamma \actson I$ satisfying Condition \ref{cond}. Choose a non-trivial abelian group $H_0$ and put $A_0 := \rL(H_0)$, $A := \rL(H_0^{(I)})$. Denote $M = \rL(H_0 \wr_I \Gamma) = A \rtimes \Gamma$.

We first prove that the action $\Gamma \actson A$ is essentially free and ergodic. It suffices to prove that every $g \in \Gamma - \{e\}$ moves infinitely many $i \in I$. Choose $n \in \N$. For every $g \in \Gamma$, denote $\Fix g := \{i \in I \mid g \cdot i = i \}$. It suffices to prove that $\cG_n := \{g \in \Gamma \mid |I - \Fix g| \leq n\}$ equals $\{e\}$. Since $h \cG_n h^{-1} = \cG_n$ for all $h \in \Gamma$ and since $\Gamma$ is icc, it suffices to prove that $\cG_n$ is finite. Choose a finite subset $\cF \subset I$ such that $|\cF| = \kappa + n$. Then,
$$\cG_n \subset \bigcup_{\cF_0 \subset \cF \; , \; |\cF_0| = \kappa} \Stab \cF_0 \; .$$
Since all $\Stab \cF_0$ are finite, $\cG_n$ is finite as well. We have proven that $\Gamma \actson A$ is essentially free. Because $\Gamma \cdot i$ is infinite for all $i \in I$, the action $\Gamma \actson A$ is ergodic as well.

Assume that $\rL \Lambda = M^t$ for some countable group $\Lambda$. The amplification of the comultiplication on $\rL \Lambda$ yields a unital $*$-homomorphism $\de : M \recht (M \ovt M)^t$. To avoid unnecessary notational complexity, in the first steps of the proof, until step \ref{step3} included, we will do as if $t \leq 1$ and consider $\de : M \recht p(M \ovt M)p$ for some projection $p \in M \ovt M$. The reader can check easily that this notational simplification does not hide any essential steps of the argument.

\begin{step}\label{step1}
There exists $v \in M \ovt M$ with $v^* v = p$ and $v \de(\rL \Gamma) v^* \subset \rL( \Gamma \times \Gamma)$.
\end{step}
\begin{proof}
Take a chain of subgroups $\Gamma_0 < \Gamma_1 < \cdots < \Gamma_n = \Gamma$ as in Condition \ref{cond}. Note that $\Gamma_1$ is non-amenable.
Put $Q = \de(\rL \Gamma_0)$ and denote by $P$ the quasi-normalizer of $Q$ inside $p (M \ovt M)p$. Note that $\de(\rL \Gamma_1) \subset P$. In the case where $\Gamma_0 < \Gamma_1$ has the relative property (T), $Q \subset P$ has the relative property (T). In the case where the centralizer of $\Gamma_0$ inside $\Gamma_1$ is non-amenable, Proposition \ref{prop.about-comult} implies that the relative commutant $Q' \cap P$ is strongly non-amenable relative to $M \ot 1$.

By Proposition \ref{prop.about-comult}.1, $Q \not\embed M \ot 1$. By Proposition \ref{prop.about-comult}.4 and because $\de(\Gamma_1) \subset P$, we have $P \not\embed M \ovt (A \rtimes \Stab i)$ for all $i \in I$. So, Corollary \ref{cor.malleable-rigid} yields $v \in M \ovt M$ with $v^* v = p$ and $v P v^* \subset M \ovt \rL \Gamma$.

Repeating the same argument and applying Corollary \ref{cor.malleable-rigid} with $N = \rL \Gamma$, we find $w \in M \ovt \rL \Gamma$ such that $w^* w = vv^*$ and $w v P v^* w^* \subset \rL \Gamma \ovt \rL \Gamma$.

We write $v$ instead of $wv$, so that $v^* v = p$ and $v P v^* \subset \rL(\Gamma \times \Gamma)$. In particular, $v \de(\rL \Gamma_1) v^* \subset \rL(\Gamma \times \Gamma)$. Write $P_k := v \de(\rL \Gamma_k) v^*$. We prove by induction on $k$ that automatically $P_k \subset \rL(\Gamma \times \Gamma)$. For $k =1$ the statement is already proven. Assume that $P_k \subset \rL(\Gamma \times \Gamma)$ for some $1 \leq k \leq n-1$. We already observed that $P_1 \not\embed \rL(\Gamma \times \Stab i)$ and $P_1 \not\embed \rL(\Stab i \times \Gamma)$ so that, a fortiori, the same holds for $P_k$ instead of $P_1$. By Lemma \ref{lem.control}.1 and because $P_k \subset P_{k+1}$ is quasi-regular, it follows that $P_{k+1} \subset \rL(\Gamma \times \Gamma)$.

Since $\Gamma = \Gamma_n$, we have proven that $v \de(\rL \Gamma) v^* \subset \rL( \Gamma \times \Gamma)$.
\end{proof}

From now on, we replace $\de : M \recht p(M \ovt M)p$ by $v \de(\, \cdot \,) v^*$ and $p$ by $vv^* \in \rL(\Gamma \times \Gamma)$, so that $\de(\rL \Gamma) \subset p\rL(\Gamma \times \Gamma)p$.

Denote $C := \de(A)' \cap p(M \ovt M) p$.

\begin{step}\label{step2}
We have $C \embed A \ovt A$.
\end{step}
\begin{proof}
We apply Theorem \ref{thm.intertwine-abelian} to the abelian von Neumann subalgebra $D := \Delta(A)$ of $p (M \ovt M)p$ that is normalized by the unitaries $(\Delta(u_g))_{g \in \Gamma}$ that belong to $p \rL(\Gamma \times \Gamma) p$. So we have to check the four assumptions (\ref{assum1})-(\ref{assum4}) of Theorem \ref{thm.intertwine-abelian}.

Since $A$ is diffuse, Proposition \ref{prop.about-comult}.1 says that $\Delta(A) \not\prec M \ot 1$ and $\Delta(A) \not\prec 1 \ot M$. So assumption (\ref{assum1}) holds.

The quasi-normalizer of $\Delta(A)$ inside $p(M \ovt M)p$ contains $\Delta(M)$. Since for every $i \in I$ we have that $\Stab i \subset \Gamma$ has infinite index, Proposition \ref{prop.about-comult}.2 implies that $\Delta(M) \not\prec M \ovt (A \rtimes \Stab i)$ and $\Delta(M) \not\prec (A \rtimes \Stab i) \ovt M$. So assumption (\ref{assum2}) holds. Since also $\rL \Gamma \subset M$ has infinite index, for the same reason assumption (\ref{assum3}) holds.

Finally, since $\Gamma$ is non-amenable and $\Stab i$ is amenable for every $i \in I$, Proposition \ref{prop.about-comult}.4 implies that $\Delta(\rL \Gamma) \not\prec \rL (\Gamma \times \Stab i)$ and $\Delta(\rL \Gamma) \not\prec \rL (\Stab i \times \Gamma)$. So also assumption (\ref{assum4}) holds.

The conclusion of Step \ref{step2} now follows from Theorem \ref{thm.intertwine-abelian}.
\end{proof}

Since $C = p(M \ovt M)p \cap \de(A)'$, the unitaries $\de(u_g)$ normalize $C$ and define an action $(\be_g)_{g \in \Gamma}$ of $\Gamma$ on $C$ given by $\be_g(d) = \de(u_g) d \de(u_g)^*$ for all $g \in \Gamma$, $d \in C$.

\begin{step}\label{step3}
If $\cH \subset \rL^2(C)$ is a finite dimensional $(\be_g)_{g \in \Gamma}$-invariant subspace, we have $\cH \subset \C 1$.
\end{step}
\begin{proof}
Define $\cK \subset p \rL^2(M \ovt M) p$ as the norm closed linear span of $\cH \de(M)$. Then, $\de(A) \cK \subset \cK$ because $\cH$ and $\de(A)$ commute. Also, $\de(u_g) \cK = \cK$ for all $g \in \Gamma$ because $\cH$ is globally invariant under $(\be_g)_{g \in \Gamma}$. So, $\cK$ is a $\de(M)$-$\de(M)$-bimodule which, by construction, is finitely generated as a right $\de(M)$-module. By Proposition \ref{prop.about-comult}.3, we have $\cK \subset \de(\rL^2(M))$ and hence $\cH \subset \de(\rL^2(M))$. Since elements of $\cH$ commute with $\de(A)$, we have $\cH \subset \de(\rL^2(A))$. Since the action of $\Gamma$ on $A$ is weakly mixing, the global invariance under $(\be_g)_{g \in \Gamma}$ forces $\cH \subset \C 1$.
\end{proof}

\begin{step}\label{step5}
We have $t = 1$ and there exists a unitary $\Omega \in M \ovt M$, a group homomorphism $\delta : \Gamma \recht \Gamma \times \Gamma$ and a character $\omega : \Gamma \recht \T$ such that
\begin{equation}\label{eq.adrian}
\Om^* \de(u_g) \Om = \om(g) \, u_{\delta(g)} \quad\text{for all}\;\; g \in \Gamma \quad\text{and}\quad \Om^* \de(A) \Om \subset A \ovt A \; .
\end{equation}
\end{step}
\begin{proof}
We apply Corollary \ref{cor.conjugacy-actions} to the crossed product $M \ovt M = (A \ovt A) \rtimes (\Gamma \times \Gamma)$. We no longer make the simplifying assumption that $t \leq 1$. So, take a projection $p \in \M_n(\C) \ovt M \ovt M$ with $(\Tr \ot \tau \ot \tau)(p) = t$. The amplified comultiplication is a unital $*$-homomorphism $\de : M \recht p(\M_n(\C) \ovt M \ovt M)p$ and by step \ref{step1} we may assume, after a unitary conjugacy, that $p \in \M_n(\C) \ot \rL(\Gamma \times \Gamma)$ and $\de(\rL\Gamma) \subset p (\M_n(\C) \ot \rL (\Gamma \times \Gamma))p$. Put $C = \de(A)' \cap p(\M_n(\C) \ovt M \ovt M)p$. Since $A$ is abelian, $C' \cap p(\M_n(\C) \ovt M \ovt M)p = \cZ(C)$. By step \ref{step2}, $C \embed A \ovt A$. By step \ref{step3}, the action $(\Ad \de(u_g))_{g \in \Gamma}$ is weakly mixing on $\cZ(C)$. Even more so, $\C 1$ is the only finite-dimensional globally $(\Ad \de(u_g))_{g \in \Gamma}$-invariant subspace of $C$. Since $\Gamma$ is an icc group, Proposition \ref{prop.about-comult}.2 implies that $\de(M) \not\embed M \ovt (A \rtimes \Centr g)$ and $\de(M) \not\embed (A \rtimes \Centr g) \ovt M$. Because the von Neumann algebra generated by $C$ and $\de(u_g), g \in \Gamma$ contains $\de(M)$, all conditions of Theorem \ref{thm.conjugacy-actions}, together with the extra condition 1 in Corollary \ref{cor.conjugacy-actions}, are satisfied.

By Corollary \ref{cor.conjugacy-actions} we get that $t = 1$ and that there  exists a unitary $\Omega \in M \ovt M$, a group homomorphism $\delta : \Gamma \recht \Gamma \times \Gamma$ and a character $\omega : \Gamma \recht \T$ such that $\Om^* \de(u_g) \Om = \om(g) \, u_{\delta(g)}$ and $\Om^* C \Om = A \ovt A$. In particular, $\Om^* \de(A) \Om \subset A \ovt A$.
\end{proof}

\begin{step}
End of the proof of Theorem \ref{thm.main-gen}.
\end{step}

\begin{proof}
Take $\Om,\delta,\omega$ as in step \ref{step5}. After step \ref{step1}, we decided to replace $\de$ by $\Ad v \circ \de$. From now on, $\de : \rL \Lambda \recht \rL \Lambda \ovt \rL \Lambda$ is again the comultiplication. The conclusion of step \ref{step5} remains of course true, replacing $\Om$ by $v^* \Om$.

Write $\delta(g) = (\delta_1(g),\delta_2(g))$. By Proposition \ref{prop.about-comult}.2, $\de(M) \not\embed M \ovt (A \rtimes S)$ whenever $S < \Gamma$ is of infinite index. Hence, the subgroups $\delta_i(\Gamma)$, $i=1,2$, are of finite index in $\Gamma$.

Applying the flip to \eqref{eq.adrian}, it follows that $u_{\delta_1(g)} \ot u_{\delta_2(g)}$ and $u_{\delta_2(g)} \ot u_{\delta_1(g)}$ are unitarily conjugate inside $M \ovt M$. Since $\delta_i(\Gamma) \subset \Gamma$ has finite index, there must exist $h \in \Gamma$ such that $\delta_2(g) = h \delta_1(g) h^{-1}$ for all $g \in \Gamma$. Replacing $\Om$ by $\Om (1 \ot u_h)$, we may assume that $\delta_1 = \delta_2$ and we write $\delta$ instead of $\delta_1,\delta_2$.

Define $\Gamma' = \delta(\Gamma)$. The co-associativity of $\de$ implies that $(u_{\delta(g)} \ot u_{\delta(g)} \ot u_g)_{g \in \Gamma'}$ and $(u_g \ot u_{\delta(g)} \ot u_{\delta(g)})_{g \in \Gamma'}$ are unitarily conjugate in $M \ovt M \ovt M$. Since $\Gamma' < \Gamma$ has finite index, it follows that there exists $h \in \Gamma$ such that $\delta(g) = h g h^{-1}$ for all $g \in \Gamma'$. Then automatically, $\delta(g) = h g h^{-1}$ for all $g \in \Gamma$. Replacing $\Om$ by $\Om(u_h \ot u_h)$, we may assume that
$$\Om^*\de(u_g) \Om = \om(g) \; u_g \ot u_g \quad\text{for all}\;\; g \in \Gamma \; .$$
If $\si(a \ot b) = b \ot a$ denotes the flip map, it follows that $\Om^* \si(\Om)$ commutes with all $u_g \ot u_g$, $g \in \Gamma$ and hence, is scalar. Similarly,
$$(\Om \ot 1)^* (\de \ot \id)(\Om)^* (\id \ot \de)(\Om)(1 \ot \Om) \quad\text{commutes with all}\quad u_g \ot u_g \ot u_g \; , \; g \in \Gamma$$
and hence, is scalar. By Theorem \ref{thm.triv-2-coc}, there exists a unitary $w \in M$ such that $\Om = \de(w^*)(w \ot w)$.

To make the end of the argument more clear, we write again explicitly the isomorphism $\pi : \rL \Lambda \recht \rL(G)^t$, instead of the implicit identification $\rL \Lambda = \rL(G)^t$. So far, we have shown that $t=1$ and we have found a unitary $w \in \rL G$ and a character $\om : \Gamma \recht \T$ such that after replacing $\pi$ by $\pi_{\om}^{-1} \circ \Ad w^* \circ \pi$, we have
$$(\pi \ot \pi)\de(\pi^{-1}(A)) = A \ovt A \quad\text{and}\quad (\pi \ot \pi)\de(\pi^{-1}(u_g)) = u_g \ot u_g$$
for all $g \in \Gamma$. By Lemma \ref{lem.elemcomult} we find an abelian subgroup $\Sigma < \Lambda$ such that $\pi^{-1}(A) = \rL \Sigma$ and an injective group homomorphism $\rho : \Gamma \recht \Lambda$ such that $\pi^{-1}(u_g) = v_{\rho(g)}$. By construction, $\Ad v_{\rho(g)}$ normalizes $\rL \Sigma$ and hence, $\Ad \rho(g)$ normalizes $\Sigma$. We have found an action of $\Gamma$ by automorphisms of $\Sigma$ and an isomorphism of groups $\delta : \Lambda \recht \Sigma \rtimes \Gamma$ satisfying $\delta(s \rho(g)) = (s,g)$ for all $s \in \Sigma$, $g \in \Gamma$. Moreover, the $*$-isomorphism $\pi \circ \pi_{\delta^{-1}} : \rL(\Sigma) \rtimes \Gamma \recht A \rtimes \Gamma$ maps $\rL \Sigma$ onto $A$ and is the identity on $u_g, g \in \Gamma$. We define $\theta : \rL \Sigma \recht A$ as the restriction of $\pi \circ \pi_{\delta^{-1}}$ to $\rL \Sigma$, ending the proof of the theorem.
\end{proof}

This ends the proof of Theorem \ref{thm.main-gen}.\hfill\qedsymbol

\subsection*{Proof of Theorem \ref{thm.main}}\mbox{}

By Theorem \ref{thm.main-gen}, we have $t=1$ and $\pi = \Ad w \circ \pi_\om \circ \pi_\theta \circ \pi_\delta$, where $w \in \rL G$ is a unitary, $\delta : \Lambda \recht \Sigma \rtimes \Gamma$ is a group isomorphism, $\om : \Lambda \recht \T$ is a character and $\theta : \rL \Sigma \recht A$ is a $*$-isomorphism satisfying $\theta \circ \al_g = \si_g \circ \theta$ for all $g \in \Gamma$.

For all $i \in I$, put $\Gamma_i := \Stab i$. Recall that $A = \rL(H_0^{(I)})$. Denote by $H_0^i < H_0^{(I)}$ the copy of $H_0$ in position $i \in I$. Define the subalgebra $B_i \subset \rL \Sigma$ given by $B_i := \theta^{-1}\bigl( \rL H_0^i \bigr)$.
We claim that $\de(B_i) \subset B_i \ovt B_i$. If $b \in B_i$, the element $(\theta \ot \theta)\de(b)$ is fixed under the automorphisms $\si_g \ot \si_g$, $g \in \Stab i$. Since $\Stab i \cdot j$ is infinite for all $j \neq i$, this implies that $(\theta \ot \theta)\de(b) \in \rL(H_0^i \times H_0^i)$. Hence, $\de(b) \in B_i \ovt B_i$. By Lemma \ref{lem.elemcomult} we find subgroups $\Sigma_i < \Sigma$ such that $B_i = \rL \Sigma_i$. By construction, the subalgebras $B_i \subset \rL \Sigma$ are independent and generate $\rL \Sigma$. Hence $\Sigma = \bigoplus_{i \in I} \Sigma_i$. Denote by $\theta_i$ the restriction of $\theta$ to $\rL \Sigma_i$. So, $\theta_i : \rL \Sigma_i \recht \rL H_0^i$ is a $*$-isomorphism.

In particular, $\Sigma_i$ is an abelian group of order $|H_0|$. So, if $|H_0|$ is a square-free integer, necessarily $\Sigma_i \cong H_0$ for every $i \in I$ and we easily conclude that $\Lambda \cong G$. For general non-trivial abelian groups $H_0$, but assuming that $\Gamma \actson I$ is transitive, choose $i_0 \in I$ and put $H_1 := \Sigma_{i_0}$. We have proven that $\Lambda \cong H_1 \wr_I \Gamma$.

In the specific case where $H_0 = \frac{\Z}{2\Z}$ or $\frac{\Z}{3\Z}$, every algebra isomorphism $\rL \Sigma_i \recht \rL H_0^i$ is group-like. So, we find characters $\gamma_i : H_0^i \recht \T$ and group isomorphisms $\rho_i : \Sigma_i \recht H_0^i$ such that $\theta_i = \pi_{\gamma_i} \circ \pi_{\rho_i}$.

By construction, $\gamma_{g \cdot i} = \gamma_i \circ \si_g^{-1}$, $\al_g(\Sigma_i) = \Sigma_{g \cdot i}$ and $\sigma_g \circ \rho_i = \rho_{g \cdot i} \circ \al_g$. So, all $\gamma_i$ combine into a $(\si_g)_{g \in \Gamma}$-invariant character $\gamma : H_0^{(I)} \recht \T$ and all $\rho_i$ combine into an group isomorphism $\rho : \Sigma \recht H_0^{(I)}$ satisfying $\rho \circ \al_g = \si_g \circ \rho$ for all $g \in \Gamma$. By construction, $\theta = \pi_\gamma \circ \pi_\rho$. We extend $\gamma$ to a character $\gamma : G \recht \T$ by putting $\gamma(g) = 1$ for all $g \in \Gamma$. We extend $\rho$ to a group isomorphism $\rho : \Sigma \rtimes \Gamma \recht G$ by putting $\rho(g) = g$ for all $g \in \Gamma$. By construction, $\pi_\theta = \pi_\gamma \circ \pi_\rho$. We have proven that
$$\pi = \Ad w \circ \pi_{\om \, \gamma} \circ \pi_{\rho \circ \delta} \; .$$
This ends the proof of Theorem \ref{thm.main}.\hfill\qedsymbol

\section{Counterexamples for plain wreath products: proof of Theorem \ref{thm.counterex}} \label{sec.counterex}

Assume that $\Gamma$ is a countable group and $\Z \hookrightarrow \Gamma$ an embedding. Let $H_0$ be a non-trivial finite abelian group. Using the \emph{co-induction construction,} we construct a new group $\Lambda$ such that $\rL (\Lambda) \cong \rL (H_0 \wr \Gamma)$.

Define the countable abelian group $\Sigma_0 := \Z[|H_0|^{-1}]$ and denote by $\al$ the automorphism of $\Sigma_0$ given through multiplication by $|H_0|$. We also denote by $(\al_k)_{k \in \Z}$ the corresponding action of $\Z$ by group automorphisms of $\Sigma_0$ and then, by automorphisms of $\rL(\Sigma_0)$. We claim that $\al$ is conjugate with a Bernoulli action with base space $\{1,\ldots,|H_0|\}$ equipped with the normalized counting measure. View $\rL^\infty(\T) \cong \rL \Z \subset \rL(\Sigma_0)$. Identify $\rL(H_0) \cong \ell^\infty(\{1,\ldots,|H_0|\})$ with the subalgebra of $\rL^\infty(\T)$ that consists of the functions that are constant on the intervals $\{\exp(2 \pi i t) \mid t \in [(j-1)/|H_0|,j/|H_0|) \}$. After all these identifications, one checks that the subalgebras $\al_k(\rL(H_0))$, $k \in \Z$ of $\rL(\Sigma_0)$ are independent and generate $\rL(\Sigma_0)$. This results into a $*$-isomorphism
$$\theta_0 : \rL(\Sigma_0) \recht \rL(H_0^{(\Z)}) \quad\text{satisfying}\quad \theta_0 \circ \al_k = \si_k \circ \theta_0 \;\;\text{for all}\;\; k \in \Z \; .$$
Here, $(\si_k)_{k \in \Z}$ denotes the Bernoulli shift on $\rL(H_0^{(\Z)})$.

We now perform the co-induction construction. Choose representatives $I \subset \Gamma$ for the coset space $\Gamma / \Z$. So, the multiplication map $I \times \Z \recht \Gamma$ is a bijection. We get an action $\Gamma \actson I : (g,i) \mapsto g \cdot i$ and a map $\om : \Gamma \times I \recht \Z$ such that $g i = (g\cdot i) \om(g,i)$ for all $g \in \Gamma$, $i \in I$. The map $\om$ is a $1$-cocycle: $\om(gh,i) = \om(g,h\cdot i) \om(h,i)$ for all $g,h \in \Gamma$ and $i \in I$. Define $\Sigma := \Sigma_0^{(I)}$ and denote by $\pi_i : \Sigma_0 \recht \Sigma$ the embedding in position $i$. Then, $\Gamma$ acts on $\Sigma$ by group automorphisms $(\beta_g)_{g \in \Gamma}$ defined as $\beta_g \circ \pi_i = \pi_{g \cdot i} \circ \al_{\om(g,i)}$.

Put $\Lambda = \Sigma \rtimes \Gamma$. Observe that $\Lambda$ is torsion free whenever $\Gamma$ is torsion free. We claim that $\rL \Lambda \cong \rL(H_0 \wr \Gamma)$.

Identifying $(\rL(H_0^{(\Z)}))^I \cong \rL(H_0^{(\Gamma)})$ through the multiplication map $I \times \Z \recht \Gamma$, the formula $\theta = \bigotimes_{i \in I} \theta_0$ defines a $*$-isomorphism
$$\theta : \rL(\Sigma) \recht \rL(H_0^{(\Gamma)}) \quad\text{satisfying}\quad \theta \circ \beta_g = \si_g \circ \theta \;\;\text{for all}\;\; g \in \Gamma \; .$$
But then, $\theta$ extends to an isomorphism of the corresponding crossed product II$_1$ factors that are isomorphic with $\rL(\Lambda)$ and $\rL(H_0 \wr \Gamma)$ respectively. This proves the claim.

We have already proven that for $\Gamma$ torsion free, there exists a torsion free group $\Lambda$ satisfying $\rL \Lambda \cong \rL(H_0 \wr \Gamma)$.

To conclude the proof of Theorem \ref{thm.counterex} we show that by varying the initial embedding $\Z \hookrightarrow \Gamma := \PSL(n,\Z)$, the above construction provides infinitely many non-isomorphic groups $\Lambda$. Assume that $\Lambda = \Sigma \rtimes \Gamma$ and $\Lambda' = \Sigma' \rtimes \Gamma$ are constructed as above from embeddings $\eta : \Z \recht \Gamma$ and $\eta' : \Z \recht \Gamma$. It suffices to prove the following.

{\bf Claim.} If for every automorphism $\delta \in \Aut(\Gamma)$, the intersection $\delta(\eta(\Z)) \cap \eta'(\Z)$ is reduced to $\{1\}$, then $\Lambda \not\cong \Lambda'$.

Assume that $\lambda : \Lambda \recht \Lambda'$ is an isomorphism of groups. Since $\PSL(n,\Z)$ has no normal abelian subgroups except from $\{1\}$, it follows that $\lambda(\Sigma) = \Sigma'$. Hence, $\lambda$ is of the form $\lambda(x,g) = (\ldots,\delta(g))$ for some automorphism $\delta$ of $\Gamma = \PSL(n,\Z)$. Since $\Sigma$ and $\Sigma'$ are abelian groups, it follows that $\lambda|_{\Sigma} \circ \beta_g = \beta_{\delta(g)} \circ \lambda|_{\Sigma}$ for all $g \in \Gamma$. Denote by $i \in I$ the coset $\eta(\Z)$ of the identity element. Take a non-trivial element $x \in \Sigma_0$. Take a finite subset $\cF \subset \Gamma/\eta'(\Z)$ such that $\lambda(\pi_i(x)) \subset \Sigma_0^\cF$. We prove that $\lambda(\pi_i(\Sigma_0)) \subset \Sigma_0^\cF$. Choose $y \in \Sigma_0$. By construction, we can find $z \in \Sigma_0$ such that both $x$ and $y$ are a multiple of $z$. So, $\lambda(\pi_i(x))$ is a multiple of $\lambda(\pi_i(z))$. Since $\Sigma_0$ is torsion free and $\lambda(\pi_i(x)) \in \Sigma_0^\cF$, it follows that $\lambda(\pi_i(z)) \in \Sigma_0^\cF$. But then, $\lambda(\pi_i(y)) \in \Sigma_0^\cF$ as well.

Since the subgroup $\pi_i(\Sigma_0)$ is globally invariant under $\eta(\Z)$, it follows that $\lambda(\pi_i(\Sigma_0))$ is globally invariant under $\delta(\eta(\Z))$. But $\lambda(\pi_i(\Sigma_0)) \subset \Sigma_0^{\cF}$. Hence, the action of $\delta(\eta(\Z))$ on $\Gamma/\eta'(\Z)$ has at least one finite orbit. Applying the assumption to the automorphism $\Ad g \circ \delta$, we have $g \delta(\eta(\Z)) g^{-1} \cap \eta'(\Z) = \{1\}$ for all $g \in \Gamma$, so that $\delta(\eta(\Z))$ acts freely on $\Gamma/\eta'(\Z)$. We have reached a contradiction.

\begin{remark}\label{rem.ex-isom}
There are essentially two sources of unexpected isomorphisms between II$_1$ factors. The first one is Connes' uniqueness theorem for amenable II$_1$ factors \cite{Co75} implying that all $\rL \Gamma$ for $\Gamma$ amenable icc, are isomorphic. Secondly, Voiculescu's free probability theory leads to striking isomorphisms between von Neumann algebras constructed as free products, see e.g.\ \cite{Vo89} and the later developments in \cite{Dy92a,Dy92b,DR98}. As an illustration we provide the following list of isomorphic group factors $\rL G$.

1. Since infinite tensor products of II$_1$ factors are McDuff, it follows that whenever $G = \bigoplus_{i=1}^\infty \Lambda_i$ is the infinite direct sum of icc groups $\Lambda_i$, then $\rL G \cong \rL(\Gamma \times G)$ for all icc amenable groups $\Gamma$.

2. We have that $\rL(\Gamma_1 * \cdots * \Gamma_n) \cong \rL \F_n$ whenever $\Gamma_1,\ldots,\Gamma_n$ are infinite amenable groups and $n \geq 2$. By \cite[Corollary 5.3]{Dy92b} the statement holds for $n=2$ and next, by induction,
\begin{align*}
\rL(\Gamma_1 * \cdots * \Gamma_n) & \cong \rL(\Gamma_1 * \cdots * \Gamma_{n-2}) * \rL(\Gamma_{n-1} * \Gamma_n)
 \cong \rL(\Gamma_1 * \cdots * \Gamma_{n-2}) * \rL(\F_2)\\
& \cong \rL(\Gamma_1 * \cdots * \Gamma_{n-2} * \Z) * \rL(\Z) \cong \rL(\F_{n-1}) * \rL(\Z) \cong \rL(\F_n) \; .
\end{align*}
In the same vein, by \cite[Theorem 1.5]{DR98} it follows that whenever $G = \Lambda_1 * \Lambda_2 * \cdots$ is the infinite free product of non-trivial groups $\Lambda_i$, then $\rL(G) \cong \rL(\F_\infty * G)$.

3. The subtlety of how $\rL G$ depends on $G$ is nicely illustrated by the following remark due to Ozawa \cite{Oz04}. Fix a non-amenable group $\Gamma$ and an infinite group $\Lambda$. Consider $G_n := \F_\infty * (\Gamma \times \Lambda)^{*n}$.
\begin{itemlist}
\item If $\Lambda$ is abelian and $\rL\Gamma \cong \M_2(\C) \ovt \rL(\Gamma)$, then all $\rL(G_n)$ are isomorphic, though non-isomorphic with $\rL \F_\infty$.
\item If $\Gamma,\Lambda$ are icc (and still $\Gamma$ non-amenable), then all $\rL(G_n)$ are non-isomorphic.
\end{itemlist}
The reason for this is the following. Fix arbitrary von Neumann algebras $P,Q$ equipped with faithful normal tracial states. Consider the II$_1$ factors $N_n := \rL \F_\infty * (P \ovt Q)^{*n}$. If $P \cong \M_2(\C) \ot P$ and if $Q$ is diffuse abelian, then all $N_n$ are isomorphic. Indeed, applying \cite[Theorem 3.5(iii)]{Dy92a} to $A = \rL \F_\infty$, $B = P \ovt Q$ and using the fact that $2$ belongs to the fundamental group of $\rL \F_\infty$, it follows that $2$ belongs to the fundamental group of $N_1$. Applying \cite[Theorem 3.5(ii)]{Dy92a} to the same algebras $A,B$ and using the obvious isomorphism $Q \cong Q \ovt \rL(\Z/2\Z)$, it follows that $N_1 \cong \M_2(\C) \ot N_2$. Since $1/2$ belongs to the fundamental group of $N_1$, we conclude that $N_1 \cong N_2$. But then, $N_1 \cong N_n$ for all $n$. On the other hand, if $P$ is a non-amenable factor and $Q$ is a diffuse factor, then the II$_1$ factors $M_n$ are non-isomorphic. When $P$ and $Q$ are semi-exact, this follows from \cite[Corollary 3.5]{Oz04}. In the general case, the methods of \cite{IPP05} can be used, see \cite[Theorem 1.4]{Pe06} and \cite[Theorem 1.1]{CH08}.

4. As observed in \cite[Proposition 6.4]{Io06}, if $\rL(H_1)$ and $\rL(H_2)$ are stably isomorphic, then $\rL(H_1 \wr \Z) \cong \rL(H_2 \wr \Z)$. In particular, all group von Neumann algebras $\rL(\F_n \wr \Z)$, $n \geq 2$, are isomorphic. This is in sharp contrast with our Theorem \ref{thm.specialmain} saying that the group $(\Z/2\Z)^{(I)} \rtimes (\F_n \wr \Z)$ is superrigid, where $I = (\F_n \wr \Z) / \Z$.

5. In \cite[Corollary 1.2]{Bo09a} it is shown that the Bernoulli actions $\F_2 \actson (X_0,\mu_0)^{\F_2}$ are orbit equivalent for different choices of the base probability space $(X_0,\mu_0)$. It follows that all $\rL(H \wr \F_2)$, $H$ a non-trivial abelian group, are isomorphic. In \cite[Theorem 1.1]{Bo09b} it is shown that for different values of $n$, the Bernoulli actions $\F_n \actson (X_0,\mu_0)^{\F_n}$ are stably orbit equivalent. Hence, for all choices of $n,m$ and all non-trivial abelian groups $H_1$, $H_2$, the II$_1$ factors $\rL(H_1 \wr \F_n)$ and $\rL(H_2 \wr \F_m)$ are stably isomorphic. In particular $\rL((H_1 \wr \F_n) \times \Lambda_1) \cong \rL((H_2 \wr \F_m) \times \Lambda_2)$ when $\Lambda_1,\Lambda_2$ are icc amenable.
\end{remark}

\section{W$^*$-superrigidity for Bernoulli actions of product groups}\label{sec.Wstarsuperrigidity}

\begin{theorem}\label{thm.Wstarsuperrigid}
Let $\Gamma$ be an icc group which admits a chain of infinite subgroups $\Gamma_0<\Gamma_1< \cdots <\Gamma_n=\Gamma$ such that $\Gamma_{k-1}$ is almost normal in $\Gamma_k$, for every $k=1,\ldots,n$ and the centralizer of $\Gamma_0$ inside $\Gamma_1$ is non-amenable. Let $(X_0,\mu_0)$ be a non-trivial standard probability space.
Then the Bernoulli action $\Gamma \actson (X,\mu):= (X_0,\mu_0)^\Gamma$ is W$^*$-superrigid: if $\Lambda \actson (Y,\eta)$ is an arbitrary free ergodic p.m.p.\ action and $\pi : \rL^\infty(Y) \rtimes \Lambda \recht \rL^\infty(X) \rtimes \Gamma$ a $*$-isomorphism, then $\Lambda \cong \Gamma$ and the actions are conjugate.

More precisely, there exist an isomorphism of groups $\delta : \Lambda \recht \Gamma$, an isomorphism of probability spaces $\Psi : Y \recht X$, a character $\omega : \Gamma \recht \T$ and a unitary $w \in \rL^\infty(X) \rtimes \Gamma$ such that $\Psi(s \cdot y) = \delta(s) \cdot \Psi(y)$ for all $s \in \Lambda$ and a.e.\ $y \in Y$ and such that
$$\pi = (\Ad w) \circ \pi_\om \circ \pi_{\Psi,\delta}$$
where $\pi_{\Psi,\delta}(b v_s) = (b \circ \Psi^{-1}) u_{\delta(s)}$ for all $b \in \rL^\infty(Y)$, $s \in \Lambda$ and $\pi_\om(au_g) = \om(g) a u_g$ for all $a \in \rL^\infty(X)$, $g \in \Gamma$.
\end{theorem}

Denote $A = \rL^\infty(X)$ and $M=A\rtimes \Gamma$. Put $B = \rL^\infty(Y)$ and identify $M = B \rtimes \Lambda$ through $\pi$.
Let $\Delta:M\rightarrow M\overline{\otimes}M$ be the unital $*$-homomorphism defined as $\Delta(bv_s)=bv_s\otimes v_s$, for all $b\in B$ and $s\in\Lambda$.
Before continuing, let us record a few useful properties of $\Delta$.

\begin{lemma}\label{lemma.about-coaction}
Let $P\subset M$ be a von Neumann subalgebra.
\begin{itemlist}
\item If $P\not\embed B$, then $\Delta(P)\nprec M\otimes 1$.

\item If $P$ is diffuse, then $\Delta(P)\nprec 1\otimes M$.

\item If $\Delta(M)\prec M\overline{\otimes}P$, then $\rL \Lambda \prec M$.

\item If $\Delta(M)\prec P\overline{\otimes}M$, there exists a non-zero projection $p\in P'\cap M$ such that $Pp\subset pMp$ has finite index.

\item If $P$ has no amenable direct summand, then $\Delta(P)$ is strongly non-amenable relative to $M\otimes 1$ and $1\otimes M$. In particular, if $N\subset M$ is an amenable von Neumann subalgebra, then $\Delta(P)\nprec M\overline{\otimes}N$ and $\Delta(P)\nprec N\overline{\otimes}M$.
\end{itemlist}
\end{lemma}
\begin{proof}
To prove 1-4, see \cite[Lemma 9.2]{Io10} or adapt  the proof of Proposition \ref{prop.about-comult}. Since $B$ is amenable, the $M$-$(M\overline{\otimes}M)$-bimodule $\bim{(\Delta(M)\otimes 1)}{\rL^2(M\overline{\otimes}M\overline{\otimes}M)}{(M\overline{\otimes}1\overline{\otimes}M)}$ is weakly contained in the coarse $M$-$(M\overline{\otimes}M)$-bimodule $\rL^2(M)\overline{\otimes}\rL^2(M\overline{\otimes}M)$. Continuing exactly as in the proof of Proposition \ref{prop.about-comult}.4 yields 5.
\end{proof}

To prove Theorem \ref{thm.Wstarsuperrigid} it is sufficient to show
that  $B\prec A$ or $\rL\Lambda\prec \rL\Gamma$. First, if $\rL
\Lambda \prec \rL \Gamma$, then \cite[Case (5) in the proof of
Theorem 9.1]{Io10} shows that automatically $B\prec A$. If $B\prec
A$ then, by \cite[Theorem A.1]{Po01b}, $B$ and $A$ are unitarily
conjugate, so that after such a unitary conjugacy, $\pi$ is
implemented by an orbit equivalence between $\Lambda \actson Y$ and
$\Gamma \actson X$, together with an $\T$-valued cocycle for the
action $\Gamma \actson X$. By the cocycle superrigidity theorem
\cite[Theorem 1.1]{Po06a}, we can assume that the orbit equivalence
is a conjugacy and that the $\T$-valued cocycle is a character.

We prove Theorem \ref{thm.Wstarsuperrigid} by contradiction assuming that $B\nprec A$ and $\rL\Lambda\nprec \rL\Gamma$. The proof consists of several steps.

\begin{step}\label{step1adrian}
There exists a unitary $v\in M\overline{\otimes}M$ such that $v\Delta(\rL\Gamma)v^*\subset \rL(\Gamma\times\Gamma)$.
\end{step}
\begin{proof}
Let $Q=\Delta(\rL\Gamma_0)$ and denote by $P$ the quasi-normalizer of $Q$ inside $M\overline{\otimes}M$. Since $\Delta(\rL\Gamma_1)\subset P$ and the centralizer of $\Gamma_0$ inside $\Gamma_1$ is non-amenable, Lemma \ref{lemma.about-coaction}.5 implies that $Q'\cap P$ is strongly non-amenable relative to $M\otimes 1$ and $1\otimes M$. Since $\Gamma_1$ is non-amenable and $\Delta(\rL\Gamma_1)\subset P$, Lemma \ref{lemma.about-coaction}.5 gives that $P\nprec M\overline{\otimes}A$ and $P\nprec A\overline{\otimes}M$.

We claim that $Q\nprec M\otimes 1$ and $Q\nprec 1\otimes M$. Indeed, by Lemma \ref{lemma.about-coaction} it suffices to prove that $\rL\Gamma_0\nprec B$. If  we assume that $\rL\Gamma_0\prec B$, then [Va07, Lemma 3.5.] implies that $B\prec (\rL\Gamma_0)'\cap M$. Since $\Gamma_0$ is infinite,  we get that $(\rL\Gamma_0)'\cap M\subset \rL\Gamma$ and therefore $B\prec \rL\Gamma$, contradicting the fact that $B$ is regular in $M$ (see \cite[Theorem 3.1]{Po03}).

Applying Corollary \ref{cor.malleable-rigid}.3 we get a unitary $v\in M\overline{\otimes}M$ such that $vPv^*\subset M\overline{\otimes}\rL\Gamma$. Repeating the last part of step \ref{step1} in the proof of Theorem \ref{thm.main-gen} yields the conclusion.
\end{proof}

From now we replace $\Delta$ by $(\Ad v) \circ \Delta$ and assume that $\Delta(\rL\Gamma)\subset \rL(\Gamma\times\Gamma)$. Let $C=\Delta(A)'\cap (M\overline{\otimes}M)$.

\begin{step}\label{step2adrian}
For every projection $p\in\Cal Z(C)$ we have $Cp\prec A\overline{\otimes}A$. Moreover there exists a unitary $u\in M\overline{\otimes}M$ such that $u\cZ(C)u^*\subset A\overline{\otimes}A$.
\end{step}
\begin{proof}
Since $\Gamma$ is non-amenable, by Lemma \ref{lemma.about-coaction} we have that $\Delta(\rL\Gamma)\nprec \rL\Gamma\otimes 1$ and $\Delta(\rL\Gamma)\nprec 1\otimes \rL\Gamma$. We claim that $\Delta(A)\nprec \rL\Gamma\overline{\otimes}M$ and $\Delta(A)\nprec M\overline{\otimes}\rL\Gamma$. If we assume the contrary, since $\Delta(M)$ is contained in the quasi-normalizer of $\Delta(A)$ inside $M\overline{\otimes}M$, \cite[Proposition 3.5]{Io10} implies that one of the following holds: $\Delta(A)\prec 1\otimes M$, $\Delta(M)\prec \rL\Gamma\overline{\otimes} M$, $\Delta(A)\prec M\otimes 1$ or $\Delta(M)\prec M\overline{\otimes}\rL\Gamma$.
Applying Lemma \ref{lemma.about-coaction} we get that either $A$ is not diffuse, $\rL(\Gamma)$ has finite index in $M$, $A\prec B$ or $\rL\Lambda\prec \rL\Gamma$, all of which give a contradiction.

Since $\{\Delta(u_g)\}_{g\in \Gamma}$ normalize $\Delta(A)$, the previous paragraph allows us to apply \cite[Theorem 6.1]{Io10} and the conclusion follows.
\end{proof}

Note that the unitaries $\{\Delta(u_g)\}_{g\in\Gamma}$ normalize $C$ and denote by $(\beta_g)_{g\in\Gamma}$ the action of $\Gamma$ on $C$ given by  $\beta_g(x)=\Delta(u_g)x\Delta(u_g)^*$, for $g\in\Gamma$ and $x\in C$.
Step \ref{step2adrian} implies that the algebra $\cZ_0:=\cZ(C)\cap \rL(\Gamma\times\Gamma)$ is completely atomic. Let $p\in\cZ_0$ be a minimal projection and let $G\subset\Gamma$ be a finite index subgroup such that $p$ is $(\beta_g)_{g\in G}$-invariant.

We claim that the action $(\beta_g)_{g\in G}$ on $\cZ(C)p$ is weakly mixing. To prove this claim, let $\cH\subset\cZ(C)p$ be a finite dimensional $(\beta_g)_{g\in G}$-invariant subspace. Then $\cH$ is contained in the quasi-normalizer of $\Delta(\rL G)p$ inside $p(M\overline{\otimes}M)p$. Since $\Delta(\rL G)\nprec \rL\Gamma\otimes 1$ and $\Delta(\rL G)\nprec 1\otimes \rL\Gamma$,  we get from \cite[Lemma 4.2]{Va07} that $\cH\subset p \rL(\Gamma\times\Gamma)p$. Thus $\cH\subset \cZ_0p=\C p$, proving the claim.

For $d\geq 1$, we denote by $\cG_d$ the group $\{u\otimes u_g|u\in\cU(\M_d(\C)), g\in\Gamma\times\Gamma\}$.

\begin{step}\label{step3adrian}
There exist $d\geq 1$, two groups $K\subset\cG\subset\cG_d$ with $K$ finite and normal in $\cG$, a group homomorphism $\delta:G\rightarrow\cG/K$, a partial isometry $w\in \M_{1,d}(\C)\otimes \rL(\Gamma\times\Gamma)$ with $ww^*=p$ and $w^*w=p_K:=|K|^{-1}\sum_{k\in K} k$,  such that
$w^*Cw=(\M_d(\C)\otimes (A\overline{\otimes}A))^{\Ad K} p_K$ and $w^*\Delta(u_g)w=\delta(g)p_K$ for all $g\in G$.
\end{step}
\begin{proof}
We apply Theorem \ref{thm.conjugacy-actions} and Corollary \ref{cor.conjugacy-actions}.2 to the crossed product $M\overline{\otimes}M=(A\overline{\otimes}A)\rtimes (\Gamma\times\Gamma)$. Since the action $(\Ad(\Delta(u_g)p))_{g\in G}$ on $\cZ(Cp)$ is weakly mixing, $Cp\prec A\overline{\otimes}A$ by step \ref{step2adrian}, and $(Cp)'\cap p(M\overline{\otimes}M)p=\cZ(Cp)$, all conditions of Theorem \ref{thm.conjugacy-actions} are indeed satisfied.

Moreover, also the extra condition 2 in Corollary \ref{cor.conjugacy-actions} holds. Indeed, if a subgroup $H$ of $\Gamma\times\Gamma$ acts non-ergodically on $A\overline{\otimes}A$, then $H\subset H_0\times\Gamma$ or $H\subset \Gamma\times H_0$ for some finite subgroup $H_0$ of $\Gamma$.
Since $G$ is non-amenable, from Lemma \ref{lemma.about-coaction} we know that $\Delta(\rL G)\nprec (A\overline{\otimes}A)\rtimes H$, for every such subgroup $H$ of $\Gamma \times \Gamma$. Thus, we also have that $N\nprec (A\overline{\otimes}A)\rtimes H$, where $N$ denotes the von Neumann algebra generated by $Cp$ and $\Delta(\rL G)p$.

Theorem \ref{thm.conjugacy-actions} and Corollary \ref{cor.conjugacy-actions}.2 provide the conclusion of step \ref{step3adrian}.
\end{proof}

\begin{step}
End of the proof of Theorem \ref{thm.Wstarsuperrigid}.
\end{step}
\begin{proof}
Denote by $\gamma : \cU(\M_d(\C)) \times \Gamma \times \Gamma \recht \Gamma$ the group morphism $\gamma(u,g,h) = h$. Put $\cG_0 := \gamma(\cG)$ and $K_0 := \gamma(K)$. By construction, $K_0$ is a finite normal subgroup of $\cG_0$ and we still denote by $\gamma$ the natural group homomorphism $\gamma : \cG/K \recht \cG_0 / K_0$. Denote by $G_1 < G$ the kernel of the homomorphism $\gamma \circ \delta$. By construction, $w^* \de(\rL G_1) w \embed M \ot 1$ and hence, $\de(\rL G_1) \embed M \ot 1$. By Lemma \ref{lemma.about-coaction} we have $\rL G_1 \embed B$ and the proof of step \ref{step1adrian} implies that $G_1$ cannot be infinite.

We consider the Fourier decomposition of elements in $\M_d(\C) \ovt M \ovt M$ with respect to the crossed product
$\M_d(\C) \ovt M \ovt M = (\M_d(\C) \ovt M \ovt A) \rtimes \Gamma$, where $\Gamma$ only acts on $A$. Note that the Fourier coefficients of a bounded sequence $x_n \in \M_d(\C) \ovt M \ovt M$ tend to zero pointwise in $\| \, \cdot \, \|_2$ if and only if
$$\|E_{\M_d(\C) \ovt M \ovt A}(a x_n b)\|_2 \recht 0 \quad\text{for all}\quad a,b \in \M_d(\C) \ovt M \ovt M \; .$$
It follows that for all $a,b \in \M_d(\C) \ovt M \ovt M$, also the Fourier coefficients of $a x_n b$ tend to zero pointwise. We also consider the Fourier decomposition of elements in $M$ with respect to the crossed product $M = A \rtimes \Gamma$. In both situations, we denote the Fourier coefficients of an element $x$ as $(x)_g$, $g \in \Gamma$. When $x \in \M_d(\C) \ovt M \ovt M$, then $(x)_g \in \M_d(\C) \ovt M \ovt A$.

Define the normal $*$-homomorphism $\theta : A \recht (\M_d(\C) \ovt A \ovt A)^{\Ad K}$ such that $w^* \de(a) w = \theta(a) p_K$ for all $a \in A$. By step \ref{step3adrian} we get that for all $x \in A \rtimes G$ and all $h \in \Gamma$, we have
$$\sum_{k \in K_0} (w^* \de(x) w)_{hk} (1 \ot 1 \ot u_{hk}) = \sum_{g \in G, \gamma(\delta(g)) = h K_0} \theta((x)_g) \delta(g) p_K \; .$$
Recall that for a fixed $h \in \Gamma$, there are only finitely many $g \in G$ satisfying $\gamma(\delta(g)) = h K_0$. So, if $x_n$ is a bounded sequence in $A \rtimes G$ whose Fourier coefficients tend to zero pointwise, the same is true for $w^* \de(x_n) w$. By the remarks in the previous paragraph, the Fourier coefficients of $\de(x_n) p$ then also tend to zero pointwise.

Next, let $x_n$ be a bounded sequence in $A \rtimes \Gamma$ whose Fourier coefficients tend to zero pointwise. Let $g_1,\ldots,g_s \in \Gamma$ be representatives for $\Gamma/G$. Define for $j =1,\ldots,s$,
$$x^j_n := E_{A \rtimes G}(u_{g_j}^* x_n) \; .$$
Then for every $j$, we have that $(x^j_n)_n$ is a bounded sequence in $A \rtimes G$ whose Fourier coefficients tend to zero pointwise. The previous paragraph, together with the formula
$$\de(x_n) p = \sum_{j=1}^s \de(u_{g_j}) \de(x^j_n) p \; ,$$
imply that the Fourier coefficients of $\de(x_n) p$ tend to zero pointwise.

Since $B \not\embed A$, Definition \ref{def.intertwine} provides a sequence of unitaries $b_n$ whose Fourier coefficients tend to zero pointwise. By the previous paragraph the same is true for $\de(b_n) p$. But for $b \in B$, we have $\de(b) = v(b \ot 1)v^*$ (recall that the unitary $v \in M \ovt M$ was given by step \ref{step1adrian} and that we conjugated the initial comultiplication $\de$ by $v$).
So, $\de(b_n)p = v(b_n \ot 1)vp$ and it follows that the Fourier coefficients of $(b_n \ot 1)vp$ tend to zero pointwise. Taking the $g$-th Fourier coefficient we get that
$$\|(vp)_g\|_2 = \|b_n (vp)_g\|_2 = \|((b_n \ot 1)v p)_g\|_2 \recht 0$$
for all $g \in \Gamma$. We reached the contradiction that $p$ must be $0$.
\end{proof}
This ends the proof of Theorem \ref{thm.Wstarsuperrigid}.\hfill\qedsymbol

\end{document}